\documentclass[12pt]{article}
\usepackage[top=1.5in, bottom=1.5in, left=1in, right=1in]{geometry}

\usepackage{amsmath,amssymb,amsthm,mathtools,amsfonts}
\usepackage{bm,bbm,soul}
\usepackage{longtable,array}
\usepackage{enumitem}
\usepackage{graphicx}
\usepackage{mathrsfs,xcolor,soul}
\usepackage[round,compress,sort]{natbib}
\usepackage{algorithm,algpseudocode}
\usepackage{booktabs}
\usepackage{placeins}
\usepackage{float}
\usepackage[hidelinks,hypertexnames=false]{hyperref}
\newcommand{\R}{\mathbb{R}}
\newcommand{\mc}{\mathcal}

\newcommand{\mr}{\mathrm}

\newcommand{\eps}{\varepsilon}
\newcommand{\X}{\mc{X}}

\newcommand{\D}{\mc{D}}
\newcommand{\C}{\mc{C}}

\newcommand{\id}{\mathrm{id}}

\theoremstyle{plain}
\newtheorem{theorem}{Theorem}[section]
\newtheorem{lemma}{Lemma}[section]
\newtheorem{proposition}{Proposition}[section]
\newtheorem{corollary}{Corollary}[section]

\theoremstyle{definition}
\newtheorem{definition}{Definition}[section]
\newtheorem{assumption}{Assumption}[section]
\newtheorem{example}{Example}[section]
\theoremstyle{remark}

\newcommand{\blind}{1}

\begin{document}

\def\spacingset#1{\renewcommand{\baselinestretch}%
{#1}\small\normalsize} \spacingset{1}

\vspace{-2em}
\if1\blind
{ 
  \title{\bf Canonical Layer Inference for Deep Gaussian Processes via Symmetry-Breaking Augmentation}
  \author{Yichen Zhu\thanks{Email: yczhu@hku.hk}\\
    Department of Statistics and Actuarial Science, the University of Hong Kong\\ 
    \hspace{.2cm}\\
    Hengrui Luo\thanks{Email: hrluo@rice.edu}\\
  Department of Statistics, Rice University \\
  and \\
  Computational Research Division, Lawrence Berkeley National Laboratory}
  \date{}
  \maketitle
}  \fi

\if0\blind
{
  \bigskip
  \bigskip
  \bigskip
  \begin{center}
    {\LARGE\bf Canonical Layer Inference for Deep Gaussian Processes via Symmetry-Breaking Augmentation}
\end{center}
  \medskip
} \fi

\bigskip
\vspace{-2em}
\begin{abstract}
Deep Gaussian processes (DGPs) are flexible priors for estimating compositional functions, but  function-level accuracy does not determine their intermediate-layer representations. The basic obstruction is an adjacent-layer symmetry: applying an invertible transformation and its inverse to two adjacent layers leaves the composed regression map unchanged. Under explicit symmetry and prior-comparability conditions, we show that the posterior may not concentrate on any prescribed intermediate-layer representation, even when the composed regression function is estimable. We introduce a layerwise exponential tilt of the DGP prior, using a penalty proportional to each hidden layer's distance from the identity map, with stronger weights on earlier layers. The methodology links an exact posterior-odds rule along likelihood-preserving symmetries to quantitative control of redundant layers. Under proxy and local prior-mass conditions, we derive a penalty window that preserves function-level fractional posterior contraction while localizing redundant layers at certain radius. The resulting design rule calibrates representation regularization to the function-estimation rate. It favors decompositions with earlier layers close to identity; uniqueness and automatic depth selection are not established. The general theory and methodology are supported with detailed examples.
\end{abstract}

\noindent%
{\it Keywords: augmentation, Gaussian processes, identifiability, layer structures, posterior contraction}   
\vfill

\newpage
\spacingset{1} 

\section{Introduction}\label{sec:intro}
Deep Gaussian processes (DGPs) model a regression function as a composition of Gaussian process (GP) layers.  They are attractive both empirically, for capturing nonstationary and compositional structures \citep{damianou2013deep,salimbeni2017doubly,luo2022sparse,zhu2024radial}, and theoretically, because suitably constructed DGP priors can attain minimax contraction rates (up to a poly-logarithm factor in sample size) for nonparametric regression of compositional smooth functions \citep{finocchio2023posterior,castillo2025deep}.  Existing guarantees, however, are mostly \emph{function-level}: they describe the learning of the composed regression map but not the learning of the individual layers. We develop a method for regularizing the intermediate representation with guarantees for both function estimation and shrinkage of redundant layers. 

DGPs are often used as hierarchical representations in which one hopes to interpret intermediate layers, for example, to determine which layers provide nontrivial transformations, which are close to the identity, and which admit a scientific interpretation. Such questions are not addressed by the function-level contraction theory alone.
To see the distinction, define  a tuple of layers  $\theta=(f_1,\ldots,f_J)$ and the composition map by
\begin{equation}\label{eq:comp}
    \C(\theta)=f_J\circ\cdots\circ f_1.
\end{equation}
The regression likelihood depends on $\theta$ only through $\C(\theta)$.  Consequently, many layer tuples can induce the same regression function.  For example, if $S$ is an invertible map, then
\begin{align*}
    f_{j+1}\circ f_j
    =
    (f_{j+1}\circ S^{-1})\circ (S\circ f_j), \quad \forall 1\le j\le J-1.
\end{align*}
Thus an invertible transformation $S$ and its inverse $S^{-1}$ can be applied to two adjacent layers without changing the composed regression function.  This is the adjacent-layer symmetry example that motivates and drives our methodology. This phenomena generalizes the identifiability issue in latent factor models \citep{papastamoulis2022identifiability,gu2020partial} to functional spaces, and cannot be simply attributed to multi-modal posterior. 

The statistical consequence is that posterior contraction of $\C(\theta)$ does not imply posterior contraction of $\theta$.  A standard product DGP prior may assign the same or comparable, probability to many symmetry-related decompositions \citep{schmidt2003bayesian,dunlop2018howdeep}.  In such cases, the likelihood does not select a unique layer representation, and the posterior may fail to concentrate on any fixed representative when the prior remains sufficiently symmetric.  The posterior contraction target is not a single decomposition, but an equivalence class, or symmetry-related class, of decompositions that represent the same composed map.

Our methodological contribution is an identity-deviation penalty whose strength varies across depth. Stronger penalties encourage earlier layers toward identity, while weaker penalties allow the later layers to approximate the regression function. The analysis connects this design to three quantitative results:
\begin{itemize}[leftmargin=*]
    \item \textbf{A quantified representation obstruction.} Theorem~\ref{thm:non-conv} bounds a layer neighborhood's posterior mass by the reciprocal of a symmetry packing number, under uniform prior comparability. This identifies when function estimation alone cannot select a layer representation.
    \item \textbf{An explicit regularization mechanism.} Theorem~\ref{thm:noniden break} gives the exact change in posterior odds under the layerwise tilt. A uniformly diverging penalty gap separates symmetry-related sets with comparable prior weights.
    \item 
    \textbf{Joint estimation and representation guarantees.} Theorems~\ref{thm:contraction}--\ref{thm:parsimonious} derive a common penalty window for function-level fractional contraction and redundant-layer localization at radius $n\epsilon_n^2/\eta_{j,n}$. Approximation of the regression function and prior mass for the later-layer composition make the window verifiable in concrete examples.
    
\end{itemize}

\begin{figure}[!h]
\centering
\makebox[\textwidth][c]{%
\includegraphics[width=1.1\linewidth]{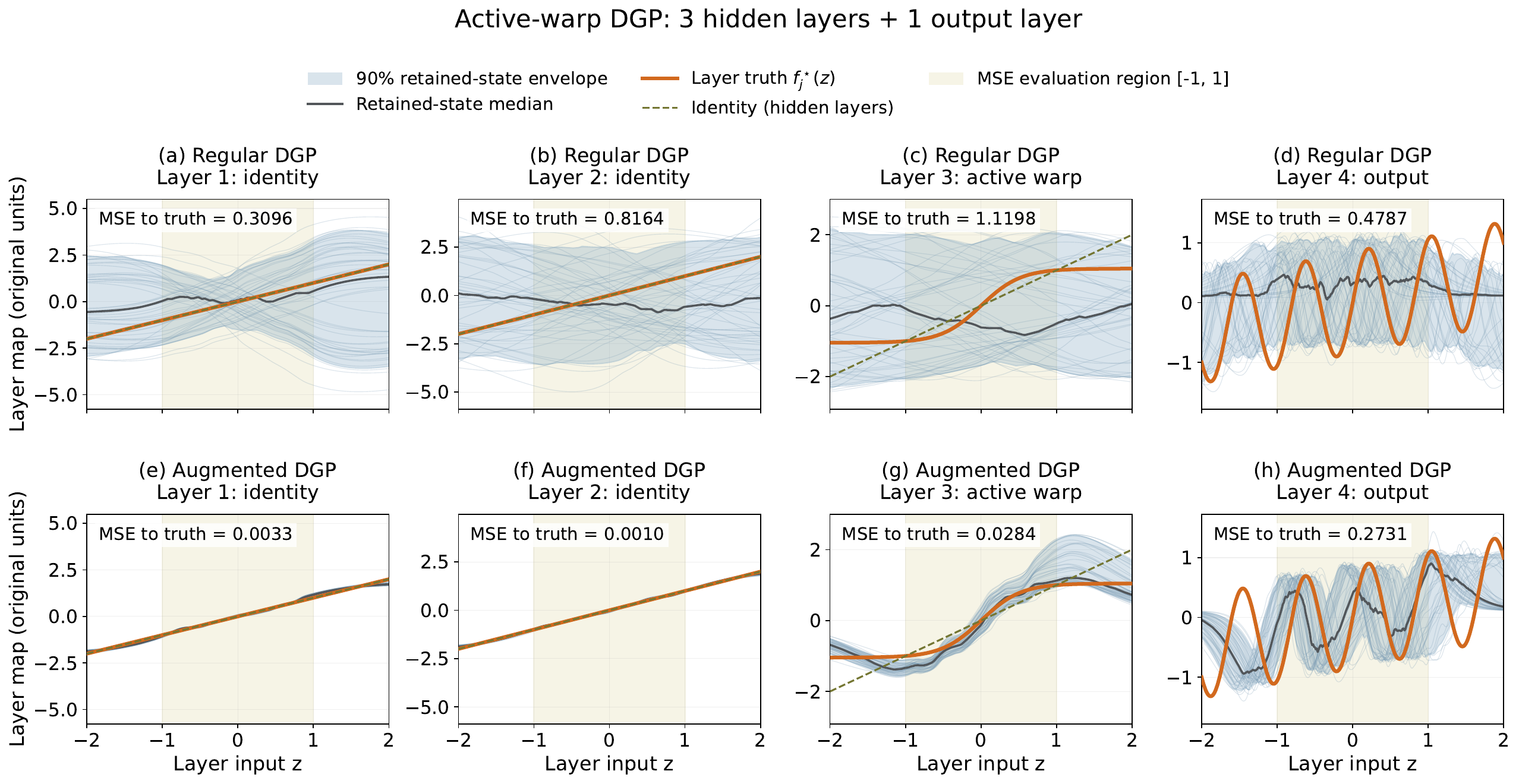}
}
\caption{Posterior anchoring in the four-layer
\texttt{active\_penultimate\_warp} design, with two redundant identity layers and two nonlinear components. The rows show the regular DGP and augmented DGP (proposed method in Section \ref{sec:asym-penalty}), 
both using identity initialization for the hidden layers. 
Annotated MSEs compare the median maps with their truths on the lightly shaded region $[-1,1]$, while penalties are evaluated on $\tilde\X=[-2,2]$. 
}
\label{fig:edge-transition-anchor}
\end{figure}

\paragraph{Related Literature.}
The closest related works fall into three strands. First, DGP theory has mainly studied the composed regression function: posterior contraction and adaptation for compositional smooth functions are developed   \citep{finocchio2023posterior,castillo2025deep}. These results do not imply recovery of the individual layers. Lower bounds for shallow GP priors establish advantages of depth for compositional classes \citep{giordano2022inability}. Compositional non-identifiability and uncertainty over DGP layers and layer-factorized variational approximations are explicitly studied by \citet{ustyuzhaninov2020compositional}; here we quantify non-concentration and study an explicit representation preference. Second, warping and deformation models use latent transformations to represent nonstationarity or non-Gaussian structure \citep{snelson2003warped,sampson1992nonparametric,schmidt2003bayesian,zammitmangion2022deep}; compositions of invertible maps also underlie normalizing flows \citep{rezende2015variational}. These literatures motivate expressive intermediate maps; our focus is the compatibility of identity-based regularization with function-level contraction.

A separate line of work controls unnecessary depth. \citet{duvenaud2014avoiding} identify pathologies from repeated GP composition, while \citet{dunlop2018howdeep} study how depth changes DGP behavior. Residual networks instead parameterize a layer as identity plus a perturbation \citep{he2016deep}. Identity mean functions are also used in intermediate GP layers by \citet{salimbeni2017doubly}.  The methodological distinction {\color{black} of our paper} is an exponential reweighting of the prior on layer functions, with explicit bounds on the penalty strengths. It penalizes realized departures from identity and permits different strengths across depth. Its statistical contribution is the connection between symmetry-induced posterior odds, preserved prior mass near useful compositions, and a shrinking radius for redundant layers. The same construction therefore supplies an explicit representation preference and conditions under which that preference retains the function-estimation rate.
\paragraph{Organization.}
Section~\ref{sec:setup} introduces the model, metrics, and terminology.  Section~\ref{sec:canonical composition}  gives symmetry and prior-comparability conditions preventing layerwise concentration.  Section~\ref{sec:asym-penalty} introduces the identity-biased augmented posterior and proves changes in posterior odds between symmetry-related representations. Section~\ref{sec:theory} proves contraction preservation and layerwise parsimony. 
Section~\ref{sec:experiments} gives concise numerical diagnostics; the supplementary materials contains auxiliary lemmas, proofs, algorithms, and experiment-design details.

\section{Model Definition and Preliminaries}\label{sec:setup}

\subsection{Deep Gaussian Processes}\label{sec:dgp}
A deep Gaussian process is the composition of multiple Gaussian processes, such that the output of one GP is the input of the next GP.
Specifically, for
$J\in\mathbb{N}$ denoting the number of layers and $\X\subset\mathbb{R}^d$ a compact set, the deep Gaussian process is a stochastic process on $\X$ defined in the following way:
\begin{align}\label{eq:dgp}
    f & = f_J \circ f_{J-1} \circ \cdots \circ f_1, \\
    f_j & = (f_{j,1}, f_{j,2}, \cdots, f_{j,d})^T, \;\forall 1\le j \le J-1,\nonumber\\
    f_{j,k} & = Z_{j,k}, \; (Z_{j,k}(x): x\in \mathbb{R}^d) \sim \mr{GP}_j, \; \forall 1\le j \le J-1, 1\le k\le d, \nonumber\\
    f_J & = Z_J, \; (Z_J(x): x\in \mathbb{R}^d) \sim \mr{GP}_J.\nonumber
\end{align}
 All layerwise Gaussian processes in \eqref{eq:dgp} are defined on $\mathbb{R}^d$; the regression likelihood evaluates their composition only on the input domain $\X$. We use independent centered coordinate processes with continuous versions and fixed depth $J$.

In this paper, we are mainly interested in the regression  problem of estimating an unknown function  with deep Gaussian process priors. That is, let  $\D_n = \{(X_i, Y_i)\}_{i=1}^n$ be the training data, where $X_1, X_2, \cdots, X_n \in \X$ are i.i.d. draws from some probability measure $\mu_X$, and let $Y_1, Y_2, \cdots, Y_n \in\mathbb{R}$ come from the sampling model
\begin{equation}\label{eq:reg}
    Y_{i}=f^*(X_{i})+\varepsilon_{i},\qquad\varepsilon_{i}\stackrel{\text{i.i.d.}}{\sim}N(0,\sigma^{2}),
\end{equation}
with $f^*:\X \to \mathbb{R}$ being the data-generating function.
By putting a DGP prior \eqref{eq:dgp} on the regression function $f^*$ in \eqref{eq:reg}, we obtain a valid Bayesian model.

The objective of deep Gaussian process regression has two parts. First, we would like to estimate the data-generating function $f^*$; Second, if the data-generating function $f^*$ admits a ``true'' decomposition into layer-dependent functions: 
\begin{equation}\label{eq:f* decom}
    f^* = f^{*}_{J} \circ f^{*}_{J-1} \circ \cdots \circ f^{*}_1,
\end{equation}
then we would like to estimate the layerwise maps $f^{*}_1, f^{*}_2, \cdots, f^{*}_J$ as latent infinite-dimensional parameters. The first objective is always well-defined, whereas the second objective is a bit tricky. We can always decompose $f^*$ into the composition form as \eqref{eq:f* decom} by adding enough identity mappings, so it is easy to see the decomposition is not unique without additional assumptions. 
Nevertheless, this does not make the second objective invalid. 
While the discovery of such a scientifically meaningful  (de)composition by the regular likelihood alone is not possible, with a nonparametric Bayesian prior like DGP, the question of identifiability, and subsequently, contraction, to such (de)composition becomes unclear.

\subsection{Function Spaces and Norms}\label{sec:tech prelim}

The likelihood is evaluated on the compact input set $\X\subset\mathbb{R}^d$, equipped with a measure $\mu_X$ that generates the observed $X_i$, while every layerwise process in \eqref{eq:dgp} is defined on $\mathbb{R}^d$. For layerwise norms we use an enlarged compact domain $\tilde{\X}\subset\mathbb{R}^d$ whose interior contains $\X$. The domain $\tilde{\X}$ is equipped with a probability measure $\tilde{\mu}$ with $\operatorname{supp}(\tilde\mu)=\tilde\X$. We assume $\mu_X\ll\tilde{\mu}$ and that $d\mu_X/d\tilde{\mu}$ is essentially bounded. For all $1\le p<\infty$, define
\begin{align*}
\|f_j\|_{\ell^p(\tilde\mu)} & \triangleq \big[\int_{\tilde{\X}} \|f_j(x)\|_p^p d\tilde{\mu}(x)\big]^{1/p}, \;\; 1\le j\le J-1, \\
\|f_J\|_{\ell^p(\tilde\mu)} & \triangleq \big[\int_{\tilde{\X}}|f_J(x)|^p\,d\tilde\mu(x)\big]^{1/p},
\end{align*}
where in the case of $p=\infty$, the integrals become the sup norms on the support of $\tilde{\mu}$. 
With a slight abuse of notation, for $\theta=(f_1,\ldots,f_J)$ define
\begin{equation}\label{eq:theta norm}
    \|\theta\|_{\ell^p}
    =
    \max\left\{\max_{1\le j<J}\|f_j\|_{\ell^p(\tilde\mu)},\,\|f_J\|_{\ell^p(\tilde\mu)}\right\}.
\end{equation}
Let $\mathcal{F}_j,1\le j\le J$ denote the layerwise maps whose sup norm, evaluated on its associated domain, is bounded:
\begin{align*}
\mathcal F_j
&=\{f:\mathbb R^d\to\mathbb R^d:\|f\|_{\ell^\infty(\tilde\mu)}<\infty\},\;\; 1\le j<J,\quad
\mathcal F_J
=\{f:\mathbb R^d\to\mathbb R:\|f\|_{\ell^\infty(\tilde\mu)}<\infty\}.
\end{align*}
Further denote $\Theta = \mc F_1\times\mc F_2\times\cdots\times\mc F_J$ as their product space, which will be the space of layer tuples studied in this paper.

In addition, we use $\Pi(d\theta)$ to denote the measure induced by the DGP prior on the space $\Theta$, which is a product measure across layers. With a little abuse of notations, we also use $\Pi(df_j)$ to denote the marginal prior law on the $j$th layer, which is itself the product of the coordinate-wise GP law of $f_{j,k}$. For each layerwise process $\mr{GP}_j$, let $\mc H_j$ denote its reproducing kernel Hilbert space (RKHS), and $\|\cdot\|_{\mc H_j}$ be its corresponding RKHS norm. 
When a restricted RKHS norm on a compact set is used explicitly below, it means the infimum of the ambient RKHS norms of all extensions agreeing on that set.

Finally, for positive sequences $a_n,b_n$, we write $a_n\lesssim b_n$ if $a_n\le Cb_n$ for a constant $C$ independent of $n$, $a_n\gtrsim b_n$ for the reverse relation, $a_n\asymp b_n$ when both hold, and $a_n\ll b_n$ (equivalently $b_n\gg a_n$) when $a_n/b_n\to0$. An expression such as $\ln q_n\gtrsim-r_n$ with $0<q_n\le1$ and $r_n>0$ is shorthand for $-\ln q_n\lesssim r_n$.

\section{Adjacent Symmetries in DGP}\label{sec:non-id}\label{sec:canonical composition}
\subsection{Compositions and Equivalent Classes}
Fix a maximum depth $J\ge 2$. For all $\theta\in\Theta$, recall the composition operator $\mc C(\theta) := f_J\circ\cdots\circ f_1$ previously defined in \eqref{eq:comp}.
Many distinct $\theta$ induce the same composed map $\mc C(\theta)$.
More precisely, for all $f\in \ell^\infty(\X)$, we can define the equivalence class, whose composition equals to $f$, as
\begin{equation}\label{eq:equi class}
\mc E(f) := \{\theta\in\Theta:\ \mc C(\theta)(x)=f(x), \forall x\in \X\}.
\end{equation}
The class $\mc E(f)$ is closed with respect to a large family of operators. Below we provide some examples on the case of $J=2$, which can be easily extended to the case of $J\ge 3$ by applying these operators on the first two layers.

\begin{example}[Sign Flip]\label{example11}
    For all mappings $f:\mathbb{R}^d\to\mathbb{R}^d$, define operators ``$-$'' and ``$\dagger$'' as
    \begin{align*}
(-f)(x):= - f(x), \quad (\dagger f)(x) := f(-x), \quad \forall  x\in\mathbb{R}^d.
\end{align*}
    Then it is easy to verify for all $f_1\in \mc F_1, f_2\in \mc F_2$, $(\dagger f_2) \circ (-f_1) = f_2 \circ f_1$. In other words, the class $\mc E(f)$ is closed with respect to the pair of operators $(-, \dagger)$.
\end{example}

\begin{example}[Translation]\label{example2}
    For a constant vector $a\in\mathbb{R}^d$, define operators $\sharp_a$ and $\flat_a$ by
    \begin{align*}
    (\sharp_a f)(x)=f(x)+a,
    \qquad
    (\flat_a f)(x)=f(x+a).
\end{align*}
    For $f_1\in\mc F_1$ and $f_2\in\mc F_2$, we have
$
    (\flat_{-a}f_2)\circ(\sharp_a f_1)=f_2\circ f_1.
$
\end{example}

\begin{example}[Affine map]\label{example:affine}
    For $a\in\mathbb{R}^d$ and $r\ne 0$, let
    \begin{align*}
    (\sharp_{a,r} f)(x):=r f(x)+a,
    \qquad
    (\flat_{a,r} f)(x):=f((x-a)/r).
\end{align*}
    Then, for all $f_1\in\mc F_1$ and $f_2\in\mc F_2$,
$
    (\flat_{a,r} f_2)\circ(\sharp_{a,r} f_1)=f_2\circ f_1.
$
\end{example}

The above three examples are special cases of the {\it adjacent layer symmetry}. Specifically, let $S$ be a bimeasurable bijection from $\mathbb{R}^d$ to itself such that the transformation below and its inverse are well defined and measurable on $\Theta$. For $1\le j\le J-1$ and $\theta=(f_1, f_2, \cdots, f_J)\in \Theta$, define the {\it adjacent layer symmetry map} as
\begin{equation}\label{eq:Phi jS}
    \Phi_{j,S}(f_1,\ldots,f_J)
    =
    (f_1,\ldots,S\circ f_j, f_{j+1}\circ S^{-1},\ldots,f_J).
\end{equation}
It is straightforward to verify that the adjacent layer symmetry map $\Phi_{j,S}$ preserves the equivalence class structure in \eqref{eq:equi class}.

\begin{proposition}\label{prop:ope}
For any adjacent layer symmetry map $\Phi_{j,S}$ that is well defined on $\Theta$, 
    \begin{align*}
\mc C(\Phi_{j,S}(\theta)) =  \mc C(\theta).
\end{align*}
\end{proposition}

Proposition \ref{prop:ope} shows, even only considering adjacent layers, the equivalence class $\mc E$ is closed under a large family of operators from $\Theta$ to itself. With three or more layers, the class of $\mc E$ can include even more elements by simply stacking multiple adjacent layer symmetry maps. For example, with $S_1, S_2$ being bijections from $\mathbb{R}^d$ to itself, one can apply them to three layers and get
\begin{align*}
\mc C(\Phi_{2,S_2} (\Phi_{1,S_1}(\theta))) = \mc C(\theta).
\end{align*}
In the regression setting \eqref{eq:reg}, the response of $Y_i$ depends on $X_i$ only through composed map $f^*$, regardless of the decomposition $f_1^*, f_2^*, \cdots, f_J^*$, making the individual layers not identifiable.

\subsection{Posterior Non-identifiability under Adjacent Symmetries}
Proposition \ref{prop:ope} describes the non-identifiability of the composition structure under the sampling model, leaving the hope of identifying a representative element among the class $\mc{E}(f)$ to the prior. However, independent centered Gaussian process priors on the layerwise coordinates do not automatically define a unique representative element.

To begin with, let us introduce the main tool for characterizing the posterior behavior, namely the $\rho$-fractional posterior.
Specifically, for $\rho\in(0,1]$, the $\rho$-fractional posterior of $\theta = (f_1, f_2, \cdots, f_J)$ is defined as
\begin{equation*}\label{eq:fractional posterior}
\Pi_\rho(d\theta|\D_n) \propto [p(\D_n|\C(\theta))]^\rho \Pi(d\theta).
\end{equation*}
When $\rho=1$, the $\rho$-fractional posterior reduces to the standard Bayesian posterior. For $\rho\in(0,1)$, contraction in R\'{e}nyi divergence can be obtained from a suitable local prior-mass condition \citep{bhattacharya2019bayesian}.

Recalling the adjacent-layer symmetry map $\Phi_{j,S}$ defined in equation \eqref{eq:Phi jS}, 
define the prior induced by the mapping $\Phi_{j,S}$ as
\begin{equation}\label{eq:transported-prior}
    \Pi^{\Phi_{j,S}}(B):=\Pi\big(\Phi_{j,S}(B)\big),
    \qquad B\subseteq\Theta\ \text{measurable}.
\end{equation}
Assume $\Pi^{\Phi_{j,S}}\ll\Pi$ and define the correctly oriented Radon-Nikod\'{y}m derivative
\begin{equation}\label{eq:radon-nikodym}
    L_{\Phi_{j,S}}(\theta)
    :=\frac{d\Pi^{\Phi_{j,S}}}{d\Pi}(\theta).
\end{equation}
The following compares the posterior mass of a measurable set $B$ and $\Phi_{j,S}(B)$ under the DGP prior.

\begin{proposition} \label{thm:non-id-general}
Let $\Phi_{j,S}$ be an admissible adjacent-layer symmetry map in the sense of \eqref{eq:transported-prior} and \eqref{eq:radon-nikodym}.  Then, for every $\Pi$-measurable $B\subset\Theta$ and every $\rho\in(0,1]$, the $\rho$-fractional posterior satisfies
\begin{align}
    \Pi_\rho(\Phi_{j,S}(B)\mid\D_n)
    =\int_B L_{\Phi_{j,S}}(\theta)\,\Pi_\rho(d\theta\mid\D_n).
\label{eq:L function}
\end{align}
In particular, if $\Pi^{\Phi_{j,S}}=\Pi$, equivalently $L_{\Phi_{j,S}}=1$ $\Pi$-almost surely, then
\[
    \Pi_\rho(\Phi_{j,S}(B)\mid\D_n)=\Pi_\rho(B\mid\D_n).
\]
\end{proposition}

Proposition \ref{thm:non-id-general} shows the ability of the fractional posterior to distinguish the set $B$ versus $\Phi_{j,S}(B)$ relies entirely on the Radon-Nikodym derivative \eqref{eq:radon-nikodym}. Unfortunately, for DGP employing some common Gaussian processes for its layers, and for some simplest adjacent layer symmetry maps, this Radon-Nikodym derivative can be upper and lower bounded, rendering the (fractional) posterior non-identifiable.

\begin{example}[Sign flip, continued]\label{eg:sign flip 2}
Assume that the relevant layerwise Gaussian process laws are centered and invariant under the output sign change and input reflection appearing in the map
\begin{align*}
    \Phi_{j,-}(f_1,\ldots,f_J)
    :=(f_1,\ldots,-f_j,\dagger f_{j+1},\ldots,f_J).
\end{align*}
Then $\Pi^{\Phi_{j,-}}=\Pi$, so $L_{\Phi_{j,-}}=1$ $\Pi$-almost surely.  Hence, for every $\Pi$-measurable $B$,
\[
    \Pi_\rho(\Phi_{j,-}(B)\mid\D_n)=\Pi_\rho(B\mid\D_n),
\]
and the posterior cannot distinguish $\theta$ from its sign-reflected representation $\Phi_{j,-}(\theta)$.
\end{example}

The proof of Example~\ref{eg:sign flip 2} is straightforward by utilizing the symmetry. It effectively shows  $\theta$ and $\Phi_{j,-}(\theta)$ are treated in a completely same way by the posterior.
The same logic applies to the translation example discussed.

\begin{example}[Translation, continued]\label{eg:translation 2}
Let the layerwise GP prior in \eqref{eq:dgp} be a stationary GP with zero mean. 
{\color{black} Suppose the constant function belongs to the RKHS, in the sense that there exists $c>0$
and $\mathbbm{1}_{\X}:\mathbb{R}^d \to \mathbb{R}$, such that $\mathbbm{1}_{\X}(x)= 1,\forall x\in \X$ and $\|\mathbbm{1}_{\X}\|_{\mc H_1}\le c$.} 
Recalling the translation maps $\sharp$ and $\flat$ defined in Example \ref{example2}, for $a\in\R^d$, the adjacent layer symmetry map
\begin{equation*}\label{eq:translation map}
    \Phi_{1,a}(f_1,f_2,\ldots,f_J)
    :=(\sharp_a f_1,\; \flat_{-a} f_2,\; f_3,\ldots,f_J)
\end{equation*}
satisfies 
\begin{equation}
L_{\Phi_{1,a}}(f_1,\cdots, f_J) = \prod_{k=1}^d \exp\Big(-a_k \langle f_{1,k},\mathbbm{1}_{\X} \rangle^{\sim}_{\mc H_1}-\frac{a_k^2}{2}\|\mathbbm{1}_{\X} \|_{\mc H_1}^2\Big), \;\; \mr{a.s.}.
\end{equation}
\end{example}

For more general operators $S$ as in Proposition~\ref{prop:ope}, the Radon-Nikod\'ym derivative before and after the transformation depends on the specific prior and closed form representation is generally not available without additional assumptions.

\subsection{Failure of Layerwise Contraction}
The posterior non-identifiability in Proposition \ref{thm:non-id-general} has a direct consequence for layerwise learning. If for a set $B\subset\Theta$, one can find a collection of adjacent layer symmetry maps, such that all the maps have lower bounded Radon-Nikod\'ym derivatives \eqref{eq:radon-nikodym} on the set $B$, while possessing disjoint images of $B$, then the posterior simply cannot distinguish between the set $B$ and its images under these adjacent layer symmetry maps. If, as the set $B$ shrinks, one can find arbitrarily many adjacent layer symmetry maps satisfying these properties with a common lower bound, then the posterior mass on $B$ can be arbitrarily small. 

Specifically, let $d_\Theta$ be a measurable metric or pseudometric on $\Theta$ and denote, for some $\epsilon>0$, define a set
\begin{align*}
    B_{\theta^*}(\epsilon,d_\Theta)=\{\theta\in\Theta:d_\Theta(\theta,\theta^*)\le\epsilon\}
\end{align*}
as the $\epsilon$-radius $d_\Theta$-ball around the tuple $\theta^*\in\Theta$. 
The standard Bayesian (fractional) posterior contraction theory concerns the (fractional) posterior mass on the ball $B_{\theta^*}(\epsilon,d_\Theta)$, hoping to show the mass goes to one in probability. However, we go through the opposite way. 
The following theorem shows that, under a packing condition for the images of adjacne layer symmetry maps, the posterior mass of $B_{\theta^*}(\epsilon,d_\Theta)$ is uniformly bounded by the reciprocal of the packing number, independently of the sample size $n$.

\begin{theorem}\label{thm:non-conv}
For the regression problem \eqref{eq:reg} with $J$-layer ($J\ge 2$) DGP prior as in \eqref{eq:dgp}, for all $\theta^*\in\Theta$, for all $\epsilon>0$, if there exists $N_\epsilon\in\mathbb{N}$ and a sequence of adjacent layer symmetry maps $\Phi_{j,S_1}, \Phi_{j,S_2}, \cdots, \Phi_{j,S_{N_\epsilon}}$, such that:
\begin{enumerate}[label=(\roman*)]
    \item The sets $\Phi_{j,S_l} (B_{\theta^*}(\epsilon,d_\Theta)), 1\le l\le N_\epsilon$ are pairwise disjoint;
    \item There exists an absolute constant $c_1>0$ such that,
    $
    \operatorname*{ess\,inf}_{\theta\in B_{\theta^*}(\epsilon,d_\Theta)} L_{\Phi_{j,S_l}}(\theta) \ge c_1, \;\forall 1\le l \le N_\epsilon \text{ for }L_{\Phi_{j,S_l}} \text{ in }\eqref{eq:radon-nikodym},
    $
    where the essential infimum is with respect to $\Pi$.
\end{enumerate}
Then, for all $\rho\in(0,1]$, for all sample size $n$ and all observed dataset $\D_n$,
\begin{align*}
    \Pi_\rho(B_{\theta^*}(\epsilon,d_\Theta)\mid\D_n)
    \le (c_1N_\epsilon)^{-1}.
\end{align*}
\end{theorem}

The theorem \ref{thm:non-conv} lays down two elements that prevent the posterior contraction to the layerwise structure: a collection of adjacent layer symmetry maps that have disjoint images and uniformly lower bounded Radon-Nikod\'{y}m derivative. 
The next result applies when the RKHS contains a function equal to $1$ on $\tilde\X$. Mat\'{e}rn priors satisfy this condition because their RKHSs contain smooth functions of compact support.

\begin{corollary}
\label{cor:non-conv matern}
For the regression problem \eqref{eq:reg} with $J$-layer ($J\ge 2$) DGP prior defined in \eqref{eq:dgp}, suppose the layerwise GP priors are zero mean stationary GPs, and suppose the constant function belongs to the RKHS, in the sense that there exists $c>0$
and $\mathbbm{1}_{\tilde{\X}}:\mathbb{R}^d \to \mathbb{R}$, such that $\mathbbm{1}_{\tilde{\X}}(x)= 1,\forall x\in \tilde{\X}$ and $\|\mathbbm{1}_{\tilde{\X}}\|_{\mc H_j}\le c, \forall j<J$. 
Then for all $\theta^*\in\Theta$,
for all $\epsilon\in(0,1/2)$, $\rho\in(0,1]$, $p\in[1,+\infty]$, 
for all sample size $n$ and all observed dataset $\D_n$,
$$\Pi_\rho(\|\theta-\theta^*\|_{\ell^p(\tilde{\mu})}\le \epsilon|\D_n) \lesssim \epsilon.$$
The implicit constant is independent of $n$, $\D_n$, $\epsilon$, and $\theta^*$, but may depend on the fixed GP kernels and $\tilde\X$.
\end{corollary}

Together, Theorem~\ref{thm:non-conv} and Corollary \ref{cor:non-conv matern} provide negative information regarding the concentration of the fractional posterior around a fixed representative. Through the adjacent layer symmetry maps, it is possible to construct a sequence of sets possessing comparable (fractional) posterior mass to the set $B_{\theta^*}(\epsilon, d_\Theta)$, thus making the posterior contraction infeasible. 
We would like to comment that this infeasibility concerns the contraction around a specific representative tuple $\theta^*$ at the per-layer level, which is not in conflict with the contraction around $f^*$ at the function level previously proved in literature \citep{finocchio2023posterior,castillo2025deep}. 
In the next section, we propose our augmented DGP, which reconciles infeasibility and contraction class by favoring intermediate layers close to identity, selecting among likelihood-equivalent decompositions without changing the DGP architecture. Unlike function-space normalizing flows \citep{shi2024universal}, it regularizes layer representations rather than constructing invertible probability transports.

\section{Symmetry-breaking by Augmentations}\label{sec:asym-penalty}
We regularize DGP representations by exponentially tilting the layer priors toward identity. The factors act directly on layer functions, and their weights vary across depth. This supplies a common mechanism for preferring lower-penalty representations among likelihood-equivalent decompositions and controlling the magnitude of redundant-layer departures from identity.
For the identity mapping $\mr{id}$  on $\mathbb{R}^d$, for all $1\le j\le J-1$, define the identity-deviation factor
\begin{equation}\label{eq:Rj}
    R_j(f_j)=\exp\{-\eta_{j,n}\|f_j-\id\|_{\ell^p(\tilde\mu)}\},
\end{equation}
and set $R_J\equiv1$. The parameter $\eta_{j,n}\ge0$ controls   strength of   the augmentation at layer $j$.

\begin{definition}[Augmented fractional posterior]\label{def:augmented_fractional_posterior}
    For $R_j$ defined in \eqref{eq:Rj}, the augmented prior is defined as
    \begin{equation*}\label{eq:augmented-prior}
        \tilde{\Pi}(d\theta) \propto \prod_{j=1}^{J-1} R_j(f_j) \;\Pi(d\theta).
    \end{equation*}
    and the augmented
    $\rho$-fractional posterior is therefore defined as
    \begin{equation}\label{eq:augmented-posterior}
        \tilde{\Pi}_\rho(d\theta|\D_n) \propto [p(\D_n|\C(\theta))]^\rho \prod_{j=1}^{J-1} R_j(f_j) \;\Pi(d\theta),
    \end{equation}
following the convention in \citet{linero2025defensive} and \citet{bhattacharya2019bayesian}.
\end{definition}

The penalty grows {\color{black} exponentially} with distance from identity: on $\|f_j-\id\|_{\ell^p(\tilde\mu)}>t$, the factor $R_j$ is at most $e^{-\eta_{j,n}t}$. This tail suppression, together with sufficient augmented prior mass near useful compositions, yields the localization radius in Theorem~\ref{thm:parsimonious}. We use the schedule
\begin{equation*}\label{eq:eta order}
    \eta_{1,n}\ge\eta_{2,n}\ge\cdots\ge\eta_{J-1,n}\ge 1.
\end{equation*}
Earlier hidden layers are penalized more strongly, leaving later layers more flexible and the output layer unpenalized. Theorems~\ref{thm:contraction}--\ref{thm:parsimonious} quantify how strong the early penalties can be while preserving the prior mass supplied by the composition of the later layers.

The augmentation changes posterior odds between symmetry-related decompositions through the difference in their total identity penalty. Specifically, for an adjacent-layer symmetry map $\Phi_{j,S}$ from \eqref{eq:Phi jS}, define the difference in total penalty
\begin{equation}\label{eq:zero energy}
\Delta_{j,S,n}(\theta)\coloneqq
\begin{cases}
\eta_{j,n}\big[\|S\circ f_j-\id\|_{\ell^p(\tilde\mu)}-\|f_j-\id\|_{\ell^p(\tilde\mu)}\big]\\
\qquad+\eta_{j+1,n}\big[\|f_{j+1}\circ S^{-1}-\id\|_{\ell^p(\tilde\mu)}-\|f_{j+1}-\id\|_{\ell^p(\tilde\mu)}\big],
& 1\le j\le J-2,\\[0.4em]
\eta_{J-1,n}\big[\|S\circ f_{J-1}-\id\|_{\ell^p(\tilde\mu)}-\|f_{J-1}-\id\|_{\ell^p(\tilde\mu)}\big],
& j=J-1,
\end{cases}
\end{equation}
where second line reflects that the final layer is unpenalized. The term $\Delta_{j,S,n}(\theta)$ is exactly the total penalty at $\Phi_{j,S}(\theta)$ minus that at $\theta$. If $\eta_{j+1,n}\ll\eta_{j,n}$ and the neighboring-layer norm difference is uniformly bounded, its contribution divided by $\eta_{j,n}$ vanishes uniformly.

The following theorem gives the exact change in posterior mass under an adjacent-layer symmetry. In particular, if the combined factor $L_{\Phi_{j,S}}(\theta)e^{-\Delta_{j,S,n}(\theta)}$ equals one almost surely, then the augmented posterior is invariant under that symmetry.

\begin{theorem}\label{thm:noniden break}
Let $B\subset\Theta$ be measurable and let $\Phi_{j,S}$ be an adjacent-layer symmetry satisfying the assumptions of Proposition~\ref{thm:non-id-general}. Then, for every $\rho\in(0,1]$,
\begin{equation}\label{eq:aug post}
\tilde\Pi_\rho(\Phi_{j,S}(B)\mid\D_n)
=
\int_B L_{\Phi_{j,S}}(\theta)
\exp\{-\Delta_{j,S,n}(\theta)\}
\tilde\Pi_\rho(d\theta\mid\D_n).
\end{equation}
If $\tilde\Pi_\rho(B\mid\D_n)>0$ and $0<m\le L_{\Phi_{j,S}}(\theta)\le M<\infty$ for $\Pi$-almost every $\theta\in B$, then
\[
m e^{-\sup_{\theta\in B}\Delta_{j,S,n}(\theta)}
\le
\frac{\tilde\Pi_\rho(\Phi_{j,S}(B)\mid\D_n)}{\tilde\Pi_\rho(B\mid\D_n)}
\le
M e^{-\inf_{\theta\in B}\Delta_{j,S,n}(\theta)}.
\]
Consequently, the ratio tends to $0$ when $\inf_B\Delta_{j,S,n}\to+\infty$. It tends to $+\infty$ when $\sup_B\Delta_{j,S,n}\to-\infty$.
\end{theorem}

Theorem~\ref{thm:noniden break} gives the exact symmetry-breaking mechanism through the product of the prior Radon-Nikod\'ym factor and the exponential penalty factor. Under the bounded Radon-Nikod\'{y}m derivative condition in the theorem, the penalty difference $\Delta_{j,S,n}$ controls the ratio of the posterior masses of $\Phi_{j,S}(B)$ and $B$. An ordering $\eta_{j+1,n}\ll\eta_{j,n}$ alone does not imply separation: the theorem requires a uniformly diverging difference in total penalty. The following examples display the first-layer equality that can obstruct such a gap:
\begin{equation}\label{eq:delta zero}
\|S\circ f_j-\id\|_{\ell^p(\tilde\mu)}-\|f_j-\id\|_{\ell^p(\tilde\mu)} = 0,
\end{equation}
which is very unlikely to happen unless $f_j$ is deliberately engineered. Below, assuming $\eta_{j+1,n}\ll \eta_{j,n}$, we continue with the previously discussed
examples of the bijection $S$, demonstrating how the adjacent layer symmetry gets removed with
the introduction of augmentation.

\begin{example}[Sign flip, continued]
    Continuing Example~\ref{eg:sign flip 2}. Having equation \eqref{eq:delta zero} is equivalent to 
    \begin{align*}
    \|f_j-\id\|_{\ell^p(\tilde\mu)}=\|f_j+\id\|_{\ell^p(\tilde\mu)}.
\end{align*}
\end{example}

\begin{example}[Affine maps, continued]
Continuing Example~\ref{example:affine}, for $S(z)=rz+a$, having equation \eqref{eq:delta zero} is equivalent to 
\[
\|f_j-\id\|_{\ell^p(\tilde\mu)}
=\|r f_j+a-\id\|_{\ell^p(\tilde\mu)}.
\]
\end{example}

\section{Function-level Contraction and Layerwise Parsimony}\label{sec:theory}\label{sec:contraction}

\subsection{General Theory}
The central methodological result is a joint calibration of function estimation and representation regularization. The proof balances two effects of the same exponential tilt: its cost on neighborhoods of useful compositions and its suppression of large identity deviations. A proxy construction transfers approximation by a later active block to the full DGP with near-identity earlier layers. This yields a common penalty window for the two guarantees. All the contraction results concern the augmented fractional posterior $\tilde\Pi_\rho(\cdot|\D_n)$.

We write $\mathbb{P}$ for probability under the Gaussian process law being discussed below.  The first assumption concerns the a layerwise small ball probability of the GP priors.

\begin{assumption}[Layerwise regularity]\label{assump:layer regular}
    There exists $\alpha>0$ such that, for every hidden-layer coordinate $Z_{j,k}$, $j<J$, $k\le d$, and every sufficiently small $\epsilon>0$, the centered small ball probability on the enlarged compact set satisfies
    \begin{align*}
    \ln \mathbb{P}(\|Z_{j,k}\|_{\ell^\infty(\tilde\mu)} <\epsilon) \gtrsim -\epsilon^{-d/\alpha}.
\end{align*}
There also exists $c>0$, such that $\forall 1\le j<J, \; \forall 1\le k \le d$, there exists $\mr{id}_k:\mathbb{R}^d\to \mathbb{R}$, such that $\mr{id}_k(x[1],x[2],\cdots, x[d]) = x[k], \;\forall x\in \tilde\X$ and $\|\mr{id}_k\|_{\mc H_j}\le c$.
\end{assumption}

Assumption \ref{assump:layer regular} implies,  by the Cameron--Martin bound for shifting a centered Gaussian small ball by an RKHS function, as in \eqref{eq:multivaraite small ball}, for all $j<J$, 
\begin{align*}
    \ln \Pi(\|f_j-\id\|_{\ell^\infty(\tilde\mu)}<\epsilon)
    \ge \ln \prod_{k=1}^d  [\mathbb{P}(\|Z_{j,k}\|_\infty <\epsilon)     
    \exp(-c^2/2)] \gtrsim -\epsilon^{-d/\alpha}.
\end{align*}
with details provided in Appendix \ref{sec:appen small ball}.

The remaining assumptions are stated in terms of two positive sequences $\eps_n\to0$ and $\tilde\eps_n\to0$, with $n\eps_n^2\to\infty$.  The quantity $\eps_n$ is the target function-level contraction rate, while $\tilde\eps_n$ is a layerwise localization scale used in the prior-mass construction.  Depending on the approximation, $\tilde\eps_n$ may be smaller than, or of the same order as, $\eps_n$.

We also introduce a proxy $f_n^*$ for the true regression function.  The proxy is $\epsilon_n$-close to $f^*$ in $\ell^2(\mu_X)$ but satisfies a mild local regularity condition on the enlarged domain $\tilde{X}$.

\begin{assumption}[Proxy regularity]\label{assump:proxy regular}
    Assume $\|f^*\|_{\ell^\infty(\X)} \le 1$. For $0<\tilde{\eps}_n\le \epsilon_n<1$ satisfying $\tilde{\epsilon}_n^{-d/\alpha} \lesssim n\epsilon_n^2$, there exists a proxy $ f^*_n$ with uniformly bounded $\|f_n^*\|_{\ell^\infty(\tilde\mu)}$ such that $\|f^*- f^*_n\|_{\ell^2(\mu_X)} \le \eps_n$ and
\begin{equation}\label{eq:proxy lipschitz}
\sup_{x,x'\in\tilde\X:\,\|x-x'\|_\infty< \tilde{\epsilon}_n} | f^*_n(x) -  f^*_n(x')| \le \epsilon_n.
\end{equation}
\end{assumption}

Assumption \ref{assump:proxy regular} puts some minimal regularity on the data-generating function, established through 
the two versions of epsilon in the assumption statement, $\tilde\epsilon_n$ and $\epsilon_n$. The $\tilde\epsilon_n$ roughly depicts the prior mass condition on the redundant layers. Under the assumption that identity mapping having bounded RKHS norms, this is solely dictated by the layerwise GP priors that satisfy Assumption \ref{assump:layer regular}. The $\epsilon_n$ governs the posterior concentration of the whole DGP model, and is affected by both the layerwise GP priors and the data-generating function.

The two scales $\epsilon_n$ and $\tilde{\epsilon}_n$ have complementary roles. The condition $\tilde\epsilon_n^{-d/\alpha}\lesssim n\epsilon_n^2$ and Assumption~\ref{assump:layer regular} give
\[
\ln\Pi(\|f_j-\id\|_{\ell^\infty(\tilde\mu)}<\tilde\epsilon_n)
\gtrsim -n\epsilon_n^2.
\]
Thus neighborhoods of identity have sufficient prior mass at the target estimation rate. Equation~\eqref{eq:proxy lipschitz} ensures that perturbing the input to $f_n^*$ by at most $\tilde\epsilon_n$ changes its output by at most $\epsilon_n$. Together, these conditions allow earlier layers near identity to precede an accurate composition of later layers without losing the estimation rate. Because this continuity condition is imposed on an approximation, the same argument can accommodate a discontinuous truth through $\ell^2$ approximation.

The third assumption describes how close the proxy data-generating function can be approximated by sample paths of a subset of deep GP layers, in the form of the classical prior mass condition. 
As long as the Assumption \ref{assump:proxy regular} is also satisfied, the prior mass around $f_n^*$ can be translated into the prior mass around the true data-generating function $f^*$ under suitable norms. Such prior mass condition is mandatory for proving either the standard posterior contraction or the fractional posterior contraction.
\begin{assumption}[Prior Mass]\label{assump:prior mass}
    There exist $1\le J_0\le J$ and $\epsilon_n\gtrsim n^{-\frac{\alpha}{2\alpha+d}}$, such that
    \begin{align*}
    \ln \Pi\big(\|f_J\circ \cdots \circ f_{J_0} - f^*_n\|_{\ell^\infty(\tilde\mu)}\le\epsilon_n,\; \eta_{j,n}\|f_j-\id\|_{\ell^p(\tilde{\mu})}\le  n\epsilon_n^2,\; J_0\le j\le J-1\big) \gtrsim -n\epsilon_n^2.
\end{align*}
\end{assumption}

In Assumption~\ref{assump:prior mass}, the composition $f_J\circ\cdots\circ f_{J_0}$ approximates $f_n^*$. We call the earlier layers $j<J_0$ redundant relative to this approximation: they can be replaced by identity maps without changing the approximating composition. The index $J_0$ belongs to the theoretical construction; it is not an estimated depth. 
In comparison with the conventional prior mass condition, Assumption \ref{assump:prior mass} additionally requires $\eta_{j,n}\|f_j-\id\|_{\ell^p(\tilde{\mu})} \lesssim n\epsilon_n^2$  for the penalized layers $J_0\le j<J$ used in this approximation, so that the augmentation terms do not negatively affect the prior mass around the proxy $f_n^*$. In many cases, we have norms of $f_j$ bounded by a constant, and the condition regarding $\eta_{j,n}$ simply translates to 
$$ \eta_{j,n}\lesssim n\epsilon_n^2, \qquad J_0\le j\le J-1.$$

Before the statement of posterior contraction results, we also need to introduce the R\'{e}nyi divergence, which is the natural divergence used for fractional posterior \citep{bhattacharya2019bayesian}. Let $\mathbb{P}_{f^*}$ be the probability distribution of a sample $(X_i, Y_i)$, which is the product of marginal distribution of $X_i$
and the conditional distribution of $Y_i|X_i$
, with its density calculates as
\begin{equation}\label{eq:pf}
	p_{f^*}(x,y) := \frac{d \mathbb{P}_{f^*}(x,y)}{dy d\mu_X(x)} = \frac{1}{\sqrt{2\pi}\sigma} e^{-\frac{[y-f^*(x)]^2}{2\sigma^2}}.
\end{equation}
Further define the product measures $d\nu_n = (dy d\mu_X(x))^{\otimes n}$ and $d\mathbb{P}^n_{f^*} = (d\mathbb{P}_{f^*})^{\otimes n} $, and let $p_{f^*}^n = d\mathbb{P}_{f^*}^n/d\nu_n$ be the product density.

\begin{definition}
	\label{def:renyi}
	For $f, f^*\in \ell^\infty(X)$, for $0<\rho<1$, define its $\rho$-R\'{e}nyi divergence as
\begin{equation}\label{eq:renyi}
	D_\rho(p_{f}^n,p_{f^*}^n):=\frac{1}{\rho-1} \ln 
	\int (p_{f}^n)^\rho (p_{f^*}^n)^{1-\rho} d\nu_n= \frac{n}{\rho-1}  \ln \int_{\X} e^{-\frac{\rho(1-\rho)}{2\sigma^2} [f(x)-f^*(x)]^2} d\mu_X(x). 	 
\end{equation}
\end{definition}

We also introduce a condition regarding the boundedness of the composed regression function, which can convert R\'{e}nyi divergence to $\ell^2$ norm.

\begin{assumption}[Bounded Output, Optional]\label{assump:bounded}
The sample path of the final layer is bounded, i.e., there exists an absolute constant $c>0$, such that $\Pi(\|f_J\|_\infty<c)=1$.
\end{assumption}

We are now ready to present the first main theorem, which shows the contraction rate achieved by a subset of deep GP layers is preserved after applying the augmentation in equation \eqref{eq:augmented-posterior}, provided the $\eta_{j,n}$ are chosen to be relatively small for that subset of layers.

\begin{theorem}[Contraction preservation]\label{thm:contraction}
    Fix $0<\rho<1$ and assume $\mu_X, \tilde{\mu}$ have upper bounded density with respect to the Lebesgue measure. Under Assumptions \ref{assump:layer regular}-\ref{assump:prior mass}, if
    $\eta_{j,n}\lesssim n\epsilon_n^2/\tilde\epsilon_n$ for all $j<J_0$, then for sufficiently large $M>0$,
    \begin{align*}
    \tilde{\Pi}_\rho\!\left(
    D_\rho\!\left(p_{\mathcal C(\theta)}^n, p_{f^*}^n\right)>M n\epsilon_n^2
    \,\middle|\,\D_n\right)
    \overset{\mathbb P_{f^*}^n}{\longrightarrow}0 .
    \end{align*}
    If, in addition, Assumption~\ref{assump:bounded} holds, then for sufficiently large $M>0$,
    \begin{align*}
    \tilde{\Pi}_\rho\big(\|\mathcal{C}(\theta)- f^*\|_{\ell^2(\mu_X)}> M \epsilon_n\,\big|\,\D_n\big)
    \overset{\mathbb P_{f^*}^n}{\longrightarrow}0 .
    \end{align*}
\end{theorem}

The first conclusion is the direct fractional posterior statement and does not require an almost sure deterministic bound on a Mat\'{e}rn sample path. The second follows from the R\'{e}nyi-to-$\ell^2(\mu_X)$ conversion under the bounded output assumption.

The next theorem bounds the posterior distance of each earlier layer from identity. This distance tends to zero when $\eta_{j,n}\gg n\epsilon_n^2$.
\begin{theorem}[Layerwise parsimony]\label{thm:parsimonious}
    Fix $0<\rho<1$ and assume $\mu_X, \tilde{\mu}$ have upper bounded density with respect to the Lebesgue measure. Under Assumptions \ref{assump:layer regular} to \ref{assump:prior mass}, if
    $n \epsilon_n^2 \lesssim \eta_{j,n} \lesssim n\epsilon_n^2/\tilde\epsilon_n$ for all $j<J_0$, then for every $1\le j<J_0$ and sufficiently large $M>0$,
    \begin{align*}
    \tilde{\Pi}_\rho(\|f_j-\id\|_{\ell^p(\tilde{\mu})}>M n\epsilon_n^2/\eta_{j,n}|\D_n) \overset{\mathbb P_{f^*}^n}{\longrightarrow} 0.
\end{align*}
\end{theorem}
Since $d\mu_X/d\tilde\mu$ is essentially bounded, the bound in $\ell^p(\tilde{\mu})$ also holds, up to a fixed multiplicative constant, in $\ell^p(\mu_X)$. Together, the theorems give a design rule for the layerwise penalties, conditional on the later-layer composition satisfying Assumption~\ref{assump:prior mass}. Their joint conclusion concerns both the composed regression and the redundant earlier layers under the same augmented posterior.
The main tuning window for redundant layers is
\begin{equation}\label{eq:eta-window-intuition}
    n\eps_n^2
    \lesssim
    \eta_{j,n}
    \lesssim
    \frac{n\eps_n^2}{\tilde\eps_n},
    \qquad j<J_0.
\end{equation}

The lower scale $n\epsilon_n^2$ determines when the penalty controls the posterior radius; taking $\eta_{j,n}\gg n\epsilon_n^2$ makes that radius shrink. The upper bound keeps the penalty cost on an identity neighborhood within $O(n\epsilon_n^2)$, preserving the prior mass used for function-level contraction. Since $\tilde\epsilon_n\to0$, the window contains choices giving both conclusions. At its upper scale, $\eta_{j,n}\asymp n\epsilon_n^2/\tilde\epsilon_n$, the redundant-layer radius is of order $\tilde\epsilon_n$.
It is worth noting that the function-level contraction requires the layerwise prior mass in $\ell^\infty$ norm, whereas the augmentation penalizes the $\ell^p$ norm. Due to the norm differences, the prior mass on redundant layers are still driven by the GP priors, which is the reason we need an upper bound on the penalty parameters for $j<J_0$.

For layers $J_0,\ldots,J-1$, the penalty on the approximation event in Assumption~\ref{assump:prior mass} must satisfy:
\begin{equation}\label{eq:eta range}
    \eta_{j,n}\|f_j-\id\|_{\ell^p(\tilde{\mu})}\lesssim n\epsilon_n^2, \;J_0\le j\le J-1,
\end{equation}
so that the prior mass dominates the augmentation and the contraction rate at the function level is preserved. In other words, while earlier layers can be strongly regularized,  later layers need to retain the ability to approximate the regression function. The rate $\epsilon_n$ is specified in Assumption \ref{assump:prior mass}, which cannot be faster than $n^{-\frac{\alpha}{2\alpha+d}}$ due to Assumption \ref{assump:proxy regular}.
The intuition is that, for  
a complex function $f^*$ that potentially belongs to some huge, irregular function families,  if more than one DGP layers are needed to estimate $f^*$, the (fractional) posterior contraction rate shall be no faster than the rate of estimating an H\"{o}lder-$\alpha$ function with a single GP. {\color{black} The next section illustrates the rate and penalty parameters with various examples.}

\subsection{Applications of the General Theory}
\label{sec:apply theory}

There are limited theoretical investigation regarding the advantages of DGPs versus single GPs on specific data-generating functions, with the existing results mostly concentrated on compositional smooth functions similar to Corollary \ref{cor:intrinsic} \citep{castillo2025deep}. To the best of our knowledge, the theoretical study of DGPs on functions with local structures like Corollary \ref{eg:non-diff} and \ref{cor:discont}, has not been done. 

We assume $\rho\in(0,1)$ and $\mu_X, \tilde{\mu}$ have upper bounded density with respect to the Lebesgue measure in all the cases discussed below. The posterior contraction rates are in R\'{e}nyi divergence, and if Assumption \ref{assump:bounded} holds, in $\ell^2({\mu_X})$ norm.
The proofs for all the corollaries in this section can be found in Appendix~\ref{sec:proof corollaries}.

\paragraph{H\"{o}lder Smooth Adaptivity.}
Denote the H\"{o}lder $\beta$-smooth function family with domain $\X$ and H\"{o}lder norm $1$ as $C^{\beta}(\X)$, which is one of the most common objects of studies in nonparametric statistics. When estimating a H\"{o}lder-$\beta$ smooth function using a Mat\'{e}rn process with regularity parameter $\alpha$,  the Mat\'{e}rn process is called ``correctly specified'' if $\alpha=\beta$, and ``oversmooth'' or ``undersmooth'' if $\alpha>\beta$ or $\alpha<\beta$, respectively. In the following, we illustrate the oversmoothed cases, where the $\eta_{j,n}$ are adjusted such that two GP layers remain active. The last but one GP layer acts as a rescaling to the input data, stretching the input from the domain $[-1,1]$ to roughly $[-n^{\frac{\alpha-\beta}{(2\beta+d)\alpha}}, n^{\frac{\alpha-\beta}{(2\beta+d)\alpha}}]$, and the last GP layer operates on the rescaled data. With properly chosen $\eta_{j,n}, 1\le j\le J$, we can make sure that the last two GP layers exercise their functionality stated above, whereas the other layers remain inactive.

\begin{corollary}[H\"{o}lder smooth, oversmooth]\label{eg:Holder over}
Let $\beta>0$ and $f^*\in C^\beta(\X)_1$. Take $\tilde\X$ to be a compact cube and assume a $C^\beta$ extension to a fixed neighborhood of $\tilde\X$, with norm bounded by a fixed constant. Consider an augmented DGP with at least two layers and independent centered, fixed-scale Mat\'ern priors for all layers, each a Mat\'{e}rn
process with regularity $\alpha$, such that $\alpha\in(\max\{\beta,1\},2\beta]$ if $d=1$ and $\alpha>\max\{\beta,1\}$ if $d\ge 2$. Choose $\eta_{J-1,n}\lesssim n^{\frac{\alpha(d-1)+\beta}{(2\beta+d)\alpha}}$ and $n^{\frac{d}{2\beta+d}} \ll \eta_{j,n} \lesssim n^{\frac{\alpha+d}{2\beta+d}} $ for all $1\le j\le J-2$, then the augmented DGP attains the posterior contraction rate $n^{-\frac{\beta}{2\beta+d}}$. Moreover, except the $J-1$ and $J$th layers, all other layers contract to the identity mapping
under the fractional augmented posterior.
\end{corollary}

\paragraph{Local Structures.}
The augmented DGP framework may also be useful for functions with local structures that are difficult to capture from global smoothness alone.

We first study the case where the function is differentiable except at one point in the domain. Via some careful choice of the proxy function that ``smoothens'' out the non-differentiable point, it is possible to attain contraction rate faster than that of H\"{o}lder-1 smooth functions.

\begin{corollary}[Non-differentiable cusp]\label{eg:non-diff}
	Let $\X=[-1,1]$ and let
	\begin{align*}
		f^*(x) = \left\{ \begin{aligned}
			x, &\quad x\in [-1,0],\\
			-x, &\quad x\in (0,1].
		\end{aligned}\right.
	\end{align*}
	Then $f^*$ is not differentiable at $x=0$. Consider an augmented DGP where the GP prior for all layers are the Mat\'{e}rn process with regularity $\alpha\le 3/2
    $. Choose  $n^{\frac{1}{2\alpha+1}}\ll\eta_{j,n}\lesssim n^{\frac{\alpha+1}{2\alpha+1}}, \forall 1\le j\le J-1$. Then the augmented DGP contracts to $f^*$ at the rate $n^{-\frac{\alpha}{2\alpha+1}}$. Moreover, except the $J$th layers, all other layers contract to the identity mapping
	under the fractional augmented posterior.
\end{corollary}

We proceed to consider the case where the data-generating function $f_n^*$ is discontinuous. Via the tool of the proxy $f_n^*$, it is possible to prove the posterior contraction rate (in $\ell^2$ norm) for this discontinuous function, without knowing the pattern of discontinuity. To the best of our knowledge, this is a new result for the GP (or DGP) community.

\begin{corollary}[Discontinuous step function]\label{cor:discont}
Consider a function $f^*$ defined as
\begin{align*}
    f^*(x) = \left\{ \begin{aligned}
        -1, &\quad x\in [-1,0],\\
        1, &\quad x\in (0,1].
    \end{aligned}\right.
\end{align*}
Then $f^*$ is discontinuous at $x=0$.  Consider an augmented DGP with $J$ layers and independent centered, fixed-scale Mat\'{e}rn process with regularity $\alpha>0$. Let $J^*=\lceil4\alpha^2/3\rceil$. If $J\ge J^*$ and $\eta_{j,n} \lesssim n^{\frac{3J^*-2\alpha}{(2\alpha+3)J^*}}$ for all $j> J-J^*$, and $\eta_{j,n}\lesssim n^{\frac{3\alpha+3}{2\alpha+3}}$ for $j\le J-J^*$, then the augmented fractional posterior contracts to $f^*$  at some rate no slower than $n^{-\frac{\alpha}{2\alpha+3}}$. In addition, if $n^{\frac{3}{2\alpha+3}} \ll \eta_{j,n} \lesssim n^{\frac{3\alpha{+3}}{2\alpha+3}}$ for all $j\le J - J^*$, then the first $J-J^*$ layers will contract around the identity mapping due to the effects of the augmentation.
\end{corollary}

The contraction rate $n^{-\frac{\alpha}{2\alpha+3}}$ shows, the estimation rate can actually be arbitrarily fast, approaching the parametric rate $n^{-1/2}$, should one choose sufficient large $\alpha$ for layerwise GPs and sufficient deep architecture for DGP, as if the discontinuity plays little effect on the estimation quality. Due to the lack of studies in literature, we do not yet know whether the rate $n^{-\frac{\alpha}{2\alpha+3}}$ is optimal or not. The proof technique for this corollary also deserves independent interests, as it exhausts all the possible $\tilde{\epsilon}_n, \epsilon_n$ that satisfy Assumptions \ref{assump:proxy regular}, \ref{assump:prior mass}, producing the most favorable rate. Both the optimal rate, as well as the techniques to find the optimal rates, are promising further research topics.

\paragraph{Intrinsic Low  Dimensional Structures.}
While the H\"{o}lder smoothness and local structures can be illustrated in unidimensional examples, certain structures only exist in $d$-dimensional ($d\ge 2$) space. We are particularly interested in intrinsic lower dimensional structures. That is, despite the true data-generating function $f^*$ is a mapping from $\mathbb{R}^d$ to itself, it admits a canonical decomposition $f^* = f_J^*\circ \cdots \circ f_2^* \circ f_1^*$, where each $f_j^*$ has either support or image belonging to a lower dimensional subspace of $\mathbb{R}^d$. In such situation, it would be useful for the DGP prior to adapt to the intrinsic lower dimensional structure, providing larger prior mass and better (fractional) posterior contraction rates. The following corollary provides such a case.

\begin{corollary}[Intrinsic structure]\label{cor:intrinsic}
	Let $\X=[-1,1]^2$ and let the true data-generating function be $f^*=g^*(x_1/2+x_2/2)$, where $g^*:[-1,1]\to\mathbb{R}\in C^\beta([-1,1])_1$ for $\beta = 1$. Assume a DGP model with $J\ge 2$, where the first $J-1$ layers have layerwise GP prior being Mat\'{e}rn process of regularity $\alpha$, for $\alpha\ge 2$, and the last layer follows Mat\'{e}rn process of regularity $\beta$. Choose $\eta_{J-1,n}\ll n^{\frac{1}{3}}$ and $n^{\frac{1}{3}} \lesssim \eta_{j,n} \lesssim n^{\frac{1}{3} + \frac{\alpha}{4}} $ for all $1\le j\le J-2$. Then $f^*$ can be optimally estimated with DGP at the rate of $n^{-\frac{1}{3}}$. 
\end{corollary}
The proposed rate in Corollary~\ref{cor:intrinsic} is the univariate H\"{o}lder rate for $g^*$; establishing it for this fixed-dimensional DGP remains contingent on a valid prior-mass argument.

\section{Numerical illustrations}\label{sec:experiments}

\paragraph{Setup and diagnostics.}
Throughout the synthetic experiments, the observed domain is $\X=[-1,1]^d$, whereas the identity penalty and all layerwise diagnostics are evaluated on the enlarged domain $\tilde\X=[-2,2]^d$.  We take $\tilde\mu$ to be the uniform probability measure on $\tilde\X$ and approximate its $L^2(\tilde\mu)$ norm on a deterministic grid $\mathcal G$.  Thus final prediction is evaluated on $\X$, while anchoring and representation diagnostics are evaluated on $\tilde\X$.  In the supplied MCMC implementation, penalty is evaluated using each retained state's GP conditional mean layer map; we use a finite-grid approximation to the function-space target.

For the main unidimensional experiments, we use $n=30$, $250$ test points, and $\sigma=0.10$.  For the main two-dimensional experiments we use $n=100$, a $23\times23$ test grid, and $\sigma=0.10$.  Each fitted model is a four-layer DGP.  We compare the ordinary fractional posterior with the identity-augmented target
\begin{align*}
    d\tilde\Pi_\rho(\theta\mid\D_n)
    \propto
    p(\D_n\mid\C(\theta))^\rho
    \exp\left\{-\sum_{j=1}^{3}\eta_{j,n}\|f_j-\id\|_{L^2(\tilde\mu)}\right\}
    d\Pi(\theta).
\end{align*}
Here the term \emph{anchor} refers to the fixed target $\id$. The tunable quantities are the penalty weights (i.e., ``anchor parameters'').  With the notational convention $\eta_{4,n}=0$, we parameterize them as
$(\eta_{1,n},\eta_{2,n},\eta_{3,n},\eta_{4,n})
    =\lambda(w_1,w_2,w_3,0),
     w_1\ge w_2\ge w_3\ge0$,
where $\lambda=0$ gives the regular DGP.  
The design-specific base schedules are listed in Appendix~\ref{subsec:synthetic-designs} and Appendix~\ref{sec:experiments-2d-designs}. 

The reported diagnostics are test root mean squared error (RMSE), the redundant-layer identity error $\Delta_{\id}$, and the partial-composition error $\Delta_{\mathrm{partial}}$. Let $\mc R$ be the set of redundant non-final layers and $\mc K$ the corresponding redundant partial-composition indices. For a posterior-mean tuple $\hat\theta=(\hat f_1,\ldots,\hat f_4)$, define $\hat C_k=\hat f_k\circ\cdots\circ\hat f_1$ and $C_k^*=f_k^*\circ\cdots\circ f_1^*$. Then
\begin{align*}
    \Delta_{\id}
    &=\frac{1}{|\mc R|}\sum_{j\in\mc R}
    \|\hat f_j-\id\|_{L^2(\mc G)}, \quad
    \Delta_{\mathrm{partial}}
    =\frac{1}{|\mc K|}\sum_{k\in\mc K}
    \|\hat C_k-C_k^*\|_{L^2(\mc G)}.
\end{align*}  Smaller values indicate a more stable and parsimonious layer representation.  For deliberately canceling symmetry traps, $\Delta_{\id}$ is the primary canonicalization diagnostic because the supplied canceling decomposition is not itself the target selected by the augmentation.  Design details are deferred to Appendix~\ref{subsec:synthetic-designs} and Appendix~\ref{sec:experiments-2d-designs}.

Figure~\ref{fig:edge-transition-anchor} illustrates the intended allocation of flexibility: the two strongly penalized redundant layers remain near identity, while the weakly penalized third layer and unpenalized output retain nonlinear structure. All four component MSEs are smaller under augmentation in this run. This supports the proposed regularization mechanism as a finite-sample diagnostic; penalty calibration and recovery of a unique decomposition are separate questions.

\paragraph{Comparison with regular DGP.}
Tables~\ref{tab:experiments-1d-summary} and ~\ref{tab:experiments-2d-summary} of Appendix \ref{subsec:synthetic-designs} show the comparison between regular DGP and the proposed augmented DGP, where our method has much parsimonious estimated structures without sacrificing too much performance. 
As the number of redundant layers are unknown, \texttt{active\_penultimate\_warp} example also shows that a weakly penalized useful hidden layer can create a more difficult finite-sample prediction trade-off.

The \texttt{holder\_correct} experiment uses the Mat\'{e}rn-$3/2$ kernel in every DGP layer, with the same kernel family for the regular and augmented fits.  Table~\ref{tab:section5-eta-window} in Appendix~\ref{subsec:dim-ntrain-sweep} reports means and standard deviations over three seeds and two initialization modes, serving as a finite-sample diagnostic of the tuning-window discussion, illustrating the expected finite-sample trade-off.  Moving away from $\lambda=0$ sharply reduces redundant-layer identity and partial-composition errors.  The smallest mean RMSE occurs at $\lambda=0.4$, whereas the strongest representation stabilization on the displayed grid occurs at $\lambda=2.5$.  Thus stronger anchoring need not monotonically improve prediction, which is consistent with the role of the upper tuning constraint in Section~\ref{sec:theory}.

\paragraph{Comparison with a single-layer GP in high dimensions.}
We further compare the regular DGP and identity-augmented DGP with a single-layer squared-exponential Gaussian process (RBF GP; equivalently, an RBF- or Matern-$\infty$- kernel GP).  In this experiment, both DGPs use residual random-Fourier-feature approximations to the RBF kernel with the default lengthscale rule $\ell=\sqrt d$, while the single-layer comparator is an exact ARD (Automatic Relevance Determination) RBF GP with one optimized lengthscale per input coordinate.  To keep this comparison interpretable, we fix the observation-noise standard deviation at $\sigma=0.10$ and report the three sample sizes $n\in\{100,500,1000\}$.  We also keep one synthetic truth fixed within each ambient dimension $d\in\{1,2,5,10\}$ rather than changing the truth with $n$.  This synthetic truth is selected using the three sample sizes jointly: among truths admissible at a given $d$, we choose the one with the largest average RMSE advantage of the better DGP variant over the RBF GP; if no truth has a positive average advantage, as at $d=1$, we choose the truth for which the RBF GP's average advantage is smallest.

\begin{table}[t]
\centering
\caption{Fixed-noise comparison with a single-layer RBF GP at $\sigma=0.10$.  One synthetic truth is held fixed for each dimension and is evaluated at $n\in\{100,500,1000\}$.  RMSE entries are mean (standard deviation) over five paired seeds.  Smaller is better; bold indicates the smallest means and variances in each row.}
\label{tab:dgp-vs-gp-favorable}
\small
\resizebox{\textwidth}{!}{%
\begin{tabular}{lccccc}
\toprule
Truth & $d$ & $n$ & RBF GP & Reg. DGP & Aug. DGP \\
\midrule
\texttt{edge\_transition} & 1 & 100 & \textbf{0.055 (0.024)} & 0.062 (0.020) & 0.081 (0.021) \\
\texttt{edge\_transition} & 1 & 500 & 0.025 (0.008) & \textbf{0.023 (0.006)} & 0.051 (0.007) \\
\texttt{edge\_transition} & 1 & 1000 & \textbf{0.013 (0.002)} & 0.018 (0.005) & 0.032 (0.009) \\
\addlinespace
\texttt{localized\_bumps\_sparse} & 2 & 100 & 0.157 (0.027) & 0.148 (0.029) & \textbf{0.146 (0.030)} \\
\texttt{localized\_bumps\_sparse} & 2 & 500 & 0.126 (0.004) & \textbf{0.067 (0.010)} & 0.079 (0.010) \\
\texttt{localized\_bumps\_sparse} & 2 & 1000 & 0.160 (0.009) & \textbf{0.049 (0.003)} & 0.065 (0.005) \\
\addlinespace
\texttt{nonlinear\_ridge\_edge} & 5 & 100 & 0.164 (0.010) & \textbf{0.131 (0.005)} & \textbf{0.131 (0.005)} \\
\texttt{nonlinear\_ridge\_edge} & 5 & 500 & 0.122 (0.007) & 0.087 (0.008) & \textbf{0.087 (0.006)} \\
\texttt{nonlinear\_ridge\_edge} & 5 & 1000 & 0.116 (0.003) & {0.073 (0.006)} & \textbf{0.073 (0.003)} \\
\addlinespace
\texttt{nonlinear\_ridge\_multiscale} & 10 & 100 & 0.332 (0.029) & \textbf{0.285 (0.027)} & 0.286 (0.027) \\
\texttt{nonlinear\_ridge\_multiscale} & 10 & 500 & 0.217 (0.005) & \textbf{0.174 (0.002)} & 0.175 (0.002) \\
\texttt{nonlinear\_ridge\_multiscale} & 10 & 1000 & 0.196 (0.003) & 0.162 (0.003) & \textbf{0.162 (0.002)} \\\bottomrule
\end{tabular}}
\end{table}

\begin{table}[h!]
\centering
\caption{Real-data RMSE comparison. RMSE entries are mean (standard deviation) over 10 random splits. ``Aug.'' and ``Reg.'' denote augmented and regular DGP fits. Smaller is better; bold indicates the smallest mean in each row. And ``Aug.W'' denotes the number of times Aug. DGP wins Reg. DGP, RBF DGP and Linear regressions. }
\label{tab:uci-rmse}
\small
\begin{tabular}{lccccccc}
\hline
Data set & $n$ & $d$ & Aug. DGP & Reg. DGP & RBF GP & Linear & Aug.W \\
\hline
Auto MPG & 398 & 7 & $2.86 (0.41)$ & $2.98 (0.61)$ & $\mathbf{2.74 (0.50)}$ & $3.36 (0.43)$ & 4/10 \\
Parkinsons Motor & 5875 & 19 & $\mathbf{6.45 (0.31)}$ & $6.78 (0.43)$ & $8.14 (0.13)$ & $7.63 (0.18)$ & 7/10 \\
Air Quality NOx & 7718 & 8 & $\mathbf{88.76 (9.27)}$ & $89.70 (9.86)$ & $90.06 (10.70)$ & $106.88 (11.04)$ & 7/10 \\
Power Plant & 9568 & 4 & $\mathbf{4.19 (0.13)}$ & $4.25 (0.15)$ & $4.21 (0.14)$ & $4.54 (0.12)$ & 8/10 \\ 
\hline
\end{tabular}
\end{table}

Table~\ref{tab:dgp-vs-gp-favorable} separates two effects that are otherwise easy to conflate: the predictive benefit of the deep architecture and the additional effect of identity-based canonicalization.  For $d=5, 10$, both DGP variants outperform the RBF GP at all three sample sizes on \texttt{nonlinear\_ridge\_edge}; the reduction in mean RMSE is about $20\%$ at $n=100$ and about $37\%$ at $n=1000$.  The regular and augmented DGPs differ by less than $0.001$ in mean RMSE in every row of this block.  The same pattern appears at $d=10$ for \texttt{nonlinear\_ridge\_multiscale}: both deep models improve on the RBF GP throughout, while their mutual difference remains below $0.001$.  These rows suggest that when a genuinely compositional nonlinear representation is useful, the identity augmentation can be nearly prediction-neutral while preserving the gain from depth.

The $d=2$ \texttt{localized\_bumps\_sparse} rows show the prediction-parsimony trade-off more sharply.  At $n=1000$, the RBF GP has mean RMSE $0.160$, compared with $0.049$ for the regular DGP and $0.065$ for the augmented DGP.  Thus the augmentation gives up some of the unpenalized DGP's predictive flexibility, but it still retains a large advantage over the single-layer GP.  

The $d=1$ \texttt{edge\_transition} rows provide the complementary case.  Here the single-layer RBF GP is already well matched to the low-dimensional shallow target and has the smallest mean RMSE at $n=100$ and $n=1000$; the regular DGP is only slightly better at $n=500$, and the augmented DGP is worse in all three rows.  This prevents interpreting the augmentation as a universal prediction device.  

For the real-data comparison, the regular and augmented DGPs use the same RBF random-feature family (three residual hidden maps and an RBF random-feature output map), while the single-layer RBF GP uses a constant-scaled isotropic RBF covariance plus a fitted white-noise kernel; its covariance hyperparameters are optimized by marginal likelihood.
Table~\ref{tab:uci-rmse} suggests that the augmented DGP remains competitive on real data and can improve over the unpenalized DGP, especially on Parkinsons Motor, Air Quality NOx and Power Plant.
 
 \section{Conclusion}

\paragraph{Structural non-identifiability of DGP.}
DGPs are defined through compositions of latent GP layers, so the likelihood depends on the layer tuple only through the composed regression map. 
We show when the prior is invariant under an adjacent-layer symmetry, or when sufficiently many disjoint symmetry-related copies have comparable prior weights as in Theorem~\ref{thm:non-conv}, the posterior needs not concentrate on a prescribed layerwise representative.

The symmetry packing bound in Theorem~\ref{thm:non-conv} explains when accurate function estimation can coexist with diffuse layerwise inference. Our methodology supplies an explicit preference through a layerwise identity penalty. Theorem~\ref{thm:noniden break} quantifies its effect on posterior odds between symmetry-related representations.

\paragraph{Canonicalization a non-identifiable representation.}
We therefore use the augmentation to impose an explicit preference rather than to recover a uniquely likelihood-identified ``true'' decomposition. Theorem~\ref{thm:noniden break} gives the exact change in posterior mass between symmetry-related sets through the product of the prior Radon-Nikod\'ym factor and the exponential difference in identity-deviation penalties. Thus the novel augmentation identifies lower-penalty representatives  without sacrificing too much performance gain from DGP.

Under the prior-mass and proxy conditions in Section~\ref{sec:theory}, Theorem~\ref{thm:contraction} preserves function-level contraction in R\'enyi divergence. The optional $L^2(\mu_X)$ conclusion requires the additional bounded-output condition. Theorem~\ref{thm:parsimonious} further gives the concentration of redundant layers towards some
$n\epsilon_n^2/\eta_{j,n}$-neighborhood of the identity, which shrinks towards the identity under the stronger scaling $\eta_{j,n}\gg n\epsilon_n^2$. Hence,  the theory establishes a precise compatibility between function-level learning and identity-biased shrinkage of redundant layers.

\paragraph{Functional spaces for DGP.}

While it is generally agreed by the community that DGPs, in comparison to single GPs, can handle more complex functions, it is yet clear what exact function families can deep composition models like DGPs handle. 

Many practical problems, despite high-dimensional or complex
in nature, have intrinsic but hidden structures. Some of these structures include local structures \citep{luo2023sharded,luo2026efficient}
and intrinsic lower dimensions \citep{schmidt2020nonparametric}. 
The classical result is usually about the adaptation to smoothness, but we conjecture DGP is able to lift adaptation to a new level, managing not just smoothness, but a huge range of structures of the underlying data-generating function.

\vspace{1cm}

\noindent\textbf{Conflict of Interest.} There is no conflict of interest.

\noindent\textbf{Supplementary materials.} The supplementary materials contain proofs, posterior inference algorithms, synthetic data-generating functions, and additional numerical results.

\medskip
\noindent\textbf{Data availability statement.} 
The code and data are available at \url{https://github.com/hrluo/CanonicalDeepGaussianProcesses}.

\medskip
\noindent\textbf{Declaration of Generative AI Use.} The authors used OpenAI ChatGPT-6 Sol to assist with  \textcolor{black}{language polishing} and code development (including algorithm design and figure preparation) and debugging. The authors reviewed and revised the \textcolor{black}{AI-polished language}, verified the \textcolor{black}{code and numerical results} are suitable for academic publication. The authors take full responsibility for the originality, accuracy, and integrity of the article.

\begin{center}
{\large\bf DISCLOSURE STATEMENT}
\end{center}
HL was supported by U.S. Department of Energy under Contract DE-AC02-05CH11231
and U.S. National Science Foundation NSF Grant DMS-2412403, DMS-2606336, DMS-2616828.

\bibliographystyle{chicago}
\bibliography{yichen}

\newpage
\setcounter{page}{1}
\begin{center}
{\large\bf SUPPLEMENTARY MATERIAL}
\end{center}
{\bf Proofs and Additional Experiments:} Supplemental Material for ``Canonical Layer Inference for Deep Gaussian Processes via Symmetry-Breaking Augmentation''. \\

\appendix\label{sec:appendix}

\section{Preliminaries}
This section collects background facts on Gaussian processes small ball probabilities, metric conversion, and fractional posterior contraction. We cite standard results where possible and provide short derivations for the specialized forms used in the paper.

\subsection{Small Ball Probability}\label{sec:appen small ball}
The regularity of a Gaussian process $(Z(x): x\in\X)$ is often measured by the centered small ball probability
\begin{align*}
\mathbb{P}(\|Z\|_{\ell^\infty(\X)} <\epsilon).
\end{align*}
The decentered small ball probability of a process $Z$ with respect to a function $f:\X\to \mathbb{R}$ is defined as
\begin{align*}
\mathbb{P}(\|Z-f\|_{\ell^\infty(\X)} <\epsilon).
\end{align*}
It measures how much prior mass is concentrated around the function $f$, which is useful in guaranteeing the contraction behavior of Bayesian posterior, provided $f$ is the data-generating function. With $\mathcal{H}$ denoting the RKHS of the Gaussian process $Z$, one can relate the decentered small ball probability with the centered small ball probability as
\begin{equation}\label{eq:decenter RKHS}
\mathbb{P}(\|Z\|_{\ell^\infty(\X)} <\epsilon) e^{-\frac{1}{2}\|f\|_{\mc H}^2} \le \mathbb{P}(\|Z-f\|_{\ell^\infty(\X)} <\epsilon) \le \mathbb{P}(\|Z\|_{\ell^\infty(\X)} <\epsilon),
\end{equation}
where the first inequality and its proof can be found in Lemma 5.2 of \cite{vdVvZ2008}, and the second inequality comes from Anderson's Lemma.

\paragraph{Small ball probability for Mat\'{e}rn processes.}
Assume $Z^\tau$ is a Mat\'{e}rn process with regularity parameter $\alpha$ and (time) rescaling parameter $\tau$, in the sense that $Z^\tau_X := Z_{\tau X}$ for all $X$ belonging to the domain of the process $Z$. Then the centered small ball probability of $Z^\tau$ can be lower bounded as (Lemma 11.36 of \cite{GhosalVanDerVaart2017})
\begin{equation}\label{eq:matern centered small ball}
    \ln\mathbb{P}(\|Z^\tau\|_{\ell^\infty(\tilde{\mu})} <\epsilon) \gtrsim - \tau^d \epsilon^{-d/\alpha}.
\end{equation}
For $f\in C^\beta(\tilde{\X})_1$ a H\"{o}lder-$\beta$ function on the support of $\tilde{\mu}$ and $\mathcal{H}^\tau$ the RKHS of process $Z^\tau$, we have
\begin{equation}\label{eq:matern decentered small ball}
    -\ln\mathbb{P}(\|Z^\tau-f\|_{\ell^\infty(\tilde{\mu})} <\epsilon) \lesssim  \tau^d \epsilon^{-d/\alpha} + \tau^{-2\alpha} \epsilon^{-\frac{2(\alpha-\beta)+d}{\beta}}.
\end{equation}

\paragraph{Small ball probability for DGPs.}
For the DGPs defined in \eqref{eq:dgp}, for each non-final layer $j<J$, for $\bar f_j = (\bar f_j [1], \bar f_j [2], \cdots, \bar f_j [d])^T\in\mathcal{F}_j$, we consider the $d$-dimensional version of the small ball probability given by
\begin{align}\label{eq:multivaraite small ball}
    \Pi(\|f_j-\bar f_j \|_{\ell^\infty(\tilde{\mu})}<\epsilon)
    & = \prod_{k=1}^d \Pi(\|f_j[k]-\bar f_j [k]\|_{\ell^\infty(\tilde{\mu})}<\epsilon) \nonumber \\
    & = \prod_{k=1}^d \Pi(\|Z_{j,k}-\bar f_j [k]\|_{\ell^\infty(\tilde{\mu})}<\epsilon) \nonumber \\
    & \ge \prod_{k=1}^d \Pi(\|Z_{j,k}\|_{\ell^\infty(\tilde{\mu})}<\epsilon) e^{-\frac{1}{2}\|-\bar f_j [k]\|_{\mc H_j}^2}.
\end{align}

In the Bayesian nonparametric analysis of DGP, the small ball probability of each individual layer $j$, if properly controlled, can lead to the small ball probability of the whole DGP, detailed in the following lemma.

\begin{lemma}\label{lem:DGP small ball}
    Let $\bar f_j\in\mathcal F_j$, $1\le j\le J$, and let
    $\epsilon_j>0$. Let $\X_j\subset\mathbb R^d$ be domains with
    $\X_1=\tilde\X$. Assume that, for every $1\le j\le J-1$,
    \[
    \epsilon_{j+1}\ge
    \sup_{\substack{x,x'\in\X_{j+1}\\ \|x-x'\|_\infty<2\epsilon_j}}
    \|\bar f_{j+1}(x)-\bar f_{j+1}(x')\|_\infty,
    \]
    and that for every $y\in\bar f_j(\X_j)$ and
    $u\in[-2\epsilon_j,2\epsilon_j]^d$ one has $y+u\in\X_{j+1}$.
    Then
    \[
    \ln \Pi\!\left(
    \|f_J\circ\cdots\circ f_1-\bar f_J\circ\cdots\circ\bar f_1\|_{\ell^\infty(\tilde\mu)}
    <2\epsilon_J\right)
    \ge
    \sum_{j=1}^J
    \ln \Pi\!\left(\|f_j-\bar f_j\|_{\ell^\infty(\X_j)}<\epsilon_j\right).
    \]
\end{lemma}

\begin{proof}
Let
\[
B_j=\{f_j:\|f_j-\bar f_j\|_{\ell^\infty(\X_j)}<\epsilon_j\}.
\]
For $x\in\tilde\X$, write
$G_j=f_j\circ\cdots\circ f_1$ and
$\bar G_j=\bar f_j\circ\cdots\circ\bar f_1$.
We prove inductively that $G_j(x),\bar G_j(x)\in\X_{j+1}$ for $j<J$ and
$\|G_j-\bar G_j\|_{\ell^\infty(\tilde\mu)}<2\epsilon_j$.
For $j=1$, the approximation error is at most $\epsilon_1$, and the domain condition places both images in $\X_2$.
If the claim holds at $j$, then
\begin{align*}
\|G_{j+1}(x)-\bar G_{j+1}(x)\|_\infty
&\le
\|(f_{j+1}-\bar f_{j+1})(G_j(x))\|_\infty\\
&\quad+
\|\bar f_{j+1}(G_j(x))-\bar f_{j+1}(\bar G_j(x))\|_\infty\\
&<\epsilon_{j+1}+\epsilon_{j+1}
=2\epsilon_{j+1}.
\end{align*}
The domain assumption then permits the next step. Hence
$\cap_{j=1}^J B_j$ is contained in the displayed composition ball.
Independence of the layer priors gives the claimed probability bound.
\end{proof}

\subsection{Conversion of Metrics}
We collect the metric calculations used in the proofs. The R\'{e}nyi divergence was defined in Section~\ref{sec:theory}; see \cite{bhattacharya2019bayesian,castillo2025deep} for background.

Recalling the density $p_f^n$ defined in \eqref{eq:pf}, further define
\[
K_n(f^*,f)
=
\int p_{f^*}^n
\ln\frac{p_{f^*}^n}{p_f^n}\,d\nu_n
\]
and
\[
V_n(f^*,f)
=
\int p_{f^*}^n
\left(
\ln\frac{p_{f^*}^n}{p_f^n}-K_n(f^*,f)
\right)^2d\nu_n .
\]
\begin{lemma}[KL divergence and $\ell^2$ norm]\label{lem:kl-lp}
For all $f,f^*$ defined on $\X$ with $\|f-f^*\|_\infty\le C$,
\[
K_n(f^*,f)
=
\frac{n}{2\sigma^2}\|f-f^*\|_{\ell^2(\mu_X)}^2
\]
and
\[
V_n(f^*,f)
\le
n\left(\frac{1}{\sigma^2}+\frac{C^2}{4\sigma^4}\right)
\|f-f^*\|_{\ell^2(\mu_X)}^2.
\]
\end{lemma}
\begin{proof}
Let $\delta(x)=f(x)-f^*(x)$. Under $f^*$, for one observation,
\[
\ln\frac{p_f(Y\mid X)}{p_{f^*}(Y\mid X)}
=
\frac{\delta(X)\varepsilon}{\sigma^2}
-\frac{\delta(X)^2}{2\sigma^2},
\qquad
\varepsilon\sim N(0,\sigma^2).
\]
Taking expectations gives the KL identity. The variance of the one-observation log-likelihood ratio is bounded by
\[
\frac{\mathbb E\delta(X)^2}{\sigma^2}
+
\frac{\mathbb E\delta(X)^4}{4\sigma^4}
\le
\left(\frac{1}{\sigma^2}+\frac{C^2}{4\sigma^4}\right)
\|\delta\|_{\ell^2(\mu_X)}^2.
\]
Independence gives the stated $n$-sample variance bound.
\end{proof}

For $\rho\in(0,1)$, recall the definition of $\rho$-R\'{e}nyi divergence as in \eqref{eq:renyi}. The following lemma gives the lower bound needed to convert R\'{e}nyi contraction into $L^2(\mu_X)$ contraction on uniformly bounded sets.

\begin{lemma}[$\rho$-R\'{e}nyi divergence and $\ell^2$ norm]\label{lem:renyi-l2}
Fix $0<\rho<1$. If $\|f-f^*\|_{\ell^\infty(\mu_X)}\le C$ for some finite constant $C$, then there is a constant $c_{\rho,\sigma,C}>0$ such that
\begin{align*}
\frac{1}{n}D_\rho(p_f^n,p_{f^*}^n)
\ge
c_{\rho,\sigma,C}\|f-f^*\|_{\ell^2(\mu_X)}^2 .
\end{align*}
\end{lemma}
\begin{proof}
For fixed $x$, a direct Gaussian calculation gives
\begin{align*}
    \int p_f(y\mid x)^\rho p_{f^*}(y\mid x)^{1-\rho}\,dy
    =
    \exp\{-a(f(x)-f^*(x))^2\},
    \qquad
    a:=\frac{\rho(1-\rho)}{2\sigma^2}.
\end{align*}
With random design, independence gives
\begin{align*}
    \int (p_f^n)^\rho(p_{f^*}^n)^{1-\rho}\,d\nu_n
    =
    \left[\int_\X \exp\{-a(f(x)-f^*(x))^2\}\,d\mu_X(x)\right]^n.
\end{align*}
Let $u(x)=a(f(x)-f^*(x))^2$. Since $\|f-f^*\|_{\ell^\infty(\mu_X)}\le C$, we have $0\le u(x)\le U:=aC^2$. The elementary inequality $1-e^{-u}\ge e^{-U}u$ for $0\le u\le U$ yields
\begin{align*}
    \int_\X e^{-u(x)}\,d\mu_X(x)
    &\le
    1-e^{-U}\int_\X u(x)\,d\mu_X(x) \\
    &=
    1-e^{-U}a\|f-f^*\|_{\ell^2(\mu_X)}^2.
\end{align*}
Using $-\ln(1-t)\ge t$ for $t\in[0,1)$, and the definition of $D_\rho$ with $\rho-1<0$, we obtain
\begin{align*}
    D_\rho(p_f^n,p_{f^*}^n)
    &=
    \frac{n}{\rho-1}\ln\left[\int_\X e^{-u(x)}\,d\mu_X(x)\right] \\
    &\ge
    \frac{n e^{-U}a}{1-\rho}\|f-f^*\|_{\ell^2(\mu_X)}^2.
\end{align*}
The displayed bound proves the claim.
\end{proof}

\subsection{Fractional Posterior Contraction}
The Bayesian fractional posterior contraction was established in \cite{bhattacharya2019bayesian}. We tailor their results for our context and notations as follows.
\begin{lemma}[Augmented fractional posterior contraction]\label{lem:frac contract}
Suppose $\|f^*\|_{\ell^\infty(\X)}<\infty$. For $\rho\in(0,1)$ and the augmented prior $\tilde{\Pi}$, if there exists $\epsilon_n\to 0$ with $n\epsilon_n^2 \to \infty$ and an absolute constant $C$, such that
\begin{equation}\label{eq:prior mass l2}
    \ln \tilde{\Pi}(\|\mathcal{C}(\theta)-f^*\|_{\ell^2(\mu_X)}\le \epsilon_n, \; \|f\|_{\ell^\infty(\mu_X)}\le C) \gtrsim -n\epsilon_n^2,
\end{equation}
then there exists a sufficiently large constant $M>0$ such that
\begin{align*}
	\tilde{\Pi}_\rho\!\left(
	D_\rho\!\left(p_{\mathcal C(\theta)}^n, p_{f^*}^n\right)>M n\epsilon_n^2
	\,\middle|\,\D_n\right)
	\overset{\mathbb P_{f^*}^n}{\longrightarrow}0 .
\end{align*}
If, in addition, Assumption~\ref{assump:bounded} holds, then for sufficiently large $M>0$,
\begin{equation}\label{eq:posterior contract l2}
    \tilde{\Pi}_\rho(\|\mathcal{C}(\theta)-f^*\|_{\ell^2(\mu_X)} > M\epsilon_n\mid\D_n) \overset{\mathbb P_{f^*}^n}{\longrightarrow} 0.
\end{equation}
\end{lemma}

\begin{proof}
By Lemma \ref{lem:kl-lp}, the prior mass condition \eqref{eq:prior mass l2} implies the corresponding KL-type prior mass condition. Theorem 3.1 of \cite{bhattacharya2019bayesian} then yields contraction of the fractional posterior in $\rho$-R\'{e}nyi divergence at rate $\epsilon_n$. If we further have the bounded condition as in Assumption \ref{assump:bounded}, Lemma \ref{lem:renyi-l2} converts the R\'{e}nyi divergence contraction to the stated $\ell^2(\mu_X)$ contraction.
\end{proof}

\section{Proofs for Section \ref{sec:canonical composition} and \ref{sec:asym-penalty}}\label{sec:proof non}
\subsection{Proof of Proposition \ref{thm:non-id-general}}
\begin{proof}[Proof of Proposition \ref{thm:non-id-general}]
By the definition of the adjacent layer symmetry map $\Phi_{j,S}$, we have $\mc C(\Phi_{j,S}(\theta)) = \mc C (\theta), \;\forall \theta\in \Theta$. Let
$Z_{\rho,n}=\int_\Theta[p(\D_n|\mc C(\vartheta))]^\rho\Pi(d\vartheta)$.
Hence, $p(\D_n|\Phi_{j,S}(\theta)) = p(\D_n|\theta)$ and
\begin{align*}
\Pi_\rho(\Phi_{j,S}(B)|\D_n) 
= & \frac{1}{Z_{\rho,n}}\int_B [p(\D_n|\Phi_{j,S}(\theta))]^\rho \Pi^{\Phi_{j,S}}(d\theta) \\
= & \frac{1}{Z_{\rho,n}}\int_B [p(\D_n|\theta)]^\rho L_{\Phi_{j,S}}(\theta)\Pi(d\theta) \\
= & \int_B L_{\Phi_{j,S}}(\theta) \Pi_\rho(d\theta|\D_n).
\end{align*}
This proves the Radon-Nikod\'ym identity. If $\Pi^{\Phi_{j,S}}=\Pi$, then $L_{\Phi_{j,S}}=1$ $\Pi$-almost surely, and the invariance statement follows.
\end{proof}

\subsection{Proof of Example \ref{eg:translation 2}}
We use the Cameron-Martin theorem for Gaussian measures on a separable Banach space; see, for example, \citet[Theorem~2.4.5]{bogachev1998gaussian}. If $Z$ is centered with Cameron-Martin space $\mc H$ and $h\in\mc H$, then
\[
\frac{d\mathcal L(Z+h)}{d\mathcal L(Z)}(g)
=
\exp\left\{\langle h,g\rangle_{\mc H}^{\sim}-\frac12\|h\|_{\mc H}^2\right\},
\]
where $\langle h,g\rangle_{\mc H}^\sim$ denotes the Paley–Wiener extension of the RKHS inner product. 
Because the convention in \eqref{eq:transported-prior} is
$\Pi^\Phi(B)=\Pi(\Phi(B))$, the first-layer transported law under
$f_1\mapsto \sharp_a f_1$ is the law of $Z_{1,k}-a_k\mathbbm{1}_{\X}$ in coordinate $k$. Hence
\[
\frac{d\mathcal L(Z_{1,k}-a_k\mathbbm{1}_{\X})}
     {d\mathcal L(Z_{1,k})}(g_k)
=
\exp\left\{
-a_k\langle \mathbbm{1}_{\X},g_k\rangle_{\mc H_1}^{\sim}
-\frac{a_k^2}{2}\|\mathbbm{1}_{\X}\|_{\mc H_1}^2
\right\}.
\]
Independence across coordinates gives the product formula in the main text. For the second layer, $f_2\mapsto \flat_{-a} f_2$ is an input translation. Stationarity makes its law invariant under this input translation on the common function space. Combining the two factors proves the displayed derivative.

\subsection{Proof of Theorem \ref{thm:non-conv}}
\begin{proof}[Proof of Theorem \ref{thm:non-conv}]
	Let $B^*:=B_{\theta^*}(\epsilon,d_{\Theta})$. By condition (i) of the theorem statement, the sets $\Phi_{j,S_l}(B^*)$, $l=1,\ldots,N_\epsilon$, are pairwise disjoint. Hence
	\begin{equation}\label{eq:cond1 imp}
		1
		\ge
		\sum_{l=1}^{N_\epsilon}\Pi_\rho(\Phi_{j,S_l}(B^*)\mid\D_n)
	\end{equation}
	By condition (ii) of the theorem statement and Proposition \ref{thm:non-id-general}, we have $\forall 1\le l\le N_\epsilon$,
	\begin{equation}\label{eq:cond2 imp}
		\Pi_\rho(\Phi_{j,S_l}(B^*)) = \int_{B^*} L_{\Phi_{j,S_l}}(\theta) \Pi_\rho(d\theta|\D_n) \ge \inf_{\theta\in B^*} L_{\Phi_{j,S_l}}(\theta) \int_{B^*} \Pi_\rho(d\theta|\D_n) \ge c_1 \Pi_\rho(B^*).
	\end{equation}
	Combining equations \eqref{eq:cond1 imp} and \eqref{eq:cond2 imp}, we have
	\begin{align*}
		\Pi_\rho(B^*\mid\D_n)
		\le (c_1N_\epsilon)^{-1}.
	\end{align*}
\end{proof}

\subsection{Proof of Corollary \ref{cor:non-conv matern}}\label{sec:cor31-review}
\par\medskip

\begin{proof}%
Let the metric $d_\Theta$ be the $\ell^p$ norm defined in equation \eqref{eq:theta norm} and denote 
$$B^*_- = \{\theta\in \Theta: \|\theta - \theta^*\|_{\ell^p(\tilde{\mu})}< \epsilon, \;\langle f_1, \mathbbm{1}_{\X} \rangle_{\mathcal{H}_1} <0 \},$$
$$B^*_+ = \{\theta\in \Theta: \|\theta - \theta^*\|_{\ell^p(\tilde{\mu})}< \epsilon, \; \langle f_1, \mathbbm{1}_{\X} \rangle_{\mathcal{H}_1}>0 \}
,$$
and $B^* = B_-^* \cup B_+^*$, which is an $\epsilon$ radius ball around $\theta^*$ in $\ell^p(\tilde \mu)$ metric. 
Let $m = \lfloor 1/(2\epsilon) \rfloor$.
Define sets
$$
\mathcal{A}_- = \Big\{a=(a_1,0,\cdots,0)\in\mathbb{R}^d: \;a_1\in \{-2\epsilon, -4\epsilon, \cdots, -2m\epsilon\} \Big\},
$$
$$
\mathcal{A}_+ = \Big\{a=(a_1,0,\cdots,0)\in\mathbb{R}^d: \;a_1\in \{2\epsilon, 4\epsilon, \cdots, 2m\epsilon\} \Big\},
$$
and $\mathcal{A} = \mathcal{A}_- \cup \mathcal{A}_+$.
In other words, $\mathcal{A}$ consists of vectors $a$ whose first coordinate is a multiplier of $2\epsilon$, while all other coordinates are zero. 

We now apply Theorem \ref{thm:non-conv} separately to sets $B^*_-$ and $B^*_+$. We first consider set $B^*_-$. Consider a collection of adjacent layer symmetry maps as the translation map in Example \ref{eq:translation map}, with $a\in\mc A_+$ as the amount of translation, denoted as
$$\Phi_{1,a}(f_1,f_2,\ldots,f_J)
:=(\sharp_a f_1,\; \flat_{-a} f_2,\; f_3,\ldots,f_J), \; a\in\mathcal{A}_+.$$
Noticing $|\mc A_+| = m \asymp 1/\epsilon$, we proceeds to verify the conditions (i) and (ii) with the maps $\Phi_{1,a}, \;\forall a\in\mathcal{A}_+$.

For condition (i), noticing for any two distinct elements $a, a'\in\mc A_+$, $|a_1-a'_1|\ge 2\epsilon$ and it is impossible to find $f_1\in\mathcal{F}_1$ satisfying $\|f_1+a - f_1^*\|_{\ell^p(\tilde{\mu})}<\epsilon$ and $\|f_1+a' - f_1^*\|_{\ell^p(\tilde{\mu})}<\epsilon$ simultaneously. Thus the sets $\Phi_{1,a}(B^*_-)$ and $\Phi_{1,a'}(B^*_-)$ are disjoint.

For condition (ii), recalling Example 
\ref{eg:translation 2}, for all $\theta \in B^*_-$, 
\begin{equation}\label{eq:lip mat 1}
L_{\Phi_{1,a}}(\theta) = \exp\Big(-a_1\langle f_1, \mathbbm{1}_{\X} \rangle_{\mc H_1} - \frac{a_1^2}{2} \|\mathbbm{1}_{\X} \|_{\mc H_1}^2 \Big).
\end{equation}
By the statement of the corollary, $|\mathbbm{1}_{\X} \|_{\mc H_1}^2 \le c$. By the definition of set $\mc A+$ and $B^*_-$, we have $a_1>0$ and $\langle f_1, \mathbbm{1}_{\X} \rangle_{\mc H_1} <0$, thus $a_1\langle f_1, \mathbbm{1}_{\X} \rangle_{\mc H_1} <0$. Therefore, 
$$
L_{\Phi_{1,a}}(\theta) \ge \exp(0 - c a_1^2) \ge \exp(-c), \quad\forall a\in\mc A_+, \;\theta\in B^*_-,
$$
where the second inequality is due to $|a_1|\le 1$. This verifies condition (ii). In view of Theorem \ref{thm:non-conv} and its proof, the conclusion of the this theorem can be applied to the set $B^*_-$, yielding
$$\Pi_\rho(B^*_-|\D_n) \lesssim 1/m \asymp \epsilon.$$
Applying exactly the same derivation on the set $B^*_+$ with translation maps $\Phi_{1,a}, \; a\in\mc A_-$, one concludes
$$\Pi_\rho(B^*_+|\D_n) \lesssim 1/m \asymp \epsilon.$$
Noticing $\{\theta: \|\theta-\theta^*\|_{\ell^p(\tilde\mu)}\le\epsilon\} = B^* = B^*_-\cup B^*_+$, we finish the proof.
\end{proof}

\subsection{Proof of Theorem \ref{thm:noniden break}}
\begin{proof}
Let $C_n = \int_{\Theta} [p(\D_n|\theta)]^\rho \prod_{j=1}^{J-1}R_j(f_j)\Pi(d\theta)$ denote the normalizing constant of the augmented fractional posterior. 
Write $\theta = (f_1, f_2, \cdots, f_J)$, and let $\theta' = \Phi_{j,S}(\theta)$, 
so that $\theta' = (f'_1, f'_2, \cdots, f'_J)$ with $f'_{j'} = f_{j'}$, for all $j'< j$ or $f'>j+1$, $\mc C(\theta')=\mc C(\theta)$, and $\Pi^{\Phi_{j,S}}(d\theta) = \Pi(\Phi_{j,S}(d\theta)) = \Pi(d\theta')$.
Then we have
\begin{align*}
\tilde\Pi_\rho(\Phi_{j,S}(B)\mid\D_n)
&=
\frac{1}{C_n}\int_{\Phi_{j,S}(B)}
[p(\D_n\mid\mc C(\theta'))]^\rho
R_j(f'_j) R_{j+1}(f'_{j+1})\prod_{\substack{1\le j'<J, \\j'\ne j, j'\ne j+1}}R_{j'}(f'_{j'}) \;
\Pi(d\theta')\\
&\overset{\theta'=\Phi_{j,S}(\theta)}{=}
\frac{1}{C_n}\int_B
[p(\D_n\mid\mc C(\theta))]^\rho
R_j(S\circ f_j) R_{j+1}(f_{j+1}\circ S^{-1})\prod_{\substack{1\le j'<J, \\j'\ne j, j'\ne j+1}}R_{j'}(f_{j'}) \;
\,\Pi^{\Phi_{j',S}}(d\theta)\\
&=
\frac{1}{C_n}\int_B
[p(\D_n\mid\mc C(\theta))]^\rho
e^{-\Delta_{j,S,n}(\theta)} \prod_{1\le j' < J}R_{j'}(f_{j'}) \;
\,\frac{\Pi^{\Phi_{'j,S}}(d\theta)}{\Pi(d\theta)} \Pi(d\theta)\\
&=
\frac{1}{C_n} \int_B
L_{\Phi_{j,S}}(\theta)e^{-\Delta_{j,S,n}(\theta)}
\left\{[p(\D_n\mid\mc C(\theta))]^\rho \prod_{1\le j' < J}R_{j'}(f_{j'}) \Pi_\rho(d\theta\mid\D_n) \right\}\\
&=
\int_B
L_{\Phi_{j,S}}(\theta)e^{-\Delta_{j,S,n}(\theta)}
\,\tilde\Pi_\rho(d\theta\mid\D_n),
\end{align*}
where the last equality is exactly the definition in \eqref{eq:zero energy}. If
$\tilde\Pi_\rho(B\mid\D_n)>0$, dividing by this quantity shows that the posterior-mass ratio is the conditional average over $B$ of
$L_{\Phi_{j,S}}(\theta)e^{-\Delta_{j,S,n}(\theta)}$. The assumed bounds
$m\le L_{\Phi_{j,S}}\le M$ give the two displayed inequalities, and the two asymptotic conclusions follow immediately.
\end{proof}

\section{Proofs of Theorems \ref{thm:contraction} and \ref{thm:parsimonious}}

Throughout this section, write
\begin{align*}
    A_n(\theta)
    :=
    \sum_{j=1}^{J-1}\eta_{j,n}
    \|f_j-\id\|_{\ell^p(\tilde\mu)}
\end{align*}
then the normalized augmented prior can be written as
\begin{align*}
    \tilde \Pi(d\theta)
    :=
    \frac{\exp\{-A_n(\theta)\}\Pi(d\theta)}
    {\int_\Theta \exp\{-A_n(\theta')\}\Pi(d\theta')}.
\end{align*}
The augmented fractional posterior can then be written as the ordinary
\(\rho\)-fractional posterior based on the prior \(\tilde\Pi\):
\begin{align*}
    \tilde\Pi_\rho(d\theta\mid\D_n)
    =
    \frac{[p(\D_n\mid \mc C(\theta))]^\rho\,\tilde\Pi(d\theta)}
    {\int_\Theta [p(\D_n\mid \mc C(\theta'))]^\rho
    \,\tilde\Pi(d\theta')}.
\end{align*}
We shall use this representation repeatedly.

We also use the following consequence of the enlarged-domain convention. As defined in Section \ref{sec:tech prelim}, $\X$ lies in the interior of $\tilde\X$. 
Thus, for \(n\) sufficiently large, the \(\tilde\epsilon_n\)-neighborhood of
\(\X\) is contained in \(\tilde \X\), namely
\begin{align*}
    \{z\in\mathbb R^d:\inf_{x\in\X}\|z-x\|_\infty
    \le \tilde\epsilon_n\}
    \subset \tilde \X .
\end{align*}
This result is used to evaluate later layers on the output of near-identity earlier layers.

\begin{lemma}[Composition of near-identity layers]
\label{lem:identity-block-composition}
Let \(1\le m<J\), and suppose that
\begin{align*}
    \|f_j-\id\|_{\ell^\infty(\tilde\X)}
    \le \delta ,
    \qquad j=1,\ldots,m,
\end{align*}
where \(m\delta\le \tilde\epsilon_n\). Then, for \(n\) sufficiently large,
\begin{align*}
    \|f_m\circ\cdots\circ f_1-\id\|_{\ell^\infty(\mu_X)}
    \le m\delta .
\end{align*}
In particular, if \(\delta=\tilde\epsilon_n/J_0\) and \(m=J_0-1\), then
\begin{align*}
    \|f_{J_0-1}\circ\cdots\circ f_1-\id\|_{\ell^\infty(\mu_X)}
    \le \tilde\epsilon_n .
\end{align*}
\end{lemma}

\begin{proof}
Let \(G_0=\id\) and \(G_j=f_j\circ G_{j-1}\). We prove by induction
that \(G_j(\X)\subset \tilde\X\) and
\begin{align*}
    \|G_j-\id\|_{\ell^\infty(\mu_X)}
    \le j\delta ,
    \qquad j=1,\ldots,m.
\end{align*}
For \(j=1\), this follows directly from
\(\|f_1-\id\|_{\ell^\infty(\tilde\mu)}\le\delta\) and
\(\X\subset\tilde\X\). Assume the claim holds for \(j-1\). Since
\((j-1)\delta\le m\delta\le\tilde\epsilon_n\), the enlarged-domain
condition implies \(G_{j-1}(\X)\subset\tilde\X\). Therefore
\begin{align*}
\begin{aligned}
    \|G_j-\id\|_{\ell^\infty(\mu_X)}
    &=
    \|f_j\circ G_{j-1}-\id\|_{\ell^\infty(\mu_X)}  \\
    &\le
    \|(f_j-\id)\circ G_{j-1}\|_{\ell^\infty(\mu_X)}
    +
    \|G_{j-1}-\id\|_{\ell^\infty(\mu_X)} \\
    &\le
    \delta+(j-1)\delta
    =
    j\delta .
\end{aligned}
\end{align*}
The induction is complete.
\end{proof}

\begin{lemma}[Augmented prior mass near identity]
\label{lem:augmented-identity-prior-mass}
Suppose Assumption \ref{assump:layer regular} holds and
\begin{align*}
    \eta_{j,n}\tilde\epsilon_n\lesssim n\epsilon_n^2,
    \qquad
    \tilde\epsilon_n^{-d/\alpha}\lesssim n\epsilon_n^2 .
\end{align*}
Then, for every redundant layer \(j<J_0\),
\begin{align*}
    \ln \tilde\Pi
    \left(
        \|f_j-\id\|_{\ell^\infty(\tilde\mu)}
        \le \tilde\epsilon_n/J_0
    \right)
    \gtrsim
    -n\epsilon_n^2,
\end{align*}
where
\begin{align*}
    \tilde\Pi(df_j)
    :=
    \frac{
        \exp\{-\eta_{j,n}\|f_j-\id\|_{\ell^p(\tilde\mu)}\}
        \Pi(df_j)
    }{
        \int
        \exp\{-\eta_{j,n}\|g-\id\|_{\ell^p(\tilde\mu)}\}
        \Pi(dg)
    }
\end{align*}
is the \(j\)th marginal of the augmented prior.
\end{lemma}

\begin{proof}
Let
\begin{align*}
    B_{j,n}
    :=
    \left\{
        f_j:
        \|f_j-\id\|_{\ell^\infty(\tilde\mu)}
        \le \tilde\epsilon_n/J_0
    \right\}.
\end{align*}
Since \(0<R_j(f_j)\le 1\), the normalizing denominator of
\(\tilde\Pi\) is at most one. Hence
\begin{align*}
\begin{aligned}
    \tilde\Pi(B_{j,n})
    &\ge
    \int_{B_{j,n}}
    \exp\{-\eta_{j,n}\|f_j-\id\|_{\ell^p(\tilde\mu)}\}
    \Pi(df_j) \\
    &\ge
    \exp\{-d^{1/p}\eta_{j,n}\tilde\epsilon_n/J_0\}
    \Pi(B_{j,n}),
\end{aligned}
\end{align*}
because $\|g\|_{\ell^p(\tilde\mu)}\le d^{1/p}\|g\|_{\ell^\infty(\tilde\mu)}$, with $d^{1/\infty}=1$. 
By the small ball probability in Assumption
\ref{assump:layer regular},
\begin{align*}
    \ln \Pi(B_{j,n})
    \gtrsim
    -(\tilde\epsilon_n/J_0)^{-d/\alpha}.
\end{align*}
Therefore
\begin{align*}
    \ln \tilde\Pi(B_{j,n})
    \gtrsim
    -\eta_{j,n}\tilde\epsilon_n
    -\tilde\epsilon_n^{-d/\alpha}
    \gtrsim
    -n\epsilon_n^2 .
\end{align*}
\end{proof}

\begin{lemma}[Augmented prior mass of the active block]
\label{lem:augmented-active-block-mass}
Consider the set
\begin{align*}
	A_{J_0}
	:=
	\Big\{(f_{J_0},\ldots,f_J):
	&\|f_J\circ\cdots\circ f_{J_0}-f_n^*\|_{\ell^\infty(\tilde\mu)}
	\le \epsilon_n, \\
	&\eta_{j,n}\|f_j-\id\|_{\ell^p(\tilde\mu)}
	\le n\epsilon_n^2,\quad J_0\le j\le J-1
	\Big\}.
\end{align*}
If Assumption \ref{assump:prior mass} holds in the sense that $\ln \Pi(A_{J_0}) \gtrsim -n\epsilon_n^2$, 
then the corresponding augmented-prior mass of the active block is bounded
below by
\begin{align*}
    \ln \tilde\Pi(A_{J_0})
    \gtrsim
    -n\epsilon_n^2.
\end{align*}
\end{lemma}

\begin{proof}
For all $j\ge J_0$, the augmentation term $R_j(f_j)$ is at most one. On $A_{J_0}$, each of the $J-J_0$ penalized active layers satisfies $\eta_{j,n}\|f_j-\id\|_{\ell^p(\tilde\mu)}\le n\epsilon_n^2$. Hence,
\begin{align*}
	\begin{aligned}
		\widetilde\Pi (A_{J_0})
		& = \frac{\int_{A_{J_0}}
			\prod_{j=J_0}^{J-1}
			\exp\{-\eta_{j,n}\|f_j-\id\|_{\ell^p(\tilde\mu)}\}
			\Pi(d f_{J_0})\cdots \Pi(df_J)}{\int_{\mc F_{J_0}\times \cdots \times \mc F_{J}}
			\prod_{j=J_0}^{J-1}
			\exp\{-\eta_{j,n}\|f_j-\id\|_{\ell^p(\tilde\mu)}\}
			\Pi(d f_{J_0})\cdots \Pi(df_J)}    \\
		& \ge \frac{\int_{A_{J_0}}
			\prod_{j=J_0}^{J-1}
			\exp(-  n\epsilon_n^2)
			\Pi(d f_{J_0})\cdots \Pi(df_J)}{\int_{\mc F_{J_0}\times \cdots \times \mc F_{J}}
			\prod_{j=J_0}^{J-1}
			\Pi(d f_{J_0})\cdots \Pi(df_J)}  \\
		&\ge
		\exp\{-(J-J_0)n\epsilon_n^2\}
		\Pi(A_{J_0}),
	\end{aligned}
\end{align*}
which, by Assumption~\ref{assump:prior mass}, is at least $\exp\{-C'n\epsilon_n^2\}$ for some constant $C'$ independent of sample size $n$. Taking logarithms gives the claim.
\end{proof}

\begin{proof}[Proof of Theorem \ref{thm:contraction}]
We first verify a Kullback-Leibler and centered log-likelihood-variance prior mass condition for the normalized augmented prior $\tilde\Pi$; this yields the R\'{e}nyi divergence conclusion directly. Assumption~\ref{assump:bounded} is used only in the final metric-conversion step. By Assumption \ref{assump:proxy regular},
\begin{align*}
    \|f_n^*-f^*\|_{\ell^2(\mu_X)}
    \le \epsilon_n .
\end{align*}
We construct a set of augmented prior mass at least
\(\exp\{-Cn\epsilon_n^2\}\) on which
\(\|\mc C(\theta)-f_n^*\|_{\ell^\infty(\mu_X)}\lesssim\epsilon_n\).
Define
\begin{align*}
\begin{aligned}
    A
    :=
    \Big\{
    \theta\in\Theta:\;
    &\|f_j-\id\|_{\ell^\infty(\tilde\mu)}
    \le \tilde\epsilon_n/J_0,
    \; j< J_0; \; (f_{J_0},\ldots,f_J)\in A_{J_0}
    \Big\},
\end{aligned}
\end{align*}
where \(A_{J_0}\) is the active block event from Lemma
\ref{lem:augmented-active-block-mass}.

First we show that \(A\) is contained in an appropriate
\(\ell^2(\mu_X)\)-ball around \(f^*\). Let $G_{J_0-1}=\id$ when $J_0=1$, and otherwise let
\begin{align*}
    G_{J_0-1}
    :=
    f_{J_0-1}\circ\cdots\circ f_1 .
\end{align*}
By Lemma \ref{lem:identity-block-composition} when $J_0>1$, and trivially when $J_0=1$,
\begin{align*}
    \|G_{J_0-1}-\id\|_{\ell^\infty(\mu_X)}
    \le \tilde\epsilon_n .
\end{align*}
For \(\theta\in A\), the image \(G_{J_0-1}(\X)\) is contained in
\(\tilde\X\), and therefore
\begin{align*}
\begin{aligned}
    \|\mc C(\theta)-f_n^*\|_{\ell^\infty(\mu_X)}
    &=
    \|
    f_J\circ\cdots\circ f_{J_0}\circ G_{J_0-1}
    -f_n^*
    \|_{\ell^\infty(\mu_X)} \\
    &\le
    \|
    (f_J\circ\cdots\circ f_{J_0}-f_n^*)\circ G_{J_0-1}
    \|_{\ell^\infty(\mu_X)} \\
    &\qquad
    +
    \|f_n^*\circ G_{J_0-1}-f_n^*\|_{\ell^\infty(\mu_X)} \\
    &\le
    \|f_J\circ\cdots\circ f_{J_0}-f_n^*\|_{\ell^\infty(\tilde\mu)}
    +
    \sup_{\|x-x'\|_\infty\le\tilde\epsilon_n}
    |f_n^*(x)-f_n^*(x')| \\
    &\le
    2\epsilon_n ,
\end{aligned}
\end{align*}
where the last inequality uses Assumptions
\ref{assump:proxy regular} and \ref{assump:prior mass}. Consequently,
\begin{align}\label{eq:A dist}
    \|\mc C(\theta)-f^*\|_{\ell^2(\mu_X)}
    &\le
    \|\mc C(\theta)-f_n^*\|_{\ell^2(\mu_X)}
    +
    \|f_n^*-f^*\|_{\ell^2(\mu_X)} \nonumber \\
    &\le
    2\epsilon_n+\epsilon_n
    =
    3\epsilon_n .
\end{align}
Moreover, the same display and the boundedness of $f_n^*$ and $f^*$ imply that $\|\C(\theta)-f^*\|_{\ell^\infty(\mu_X)}$ is uniformly bounded on $A$. By Lemma~\ref{lem:kl-lp}, $A$ is therefore contained in the Kullback-Leibler and log-likelihood-variance neighborhood used by the fractional-posterior contraction theorem.

We proceed to show the set $A$ has sufficient prior mass under the augmented prior $\tilde\Pi$.
Since the
original DGP prior is a product prior and the augmentation factorizes over
layers, the normalized augmented prior also factorizes over layers. Thus
\begin{align*}
    \tilde\Pi (A)
    =
    \left[
    \prod_{j=1}^{J_0-1}
    \tilde\Pi
    \left(
        \|f_j-\id\|_{\ell^\infty(\tilde\mu)}
        \le \tilde\epsilon_n/J_0
    \right)
    \right]
    \tilde\Pi (A_{J_0}).
\end{align*}
By Lemmas \ref{lem:augmented-identity-prior-mass} and
\ref{lem:augmented-active-block-mass},
\begin{align}\label{eq:A prior mass}
    \ln \tilde\Pi (A)
    \gtrsim
    -n\epsilon_n^2 .
\end{align}
Combining equations \eqref{eq:A dist} and \eqref{eq:A prior mass} with the uniform bound in the preceding paragraph, the augmented prior assigns at least $\exp\{-Cn\epsilon_n^2\}$ mass to a Kullback-Leibler and centered log-likelihood-variance neighborhood of $f^*$ of order $n\epsilon_n^2$, by Lemma~\ref{lem:kl-lp}. 
Hence, by Lemma \ref{lem:frac contract}, for sufficiently large $M$,
\[
\tilde{\Pi}_\rho\!\left(
D_\rho(p_{\mc C(\theta)}^n,p_{f^*}^n)
>M n\epsilon_n^2
\,\middle|\,\D_n\right)
\overset{\mathbb P_{f^*}^n}{\longrightarrow}0.
\]
And if Assumption~\ref{assump:bounded} also holds, 
\[
\tilde{\Pi}_\rho\!\left(
\|\mc C(\theta)-f^*\|_{\ell^2(\mu_X)}
>M\epsilon_n
\,\middle|\,\D_n\right)
\overset{\mathbb P_{f^*}^n}{\longrightarrow}0
\]
after enlarging $M$ if necessary.
\end{proof}

\begin{proof}[Proof of Theorem \ref{thm:parsimonious}]
Fix \(1\le j<J_0\). Let
\begin{align*}
    t_{j,n}
    :=
    M \frac{n\epsilon_n^2}{\eta_{j,n}},
    \qquad
    E_j
    :=
    \left\{
        \theta:
        \|f_j-\id\|_{\ell^p(\tilde\mu)}
        > t_{j,n}
    \right\}.
\end{align*}
We prove
\begin{align*}
    \tilde\Pi_\rho(E_j\mid \D_n)
    \overset{\mathbb P_{f^*}^n}{\longrightarrow}0.
\end{align*}
Since \(j<J_0\) is arbitrary and \(J\) is fixed, the theorem follows by a
finite union bound.

Let
\begin{align*}
    r_n(f,f^*)
    :=
    -\ln
    \prod_{i=1}^n
    \frac{p(Y_i\mid X_i,f)}
         {p(Y_i\mid X_i,f^*)}.
\end{align*}
Then
\begin{align*}
    \tilde\Pi_\rho(E_j\mid\D_n)
    =
    \frac{
        \int_{E_j}
        \exp\{-\rho r_n(\mc C(\theta),f^*)\}
        \tilde\Pi(d\theta)
    }{
        \int_{\Theta}
        \exp\{-\rho r_n(\mc C(\theta),f^*)\}
        \tilde\Pi(d\theta)
    }.
\end{align*}

We first lower bound the denominator. From the proof of Theorem
\ref{thm:contraction}, there exists a set \(A\) such that
\begin{align*}
    \ln\tilde\Pi(A)\ge  -C_A n\epsilon_n^2
\end{align*}
for some constant \(C_A>0\), and \(A\) is contained in a Kullback-Leibler and centered log-likelihood-variance neighborhood of \(f^*\) of
order \(n\epsilon_n^2\). Therefore, by the usual denominator lower-bound
argument for Bayesian posteriors (see, e.g., proof of Theorem 3.1 of \cite{bhattacharya2019bayesian}), there exists an event \(\Omega_n\) with
\(  \mathbb P_{f^*}^n(\Omega_n)\to1\), such that on \(\Omega_n\),
\begin{align*}
    \int_{\Theta}
    \exp\{-\rho r_n(\mc C(\theta),f^*)\}
    \tilde\Pi(d\theta)
    \ge
    \exp\{-C_D n\epsilon_n^2\}.
\end{align*}
Here \(C_D>0\) is a constant independent of sample size \(n\).

It remains to upper bound the enumerator. Since
\begin{align*}
    \mathbb E_{f^*}^n
    \exp\{-\rho r_n(\mc C(\theta),f^*)\}
    =
    \exp\{-(1-\rho)D_\rho(p_{\mc C(\theta)}^n,p_{f^*}^n)\}
    \le 1,
\end{align*}
Fubini's theorem gives
\begin{align*}
\begin{aligned}
    \mathbb E_{f^*}^n
    \int_{E_j}
    \exp\{-\rho r_n(\mc C(\theta),f^*)\}
    \tilde\Pi(d\theta)
    &\le
    \tilde\Pi(E_j).
\end{aligned}
\end{align*}
Because the augmented prior factorizes across layers, let
\[
\Xi_{j,n}
=
\int_{\mc F_j}
\exp\{-\eta_{j,n}\|g-\id\|_{\ell^p(\tilde\mu)}\}\Pi(dg)
\]
be the normalizing constant of its $j$th marginal. The tail numerator satisfies
\begin{align*}
\int_{\|g-\id\|_{\ell^p(\tilde\mu)}>t_{j,n}}
\exp\{-\eta_{j,n}\|g-\id\|_{\ell^p(\tilde\mu)}\}\Pi(dg)
\le \exp\{-\eta_{j,n}t_{j,n}\}
=\exp\{-M n\epsilon_n^2\}.
\end{align*}
 Then
\[
\tilde\Pi(E_j)
=
\frac{
\int_{\|f_j-\id\|_{\ell^p(\tilde\mu)}>t_{j,n}}
\exp\{-\eta_{j,n}\|f_j-\id\|_{\ell^p(\tilde\mu)}\}\Pi(df_j)}
{\Xi_{j,n}}.
\]
For $
B_{j,n}=\{g:\|g-\id\|_{\ell^\infty(\tilde\mu)}\le\tilde\epsilon_n\}
$, 
Assumption~\ref{assump:layer regular} gives
\[
\Xi_{j,n}
\ge
e^{-d^{1/p}\eta_{j,n}\tilde\epsilon_n}\Pi(B_{j,n})
\ge
\exp\{-C n\epsilon_n^2\},
\]
because $\eta_{j,n}\tilde\epsilon_n\lesssim n\epsilon_n^2$ and
$\tilde\epsilon_n^{-d/\alpha}\lesssim n\epsilon_n^2$. Hence
\[
\tilde\Pi(E_j)\le \exp\{-(M-C)n\epsilon_n^2\}.
\]
On the event \(\Omega_n\),
\begin{align*}
\begin{aligned}
    \tilde\Pi_\rho(E_j\mid\D_n)
    &\le
    \exp\{C_D n\epsilon_n^2\}
    \int_{E_j}
    \exp\{-\rho r_n(\mc C(\theta),f^*)\}
    \tilde\Pi(d\theta).
\end{aligned}
\end{align*}
Taking expectation and using the preceding numerator bound,
\begin{align*}
\begin{aligned}
    \mathbb E_{f^*}^n
    \left[
        \tilde\Pi_\rho(E_j\mid\D_n)
        \mathbf 1_{\Omega_n}
    \right]
    &\le
    \exp\{C_D n\epsilon_n^2\}
    \tilde\Pi(E_j) \\
    &\le
    \exp\{-(M-C-C_D)n\epsilon_n^2\}.
\end{aligned}
\end{align*}
For a sufficiently large fixed constant \(M\) and \(n\epsilon_n^2\to\infty\), the right-hand side
converges to zero. Also \(\mathbb P_{f^*}^n(\Omega_n^c)\to0\). Hence
\begin{align*}
    \mathbb E_{f^*}^n
    \tilde\Pi_\rho(E_j\mid\D_n)
    \to0,
\end{align*}
and Markov's inequality yields
\begin{align*}
    \tilde\Pi_\rho(E_j\mid\D_n)
    \overset{\mathbb P_{f^*}^n}{\longrightarrow}0.
\end{align*}
The proof is complete.
\end{proof}

\section{Supplement and Proofs for Section \ref{sec:apply theory}}\label{sec:proof corollaries}
\subsection{The Correctly Specified H\"{o}lder Class}\label{cor:holder correct}
We state a corollary regarding applying our main theorems to the H\"{o}lder class, where the prior regularity correctly matches that of the true data-generating function.
\begin{corollary}[H\"{o}lder smooth, correctly specified]\label{eg:Holder correct}
Let $f^* \in C^\beta(\X)_1 $ for $\beta\ge 1$, and let the GP prior for all layers be the Mat\'{e}rn
process with regularity $\beta$. Then $f^*$ can be optimally estimated with a
single GP layer at the rate $n^{-\frac{\beta}{2\beta+d}}$. Choose  $n^{\frac{d}{2\beta+d}} \ll \eta_{j,n} \lesssim n^{\frac{\beta+d}{2\beta+d}}$ for all $1\le j\le J-1$ and $\eta_{J,n}\lesssim n^{\frac{d}{2\beta+d}}$ for the final layer.
Then, except for the final layer, all other layers contract to the identity mapping
under the fractional augmented posterior. Moreover, the DGP recovers the optimal posterior contraction rate $n^{-\frac{\beta}{2\beta+d}}$.
\end{corollary}

\begin{proof}[Proof of Corollary \ref{eg:Holder correct}]
Let $\epsilon_n=n^{-\beta/(2\beta+d)}$ and $\tilde\epsilon_n=c_0\epsilon_n$ for some sufficiently small constant $c_0$. Then $\tilde{\epsilon}_n^{-d/\beta} \asymp n\epsilon_n^2$. Let $f_n^* = f^*$. Since $\beta\ge1$, the function $f^*$ is Lipschitz and Assumption \ref{assump:proxy regular} holds. To prove Assumption \ref{assump:prior mass}, let $J_0=J$, then it suffices to prove 
\[
\ln \Pi(\|Z_J-f^*\|_{\ell^\infty(\tilde\mu)}<\epsilon_n)
\gtrsim -n\epsilon_n^2,
\]
which follows equation~\eqref{eq:matern decentered small ball}. The tuning window of $\eta_{j,n}$ follows \eqref{eq:eta-window-intuition} and $\epsilon_n \asymp \tilde \epsilon_n = n^{-\frac{\beta}{2\beta+d}}$.
\end{proof}

\subsection{Proof for Corollaries in Section \ref{sec:apply theory}}
In all the corollaries of Section \ref{sec:apply theory}, the layerwise GP priors are set to be Mat\'{e}rn processes, so Assumption \ref{assump:layer regular} is satisfied. Assumption \ref{assump:bounded} is not proved, as it does not hold for Mat\'{e}rn. In practice, one can choose to artifically apply a bound on the final layer, or leave it be and obtain R\'{e}nyi divergence contraction. Our efforts center on verifying Assumption \ref{assump:proxy regular}, \ref{assump:prior mass}, and equation \eqref{eq:eta-window-intuition}, \eqref{eq:eta range} for the tuning parameters $\eta_{j,n}$. The former is mainly proved by construction, utilizing the preliminary result of Lemma \ref{lem:DGP small ball}; the latter is simply attained by choosing some appropriate $\eta_{j,n}$, which is usually quite straightforward once some proper $\tilde{\epsilon}_n$ and $\epsilon_n$ are identified.

\begin{proof}[Proof of Corollary \ref{eg:Holder over}]
Identify $f^*$ with its assumed H\"older extension to a fixed neighborhood of $\tilde\X$. Multiplying that extension by a smooth cutoff yields a compactly supported global extension agreeing on a smaller fixed neighborhood of $\tilde\X$. The rescaled Mat\'ern bounds below are used only for $\tau_n\ge1$, with constants uniform on that fixed cube. For the oversmoothed case, we choose $f^*_n = f^*$, $\tilde{\epsilon}_n = n^{-\frac{\alpha}{2\beta+d}}/\sqrt{d}$ and $\epsilon_n = n^{-\frac{\beta}{2\beta+d}}$. Since the extension of $f^*$ has a fixed H\"older-$\beta$ norm, for all sufficiently large $n$ we have
\begin{align*}
 \sup_{\|x-x'\|_\infty\le \tilde\epsilon_n} |f^*_n(x) - f^*_n(x')| \le C(\sqrt{d}\tilde\epsilon_n)^{\min\{\beta,1\}} = C n^{-\frac{\alpha\min\{\beta,1\}}{2\beta+d}} \le n^{-\frac{\beta}{2\beta+d}} = \epsilon_n.
\end{align*}
    Thus Assumption \ref{assump:proxy regular} is satisfied. We need at least two active GP layers to satisfy Assumption \ref{assump:prior mass}. For notational simplicity, we set $J=2$. In the case of $J\ge 3$, the following procedure can be applied to the last two layers to reach the same conclusion.
    The proof is done via construction. Let $\tau_n = n^{\frac{\alpha-\beta}{(2\beta+d)\alpha}}$. Define functions $\bar{f}_1$, $\bar{f}_2$ as:
    \begin{align*}
\bar{f}_1(x) = \tau_n x,
\end{align*}
    \begin{align*}
\bar{f}_2(x) = f^*(\tau_n^{-1} x).
\end{align*}
    Let
    $\X_{2,n}=\tau_n\tilde\X+[-2\tilde\epsilon_n,2\tilde\epsilon_n]^d$.
    Define $B_1$ and $B_2$ by
    \begin{align*}
    B_1 &= \{f_1\in\mathcal{F}_1: \|f_1-\bar{f}_1\|_{\ell^\infty(\tilde{\mu})} \le \tilde\epsilon_n  \},\\
    B_2 &= \{f_2\in\mathcal{F}_2: \|f_2-\bar{f}_2\|_{\ell^\infty(\X_{2,n})} \le C_0\epsilon_n  \},
\end{align*}
    where $C_0>0$ is a sufficiently large fixed constant. Since the extension of $f^*$ has bounded H\"older-$\beta$ norm on a fixed neighborhood of $\tilde\X$,
    \begin{align*}
    \sup_{\substack{x,x'\in\X_{2,n}\\ \|x-x'\|_\infty<2\tilde\epsilon_n }} |\bar f_2(x) - \bar{f}_2(x')|
    \le C(2\sqrt{d}\,\tau_n^{-1}\tilde\epsilon_n)^{\min\{\beta,1\}}
    \lesssim \epsilon_n.
\end{align*}
    Hence $C_0$ can be chosen so that Lemma \ref{lem:DGP small ball} applies with $\epsilon_1=\tilde\epsilon_n$ and $\epsilon_2=C_0\epsilon_n$, giving
    \begin{align*}
    \ln \Pi(\|f_2\circ f_1  - \bar{f}_2 \circ \bar{f}_1  \|_{\ell^\infty(\tilde{\mu})} <2C_0\epsilon_n) \ge \ln \Pi(B_1) + \ln \Pi(B_2).
\end{align*}
    For the prior mass of $B_1$, choose fixed $h_k\in C_c^\infty(\mathbb R^d)$ with $h_k(x)=x[k]$ on $\tilde\X$. These are ambient Mat\'ern RKHS elements. Apply shifting to $\tau_n h_k$, which agrees with $\bar f_1[k]$ on $\tilde\X$ (the global linear map itself is not an ambient RKHS element). With $Z_1$ denoting the GP prior on each coordinate of layer $1$,
    \begin{align*}
        -\ln\Pi(B_1)
        \le & \frac{\tau_n^2}{2}\sum_{k=1}^d\|h_k\|_{\mathcal{H}_1}^2 - d \ln \Pi(\|Z_1\|_{\ell^\infty(\tilde\X)}<\tilde\epsilon_n ) \\
        \lesssim & \tau_n^2 + \tilde\epsilon_n ^{-d/\alpha} \asymp n^{\frac{d}{2\beta+d}},
\end{align*}
    where the last equation utilizes $2(\alpha-\beta)/\alpha \le d $, which always holds if $d\ge 2$ and holds for $\alpha\le2\beta$ if $d=1$.
    For the prior mass of $B_2$, let $\tilde\X^+$ be a fixed compact neighborhood of $\tilde\X$ on which the stated extension of $f^*$ is H\"{o}lder-$\beta$. Since $\X_{2,n}/\tau_n\subset\tilde\X^+$ for all sufficiently large $n$, the standard rescaled Mat\'{e}rn bound on the fixed set $\tilde\X^+$ gives
    \begin{align*}
    -\ln\Pi(B_2)
     \lesssim \tau_n^d \epsilon_n^{-d/\alpha} + \tau_n^{-2\alpha} \epsilon_n^{-\frac{2(\alpha-\beta)+d}{\beta}}
     \asymp n^{\frac{d}{2\beta+d}},
\end{align*}
    where fixed factors such as $C_0$ do not affect the order. Therefore,
    \begin{align*}
\ln \Pi(\|f_2\circ f_1  - \bar{f}_2 \circ \bar{f}_1  \|_{\ell^\infty(\tilde{\mu})} <2C_0\epsilon_n ) \gtrsim  -n^{\frac{d}{2\beta+d}}.
\end{align*}
    Absorbing the fixed factor $2C_0$ into the rate constant, Assumption \ref{assump:prior mass} is satisfied with rate $\epsilon_n = n^{-\frac{\beta}{2\beta+d}}$, by choosing $\eta_{1,n}\lesssim n \epsilon_n^2 / \tau_n \asymp n^{\frac{(d-1)\alpha+\beta}{(2\beta+d)\alpha} }$.

    For the case that $J>2$, the above procedure applies to the last two layers. In addition, choose $n^{\frac{d}{2\beta+d}} \ll \eta_{j,n} \lesssim n^{\frac{\alpha+d}{2\beta+d}} $ for $j\le J-2$; the strict lower order makes the layerwise radius vanish, while the upper order satisfies \eqref{eq:eta-window-intuition}.
\end{proof}

We first prove corollary \ref{cor:discont}, during which we make some technical preparations.
\begin{proof}[Proof of Corollary~\ref{cor:discont}]
For notational simplicity, we first consider the case that $J = J^*$. The case where $J>J^*$ will be dealt at the end of the proof.

The idea of estimating a discontinuous function $f^*$ in fractional posterior is to find a continuous $f^*_n$ that is sufficiently close to $f^*$ in $\ell^2$ norm. Such $f^*$ inevitably bears huge RKHS norm (if not completely outside the support of RKHS), making it unsuitable to estimate with a single GP. However, through delicate construction, it is possible that $f^*$ admits a decomposition
\begin{align*}
f^*_n = \bar f_J \circ \bar f_{J-1} \circ \cdots \circ \bar f_1,
\end{align*}
such that each $\bar{f}_j, 1\le j\le J$ has controllable RKHS norm and sufficient (single-layer) prior mass in its neighborhood. To begin with, let us introduce an odd, nondecreasing function $\phi: [-1,1] \to [-1,1]$, such that:
\begin{itemize}
    \item $\phi$ is at least $\lceil \alpha+1/2 \rceil$th differentiable,
    \item $\phi(x)=-1$ for $x\in[-1,-1/2]$,
    \item $\phi(x)=1$ for $x\in[1/2, 1]$,
    \item $|\phi(x)|\ge 2|x|$ for $x\in[-1/2,1/2]$,
    \item the restricted Mat\'{e}rn RKHS norm of $\phi$ is finite.
\end{itemize}
In other words, $\phi$ is a sufficiently differentiable function that goes from $-1$ to $1$ on $[-1/2,1/2]$, while being constants for the remaining parts. We can construct $\phi$ by integrating an even smooth bump supported on $[-1/2,1/2]$ that decreases on $(0,1/2)$, and normalizing its range to $[-1,1]$. The specific form of $\phi$ does not affect our proof.

For all $\tau\ge 1$, we define a rescaled version of the above mentioned function $\phi$ as:
\begin{align*}
  \phi^\tau(x) = \left\{
\begin{aligned}
& \phi(\tau x), \;\; && |x|\le 1/\tau,\\
& -1, \;\; && x<-1/\tau, \\
& 1, \;\; && x> 1/\tau.
\end{aligned}\right.
\end{align*}
In other words, $\phi^\tau(\cdot)$ ``squeezes'' the input domain of $\phi(\cdot)$ by a factor of $\tau$, and defines the remaining parts by constant extension. Fix $K=[-R,R]$ with $R\ge2$ and $\tilde\X\subset(-R,R)$, and choose $\chi\in C_c^\infty(\mathbb R)$ equal to one on $K$. Use $\bar f_j=\chi\phi^{\tau_j}$ as the ambient RKHS extensions. Writing $s=\alpha+1/2$, Fourier scaling of the derivative of $\chi\phi^\tau$ gives, for an integer $N>s$, 
\begin{equation}\label{eq:sobolev tau}
\|\chi\phi^\tau\|_{\mathcal H}^2
\asymp \|\chi\phi^\tau\|_{H^s(\mathbb R)}^2
\lesssim 1+\int_1^\infty \omega^{2s-2}(1+\omega/\tau)^{-2N}\,d\omega \lesssim 1+\tau^{2\alpha}.
\end{equation}
Then the derivative is the sum of a fixed smooth function supported outside $K$ and $\tau\phi'(\tau\,\cdot)$. On $K$, the extension agrees with $\phi^\tau$ and takes values in $[-1,1]$.
Moreover, for all $x\in[-1,1]$, we have either $|\phi^\tau(x)|=1$ (for $|x|>1/(2\tau)$) or $|\phi^\tau(x)|\ge 2\tau |x|$ (for $|x|\le 1/(2\tau)$).
Let us define $f_n^*$ as the composition of $J$ rescaled version of $\phi$ as
\begin{align*}
f_n^* = \phi^{\tau_J} \circ \phi^{\tau_{J-1}} \circ \cdots \circ \phi^{\tau_1}.
\end{align*}
By the analysis above, for all $x\in[-1,1]$, we either have $|f_n^*(x)|=1$, or $|f_n^*(x)|\ge \prod_{j=1}^J (2\tau_j) |x|$, which can only happen for $|x|\le 1/  \prod_{j=1}^J (2\tau_j) $. Therefore, we can bound the $\ell^2(\mu_X)$ norm between $f^*$ and $f_n^*$ as
\begin{align*}
\|f^* - f_n^*\|^2_{\ell^2(\mu_X)} \le \int_{ -[\prod_{j=1}^J (2\tau_j)]^{-1}}^{ [\prod_{j=1}^J (2\tau_j)]^{-1}} 4\,d\mu_X(x) \le C_X 2^{-J+3} \left(\prod_{j=1}^J \tau_j\right)^{-1},
\end{align*}
where $C_X$ bounds the design density. Thus, a sufficient condition for $\|f^* - f_n^*\|_{\ell^2(\mu_X)}\le \epsilon_n$ is
\begin{equation}\label{eq:discont proxy close}
\prod_{j=1}^J \tau_j \ge C_X 2^{-J+3} \epsilon_n^{-2}.
\end{equation}
Each $\phi^\tau_j$ is Lipschitz continuous with constant $C\tau_j$, with $C$ an absolute constant determined by the construction of $\phi$. Thus a sufficient condition for $\sup_{\|x-x'\|_\infty< \tilde{\epsilon}_n} | f^*_n(x) -  f^*_n(x')| \le \epsilon_n$ is
\begin{equation}\label{eq:discont proxy cont}
\tilde{\epsilon}_n \prod_{j=1}^J (C\tau_j) \le \epsilon_n.
\end{equation}
In addition, we also need
\begin{equation}\label{eq:discont 2 rates}
\tilde{\epsilon}_n^{-1/\alpha} \lesssim n\epsilon_n^2.
\end{equation}
The equations \eqref{eq:discont proxy close}, \eqref{eq:discont proxy cont} and \eqref{eq:discont 2 rates} jointly provide a sufficient condition for Assumption \ref{assump:proxy regular}. In order to satisfy Assumption \ref{assump:prior mass}, we need to specifically construct the set with enough prior mass. Specifically, define set $A_j, 1\le j\le J$ as
\begin{align*}
A_j = \{f_j\in\mathcal{F}_j: \|f_j - \bar f_j\|_{\ell^\infty(K)} < c_0\tilde{\epsilon}_n\}.
\end{align*}
Here $c_0>0$ is a sufficiently small fixed constant. Recaling Lemma \ref{lem:DGP small ball}, since $\tilde{\epsilon}_n \prod_{j'=j}^J (C\tau_{j'}) \le \epsilon_n$ for all $j=1, 2, \cdots, J$, all perturbed intermediate images stay in $K$ for large $n$, since each layer is close on $K$ to a map taking values in $[-1,1]$. Propagating the layer errors gives
\begin{align*}
\sup_{\theta\in \prod_{j=1}^J A_j}\|\mathcal C(\theta)-f_n^*\|_{\ell^\infty(\tilde\X)}
&\le c_0\tilde\epsilon_n\sum_{j=1}^J\prod_{k=j+1}^J(C\tau_k)\\
&\le c_0J\tilde\epsilon_n\prod_{k=1}^J(C\tau_k)
\le\epsilon_n,
\end{align*}
after taking $C\ge1$ and $c_0\le1/J$.
Thus, it suffices to bound the prior mass on $\prod_{j=1}^J A_j$. In view of equations \eqref{eq:decenter RKHS} and \eqref{eq:matern centered small ball}, we have
\begin{align*}
\ln\Pi(A_j) \gtrsim - \frac{1}{2}\|\bar f_j\|_{\mc{H}_j}^2 - \tilde{\epsilon}_n^{-1/\alpha} \gtrsim -\tau_j^{2\alpha} - \tilde{\epsilon}_n^{-1/\alpha}.
\end{align*}
Since we already requires \eqref{eq:discont 2 rates}, Assumption \ref{assump:prior mass} is satisfied if
\begin{equation}\label{eq:discont mass decenter}
    \tau_j^{2\alpha} \lesssim n\epsilon_n^2, \;\;\forall 1\le j \le J
\end{equation}
and
\begin{equation}\label{eq:discont eta}
    \tau_j \eta_{j,n} \lesssim n\epsilon_n^2, \;\;\forall 1\le j \le J.
\end{equation}

Jointly solve the inequalities \eqref{eq:discont proxy close} to \eqref{eq:discont mass decenter} yield the feasible spaces for $\tau_j, \forall j$ and $\tilde{\epsilon}_n, \epsilon_n$. This is by no means an easy task. Instead, we impose some restrictions. First, since all the inequalities are symmetric with respect to $j$, we enforce $\tau_1=\tau_2 =\cdots =\tau_J \triangleq\tau$. Second, as long as \eqref{eq:discont 2 rates} is satisfied, there is no downside in having a smaller $\tilde\eps_n$. Hence we let the equality in \eqref{eq:discont 2 rates} hold and substitute  $\tilde\eps_n$ with expressions involving $\epsilon_n$. Therefore, the previously mentioned four equations yield the following three bounds for $\tau$:
\begin{align*}
\left\{
\begin{aligned}
&\tau^J \ge {C_X} 2^{-J+3}\epsilon_n^{-2},\\
&\tau^J \le C^{-J} n^\alpha \epsilon_n^{2\alpha+1},\\
&\tau^{2\alpha}\lesssim n\epsilon_n^2.
\end{aligned}\right.
\end{align*}
In order for the feasible set of $\tau$ to be non-empty, we require both upper bounds to be greater than the lower bound, which yield
\begin{equation}\label{eq:discont eps}
\epsilon_n \gtrsim n^{-\frac{\alpha}{2\alpha+3}} + n^{-\frac{J}{2J+4\alpha}} \asymp n^{-\frac{\alpha}{2\alpha+3}}, \quad \tilde\epsilon_n \asymp n^{-\frac{3\alpha}{2\alpha+3}}.
\end{equation}
where the last formula comes from $J \ge \frac{4}{3}\alpha^2$. Take $\epsilon_n=L n^{-\alpha/(2\alpha+3)}$, $\tau=(C_X2^{-J+3})^{1/J}\epsilon_n^{-2/J}$, and $\tilde\epsilon_n=\epsilon_n/[2(C\tau)^J]$, with a sufficiently large fixed $L$. These choices satisfy the displayed bounds.
Finally, notice elements in each $A_j$ have bounded $\ell^p$ norms, hence we can choose
\begin{align*}
 \eta_{j,n} \tau \lesssim n\epsilon_n^2, \quad \text{hence}\quad \eta_{j,n} \lesssim n^{\frac{3J^*-2\alpha}{(2\alpha+3)J^*}}.
\end{align*}
This concludes the proof for the case when $J = J^*$. If $J>J^*$, the above proof procedure holds for the last $J^*$ without modification. For the first $J-J^*$ layers, simply choosing $n^{\frac{3}{2\alpha+3}} \ll \eta_{j,n} \lesssim n^{\frac{3\alpha+3}{2\alpha+3}}$ to satisfy equation \eqref{eq:eta-window-intuition}.
\end{proof}

We now prove Corollary~\ref{eg:non-diff}, where we make use of the $\phi$ function defined before.
\begin{proof}[Proof of Corollary~\ref{eg:non-diff}]
We prove for the case that $J=1$. If $J>1$, choosing $\eta_{j,n}$ as in the statement for $j\le J-1$ satisfies equation \eqref{eq:eta-window-intuition}. To begin with, let us recall $\phi$ and $\phi^\tau$ defined in the proof of Corollary \ref{cor:discont}. Define a constant $C_{int}$ as
$$C_{int} = - \int_{-1/2}^0 \phi(x) dx,$$
then $C_{int}\in(0,1/2)$ and is fixed once the choice of $\phi$ is fixed. Now define function $f_n^*(x), \forall x\in[-1,1]$ as
$$f_n^*(x) = -\frac{1/2-C_{int}}{\tau} - \int_0^x \phi^\tau(w) dw.$$
By the definition of  $\phi^\tau$, it is easy to see $f_n^*(0)=- (1/2-C_{int})/\tau$, is an even function, and is increase on $[-1,0]$. Moreover, 
$$
f_n^*(-1/(2\tau)) = -\frac{1/2-C_{int}}{\tau} + \int_{-1/(2\tau)}^0 \phi(\tau w) dw = -\frac{1/2-C_{int}}{\tau} - \frac{C_{int}}{\tau} = -\frac{1}{2\tau}.
$$
Noticing $\phi^\tau(x)=-1$, then $\forall x\in [-1, -1/(2\tau)]$, we have $f_n^*(x) = f^*(x)$. For $x\in[-1/(2\tau),0]$, both $f^*(x)$ and $f^*_n(x)$ are increasing functions taking values in $[-1/(2\tau),0]$. Thus we have
$$
\|f^*-f_n^*\|_{\ell^2(\mu_X)}^2 \le 2C_X \int_{-1/(2\tau)}^0 [f^*(x)-f_n^*(x)]^2 dx \le C_X[1/(2\tau)]^2 \cdot 2\cdot [1/(2\tau)] \le C_X/(4\tau^3).
$$
Here $C_X=\|d\mu_X/dx\|_\infty$.
We now consider Assumption \ref{assump:proxy regular}. Choosing $\tilde{\epsilon}_n = \epsilon_n = n^{-\frac{\alpha}{2\alpha+1}}$, then it is easy to $\tilde{\epsilon}_n^{-d/\alpha} = n^{\frac{1}{2\alpha+1}} = n \epsilon_n^2$. Also notice the derivative of $f_n^*(x)$ is just $-\phi^\tau(x)$, with its absolute value bounded by $1$. Thus $f_n^*$ is Lipschitz-1. Therefore, the Assumption \ref{assump:proxy regular} is satisfied as long as
\begin{equation}\label{eq:proxy close nondiff}
C_X/(4\tau^3) \le \epsilon_n^2 .
\end{equation}
We proceed to consider Assumption \ref{assump:prior mass}. Since we only have one active layer, it is simply the prior mass of Mat\'{e}rn process. Noticing $f_n^*(x)\in[-1,0]$ for all $x\in[-1,1]$, and the derivative of $f_n^*(x)$ is just $-\phi^\tau(x)$ and applying the Fourier argument in \eqref{eq:sobolev tau} to the second derivative, we have, for $\chi$ as above and $f_n^*$ extended by the same formula to $\mathbb R$, 
$$
\|\chi f_n^*\|_{\mc H}^2 \lesssim 1+\ln\tau+ \tau^{2(\alpha-1)}.
$$
Then, in view of equations \eqref{eq:matern centered small ball} and \eqref{eq:decenter RKHS}, 
$$\ln \Pi(\|f_J-f_n^*\|_{\ell^\infty(\tilde\mu)}\le \epsilon_n) \gtrsim - \frac{1}{2}\|\chi f_n^*\|_{\mc H}^2 - \epsilon_n^{-1/\alpha} \gtrsim -1-\ln\tau - \tau^{2(\alpha-1)} - n^{\frac{1}{2\alpha+1}}.$$
In order to satisfy Assumption \ref{assump:prior mass}, we need 
\begin{equation}\label{eq:prior mass nondiff}
1+\ln\tau+\tau^{2(\alpha-1)} \lesssim n\epsilon_n^2.
\end{equation}
Combining equations \eqref{eq:prior mass nondiff} and \eqref{eq:proxy close nondiff}, taking the values $\tau = c_\tau n^{\frac{1}{2\alpha+1}}$ and $\epsilon_n = n^{-\frac{\alpha}{2\alpha+1}}$, with $c_\tau\ge\max\{1,(C_X/4)^{1/3}\}$ and $0<\alpha\le3/2$, we finish the proof.
\end{proof}

\par\medskip 
\begin{proof}[Proof of Corollary \ref{cor:intrinsic}] Choose $f_n^*=f^*$, $\tilde{\epsilon}_n = n^{-\frac{\alpha}{2\alpha+2}}/\sqrt{2}$, and $\epsilon_n \asymp n^{-\frac{\beta}{2\beta+1}} = n^{-\frac{1}{3}}$. It is straightfoward to see $\tilde\epsilon_n^{-d/\alpha} \asymp n^{\frac{d}{2\alpha+d}}\lesssim n^{\frac{1}{3}} \asymp n\epsilon_n^2$. Since $f_n^*$ belong to H\"{o}lder space with regularity $\beta=1$, $f_n^*$ is Lipschitz, hence Assumption \ref{assump:proxy regular} holds. It remains to prove Assumption \ref{assump:prior mass}. Without loss of generality, we do the construction for the case of $J=2$, with the consensus that the same procedure can be applied to the last two layers if $J\ge 3$.

Define functions $\bar{f}_1$, $\bar{f}_2$ as:
    \begin{align*}
\bar{f}_1(x[1],x[2]) = ((x[1]+x[2])/2,0),
\end{align*}
$$
\bar{f}_2(x[1],x[2]) = (g^*(x[1]),0).
$$
Then it is easy to verify that $\bar{f}_2 \circ \bar{f}_1 = f^*$.
Define $B_1$ as a set of mappings from $\mathbb{R}^d$ to itself, and $B_2$ as a set of functions from $\mathbb{R}^d$ to $\mathbb{R}$, as
$$
B_1 = \{f_1\in\mathcal{F}_1: \|f_1-\bar{f}_1\|_{\ell^\infty(\tilde{\mu})} \le \tilde\epsilon_n  \}, 
$$
$$
    B_2 = \{f_2\in\mathcal{F}_2: \|f_2-\bar{f}_2\|_{\ell^\infty((1+\tilde\epsilon_n) \tilde{\X}) \cap \{|x[2]|\le\tilde\epsilon_n\} } \le \epsilon_n  \}.
$$
Since $g^*$ is a H\"{o}lder-$\beta$ smooth function with norm no greater than $1$, we have
\begin{align*}
\sup_{\|x-x'\|_\infty<2\tilde\epsilon_n } |\bar f_2(x) - \bar{f}_2(x')| 
\le & \sup_{\|x[1]-x'[1]\|_\infty<2\tilde\epsilon_n } |\bar g^*(x[1]) - g^*(x[1)| \le 2\tilde{\epsilon}_n \le 2\epsilon_n.
\end{align*}
Thus by Lemma \ref{lem:DGP small ball}, we have
\begin{equation}\label{eq:intrinsic merge}
    \ln \Pi(\|f_2\circ f_1  - \bar{f}_2 \circ \bar{f}_1  \|_{\ell^\infty(\tilde{\mu})} <4\epsilon_n) \ge \ln \Pi(B_1) + \ln \Pi(B_2).
\end{equation}
The prior mass for $B_1$ is just the 
standard prior mass for Mat\'{e}rn process, written as
\begin{equation}\label{eq:intrinsic B1}
\ln \Pi(B_1) \gtrsim -\tilde\epsilon_n^{-d/\alpha} \ge -n^{\frac{1}{2\beta+1}}. 
\end{equation}
To derive the prior mass for $B_2$, we need to introduce an auxiliary, degenerated GP. Specifically, let $(Z_X)_{X\in\mathbb{R}^2}$ be the original $\beta$-regular Mat\'{e}rn process on $\mathbb{R}^2$. Now define a degenerated version of $Z_X$ as
$$
\bar{Z}_{(X[1],X[2])} \triangleq Z_{(X[1],0)}, \forall X\in\mathbb{R}^2.
$$
In other words, $\bar{Z}_X$ is essentially the variant of $Z_X$ where $X$ is projected to the second coordinate axis. If $X$ is restricted to a small band near the second coordinate axis, it is reasonable to conjecture that $\bar{Z}_X$ and $Z_X$ are very close to each other, as in the following lemma.

\begin{lemma}\label{lem:Z degen}
For $\beta = 1$, for the process $Z_X$ and $\bar{Z}_X$ defined above, 
\begin{equation*}
\mathbb{P}\left(\|Z-\bar Z\|_{\ell^\infty((1+\tilde\epsilon_n) \tilde{\X} \cap \{|x[2]|\le\tilde\epsilon_n\} } \ge t \right) \le \exp\left[\frac{(t-C' \tilde{\epsilon}_n \ln(1/\tilde{\epsilon}_n))^2}{2C^2\tilde{\epsilon}^2_n \ln(1/\tilde{\epsilon}_n) }  \right].
\end{equation*}
\begin{proof}
Define $W_X = Z_X - \bar Z_X$, then $W_X$ is again a centered Gaussian process. We employ Lemma 8 and 9 of the supplementary materials of \cite{castillo2025deep} to control the sup norm of the process $W_X$. First, let $\lambda$ be the intrinsic semimetric
$$\lambda^2(X_1, X_2) = \mathbb{E}(W_{X_1}-W_{X_2})^2.$$
Then, for $\|X_1-X_2\|_2<1$, applying Lemma 1 of \cite{szabo2026vecchia} with Mat\'{e}rn regularity being $1$, we have two bounds for the intrinsinc semimetric
\begin{align*}
\lambda^2(X_1, X_2) 
= & \mathbb{E}(Z_{X_1} - Z_{(X_1[1],0)}-Z_{X_2} + Z_{(X_2[1],0)})^2 \\
\le & 2\mathbb{E}(Z_{X_1} - Z_{X_2})^2 + 2\mathbb{E}(Z_{(X_1[1],0)}- Z_{(X_2[1],0)})^2\\
\le & C^2 \|X_1-X_2\|_2^2 \ln(1/\|X_1-X_2\|_2),
\end{align*}
\begin{align}\label{eq:W var}
\lambda^2(X_1, X_2) 
= & \mathbb{E}(Z_{X_1} - Z_{(X_1[1],0)}-Z_{X_2} + Z_{(X_2[1],0)})^2 \nonumber \\
\le & 2\mathbb{E}(Z_{X_1} - Z_{(X_1[1],0)})^2 + 2\mathbb{E}(Z_{X_2} - Z_{(X_2[1],0)})^2 \nonumber \\
\le & C^2 [(X_1[1])^2 \ln(1/\|X_1\|_2) + (X_2[1])^2 \ln(1/\|X_2\|_2) ] \nonumber \\
\le & 2 C^2 \tilde{\epsilon}_n^2 \ln(1/\tilde{\epsilon}_n).
\end{align}
for some absolute constant $C$. In other words, for $\xi\in(0,1)$, if $\|X_1-X_2\|_2 \le C^{-1}\xi   (\ln(1/\xi))^{-1/2}$, then $\lambda^2(X_1, X_2) \le \xi$. Thus, without loss of generality, assuming $(1+\tilde\epsilon_n) \tilde{\X}\subset[-2,2]^2$\footnote{$\tilde\X$ only needs to be a compact set with $[-1,1]^2$ belonging to its interior, so this assumption holds trivialy.}, we have
$$\ln N(\xi, (1+\tilde\epsilon_n) \tilde{\X} \cap \{|x[2]|\le\tilde\epsilon_n\}, \lambda) \le \ln\frac{2\tilde{\epsilon}_n}{C^2\xi^2 / \ln(1/\xi)} = \ln(2\tilde\epsilon_n C^{-2}) + 2\ln(1/\xi) + \ln \ln (1/\xi).
$$
Noticing $W_X = 0$ for all $X$ satisfying $X[1]=0$, applying Lemma 9 of the supplementary materials of \cite{castillo2025deep}, one obtains
\begin{align}\label{eq:EW inf}
	\mathbb{E} \|W\|_{\ell^\infty((1+\tilde\epsilon_n) \tilde{\X}) \cap \{|x[2]|\le\tilde\epsilon_n\} }
\le & 0 + 4\sqrt{2} \int_0^{\sqrt{2}C\tilde{\epsilon}_n} \ln N(\xi, (1+\tilde\epsilon_n) \tilde{\X}) \cap \{|x[2]|\le\tilde\epsilon_n\}, \lambda) d\xi \nonumber\\
\le & C' \tilde{\epsilon}_n \ln(1/\tilde{\epsilon}_n),
\end{align}
for some absolute constant $C'$, where the integration make uses of the fact that $\ln\ln(1/\xi)<\ln(1/\xi)$ and $\int_0^x \ln(1/\xi) = x - x\ln x$. Combining equations \eqref{eq:W var}, \eqref{eq:EW inf} and Lemma 8 of the supplementary materials of \cite{castillo2025deep}, we have
$$
\mathbb{P}\left(\|W\|_{\ell^\infty((1+\tilde\epsilon_n) \tilde{\X} \cap \{|x[2]|\le\tilde\epsilon_n\} } \ge t \right) \le \exp\left[\frac{(t-C' \tilde{\epsilon}_n \ln(1/\tilde{\epsilon}_n))^2}{2C^2\tilde{\epsilon}^2_n \ln(1/\tilde{\epsilon}_n) } \right].
$$
\end{proof}
\end{lemma}

Apply Lemma \ref{lem:Z degen} with $t=\epsilon_n$ and notice $\epsilon_n \asymp n^{-1/3} >> n^{-\frac{\alpha}{2\alpha+2}} \ln(n) \asymp \tilde\epsilon_n \ln(1/\tilde\epsilon_n)$, one obtains, 
\begin{equation}\label{eq:Z Zbar diff}
\Pi\left(\|Z-\bar Z\|_{\ell^\infty((1+\tilde\epsilon_n) \tilde{\X} \cap \{|x[2]|\le\tilde\epsilon_n\} } \ge \epsilon_n \right) \le \exp\left[-C'' n^{\frac{\alpha}{2}-\frac{2}{3}} \right] \le \exp\left[-C'' n^{\frac{1}{3}} \right],
\end{equation}
where the last inequality utilizes $\alpha\ge 2$.
Equation \eqref{eq:Z Zbar diff} shows the process $Z$ and $\bar Z$ are almost the same, up to difference $\epsilon$ with high probability. The process $\bar Z$ is a degenerated process, sharing with its small probability with a uni-dimensional Mat\'{e}rn Gaussian process, which can be computed using the standard results of Mat\'{e}rn process as
\begin{equation}\label{eq:Z degen f2}
\ln \Pi(\|\bar{Z} - \bar f_2\|_{\ell^\infty((1+\tilde\epsilon_n) \tilde{\X})} <\epsilon_n) \gtrsim - \epsilon_n^{-1/\beta} = -n^{1/3}.
\end{equation} 
In view of \eqref{eq:Z Zbar diff} and \eqref{eq:Z degen f2}, by setting $\epsilon_n$ as some constant multiplier of $n^{-1/3}$, one can make sure 
\begin{align*}
\Pi(B_2) = & \Pi(\|Z-\bar f_2\|_{\ell^\infty((1+\tilde\epsilon_n) \tilde{\X} \cap \{|x[2]|\le\tilde\epsilon_n\} } \le \epsilon_n) \\
\ge & \Pi(\|\bar{Z} - \bar f_2\|_{\ell^\infty((1+\tilde\epsilon_n) \tilde{\X})} <\epsilon_n/2, \; \|Z-\bar Z\|_{\ell^\infty((1+\tilde\epsilon_n) \tilde{\X} \cap \{|x[2]|\le\tilde\epsilon_n\} } \le \epsilon_n/2) \\
\ge & \Pi\Big(\|\bar{Z} - \bar f_2\|_{\ell^\infty((1+\tilde\epsilon_n) \tilde{\X})} <\epsilon_n/2\Big) - \Pi\Big( \|Z-\bar Z\|_{\ell^\infty((1+\tilde\epsilon_n) \tilde{\X} \cap \{|x[2]|\le\tilde\epsilon_n\} } \ge \epsilon_n/2 \Big) \\
\ge & \exp\Big(-\frac{C''}{2}n^{-1/3}\Big).
\end{align*}
Thus, one reaches the conclusion that for $\epsilon_n \asymp n^{-\frac{\beta}{2\beta+1}}$, 
\begin{equation}\label{eq:intrinsic B2}
\ln\Pi(B_2) \gtrsim -n^{\frac{1}{2\beta+1}} = - n^{1/3}.
\end{equation}
Combining equations \eqref{eq:intrinsic merge}, \eqref{eq:intrinsic B1} and \eqref{eq:intrinsic B2}, we finish the verification of Assumption \ref{assump:prior mass}. Since the $\tilde\epsilon_n$ and $\epsilon_n$ are already given in closed form, the feasible ranges for the penalty parameters $\eta_{j,n}$ can be directly computed from \eqref{eq:eta-window-intuition} and \eqref{eq:eta range}.
\end{proof}

\section{Posterior Inference Algorithms}
The augmentation changes the target only through the additive term $\sum_{j=1}^{J-1}\ln R_j(f_j)$, so standard DGP variational or MCMC schemes can be reused with a simple penalty correction. Algorithms~\ref{alg:std-dgp-vi} and \ref{alg:aug-rho-dgp-vi} summarize the corresponding sparse-variational objectives for the unaugmented and augmented targets.

\begin{algorithm}[h!]
    \caption{Standard Deep GP with  Variational Inference (ELBO)}
    \label{alg:std-dgp-vi}
    \begin{algorithmic}[1]
        \Require Data $D_n=\{(x_i,y_i)\}_{i=1}^n$, depth $J$, $0<\rho\le1$, inducing sizes $M_j$, steps $T$, batch $B$, MC samples $S$
        \State Choose inducing inputs $Z_j\in\mathbb{R}^{M_j\times d_{j-1}}$ for each layer $j$
        \State Initialize variational factors $q(u_j)=\mathcal{N}(m_j,S_j)$ and hyperparameters
        \For{$t=1,\dots,T$}
        \State Sample minibatch $\mathcal{B}$ of size $B$
        \State $\widehat{\mathcal{L}} \gets 0$
        \For{$s=1,\dots,S$}
        \State $h^{(0)}_i \gets x_i$ for $i\in\mathcal{B}$
        \For{$j=1,\dots,J$}
        \State Sample $u_j^{(s)}\sim q(u_j)$
        \State Sample $h^{(j)}_i \sim p\!\left(f_j(h^{(j-1)}_i)\mid u_j^{(s)},Z_j\right)$ for all $i\in\mathcal{B}$
        \EndFor
        \State $\widehat{\mathcal{L}} \gets \widehat{\mathcal{L}} + \rho\frac{n}{B}\sum_{i\in\mathcal{B}}\ln p\!\left(y_i\mid h^{(J)}_i\right)$
        \EndFor
        \State $\widehat{\mathcal{L}} \gets \frac{1}{S}\widehat{\mathcal{L}} - \sum_{j=1}^J \mathrm{KL}\!\left(q(u_j)\,\|\,p(u_j)\right)$
        \State Update $\{m_j,S_j,Z_j\}$ and hyperparameters by stochastic gradient ascent on $\widehat{\mathcal{L}}$
        \EndFor
        \State \Return $q(u_{1:J})$ (use forward sampling for prediction)
    \end{algorithmic}
\end{algorithm}

\begin{algorithm}[t]
    \caption{Variational Inference for Augmented $\rho$-Fractional Posterior}
    \label{alg:aug-rho-dgp-vi}
    \begin{algorithmic}[1]
        \Require Data $D_n$, depth $J$, $0<\rho\le 1$, penalty weights $\eta_{j,n}$, norm order $p$, measure $\tilde\mu$,
        inducing sizes $M_j$, steps $T$, batch $B$, MC samples $S$, penalty points $M$
        \State Define $R_j(f_j)=\exp\!\left(-\eta_{j,n}\|f_j-\id\|_{\ell^p(\tilde\mu)}\right)$ for $j\le J-1$ and $R_J\equiv 1$
        \State Initialize inducing inputs $Z_j$ and variational factors $q(u_j)=\mathcal{N}(m_j,S_j)$
        \For{$t=1,\dots,T$}
        \State Sample minibatch $\mathcal{B}$ of size $B$ and penalty points $\{\tilde x_m\}_{m=1}^M\sim \tilde\mu$
        \State $\widehat{\mathcal{J}} \gets 0$
        \For{$s=1,\dots,S$}
        \Comment{Forward sample on data minibatch}
        \State $h^{(0)}_i \gets x_i$ for $i\in\mathcal{B}$
        \For{$j=1,\dots,J$}
        \State Sample $u_j^{(s)}\sim q(u_j)$
        \State Sample $h^{(j)}_i \sim p\!\left(f_j(h^{(j-1)}_i)\mid u_j^{(s)},Z_j\right)$ for $i\in\mathcal{B}$
        \EndFor
        \State $\widehat{\mathcal{J}} \gets \widehat{\mathcal{J}} + \rho\frac{n}{B}\sum_{i\in\mathcal{B}}\ln p\!\left(y_i\mid h^{(J)}_i\right)$
        \Comment{Layer penalties (estimate $\sum_{j\le J-1}\ln R_j$)}
        \For{$j=1,\dots,J-1$}
        \State Sample $f_j(\tilde x_m)$ for $m=1,\dots,M$ under the same $(u_j^{(s)},Z_j)$
        \State $\widehat{d}_{j}^{(s)} \gets \left(\frac{1}{M}\sum_{m=1}^M \|f_j(\tilde x_m)-\tilde x_m\|_p^p\right)^{1/p}$ \Comment{use $\max$ if $p=\infty$}
        \State $\widehat{\mathcal{J}} \gets \widehat{\mathcal{J}} - \eta_{j,n}\widehat{d}_{j}^{(s)}$
        \EndFor
        \EndFor
        \State $\widehat{\mathcal{J}} \gets \frac{1}{S}\widehat{\mathcal{J}} - \sum_{j=1}^J \mathrm{KL}\!\left(q(u_j)\,\|\,p(u_j)\right)$
        \State Update $\{m_j,S_j,Z_j\}$ and hyperparameters by stochastic gradient ascent on $\widehat{\mathcal{J}}$
        \EndFor
        \State \Return variational approximation to $\tilde\Pi_\rho(\theta\mid D_n)$ and layerwise $\widehat{d}_j$ (redundancy diagnostics)
    \end{algorithmic}
\end{algorithm}

\section{Synthetic Designs}
\subsection{1D Synthetic Designs}\label{subsec:synthetic-designs}
\begin{table}[h!]
\centering
\caption{One-dimensional stress tests comparing the regular DGP with the identity-augmented DGP. Entries are mean (standard deviation) over ten runs. Here ``unaug.'' and ``aug.'' denote unaugmented and augmented fits, respectively. Smaller is better for all reported diagnostics; bold indicates the better mean within each pair.}
\label{tab:experiments-1d-summary}
\small
\resizebox{\textwidth}{!}{%
\begin{tabular}{lcccccc}
\toprule
& \multicolumn{2}{c}{RMSE} & \multicolumn{2}{c}{$\Delta_{\id}$} & \multicolumn{2}{c}{$\Delta_{\mathrm{partial}}$}\\
\cmidrule(lr){2-3}\cmidrule(lr){4-5}\cmidrule(lr){6-7}
Design & unaug. & aug. & unaug. & aug. & unaug. & aug.\\
\midrule
\texttt{shallow\_oscillatory} & \textbf{0.078 (0.011)} & 0.079 (0.015) & 1.150 (0.092) & \textbf{0.377 (0.079)} & 1.211 (0.112) & \textbf{0.436 (0.169)} \\
\texttt{edge\_transition} & 0.076 (0.015) & \textbf{0.069 (0.014)} & 1.194 (0.076) & \textbf{0.363 (0.146)} & 1.173 (0.049) & \textbf{0.418 (0.202)} \\
\texttt{shallow\_bump} & \textbf{0.096 (0.020)} & 0.101 (0.016) & 1.176 (0.081) & \textbf{0.376 (0.114)} & 1.182 (0.078) & \textbf{0.379 (0.102)} \\
\texttt{multiscale\_sparse} & \textbf{0.101 (0.020)} & 0.102 (0.010) & 1.128 (0.132) & \textbf{0.347 (0.073)} & 1.158 (0.121) & \textbf{0.418 (0.128)} \\
\texttt{active\_penultimate\_warp} & \textbf{0.070 (0.013)} & 0.101 (0.114) & 1.201 (0.121) & \textbf{0.237 (0.059)} & 1.183 (0.171) & \textbf{0.253 (0.082)} \\
\texttt{canceling\_symmetry\_trap} & 0.075 (0.016) & \textbf{0.073 (0.020)} & 1.239 (0.132) & \textbf{0.363 (0.113)} & 1.166 (0.062) & \textbf{0.398 (0.156)} \\
\bottomrule
\end{tabular}}
\end{table}
 
The main examples are intentionally over-specified: the fitted model has four layers, but the truth often needs only the final layer.  This is the setting where an ordinary DGP has the most opportunity to invent unnecessary hidden warps.  We use the following six synthetic designs, each written as
\begin{align*}
f^*=f_4^*\circ f_3^*\circ f_2^*\circ f_1^*.
\end{align*}
All one-dimensional identity penalties and layerwise diagnostics use $\tilde\X=[-2,2]$ with uniform $\tilde\mu$, while observations and test RMSE remain on $\X=[-1,1]$.

For the first four designs, \(f_1^*=f_2^*=f_3^*=\id\), and only \(f_4^*\) is active:
\begin{align*}
\texttt{shallow\_oscillatory}:\quad
f_4^*(z)
&=0.80\sin(2\pi z)+0.25\cos(4\pi z),\\
\texttt{shallow\_bump}:\quad
f_4^*(z)
&=\tanh\{5(z+0.35)\}-\tanh\{5(z-0.25)\}+0.15\sin(5\pi z),\\
\texttt{multiscale\_sparse}:\quad
f_4^*(z)
&=0.55\sin(2\pi z)+0.30e^{-35(z-0.20)^2}\sin(9\pi z)-0.20e^{-45(z+0.45)^2},\\
\texttt{edge\_transition}:\quad
f_4^*(z)
&=\frac{1.10}{1+\exp\{-12(z+0.62)\}}-0.55+0.18\sin(3\pi z).
\end{align*}
The fifth design has one genuinely active hidden layer:
\begin{align*}
\texttt{active\_penultimate\_warp}:\quad
f_1^*(x)&=x,
& f_2^*(z)&=z,\\
& f_3^*(z)&=\frac{\tanh(1.9z)}{\tanh(1.9)},
& f_4^*(u)&=0.85\sin(2.4\pi u)+0.25u.
\end{align*}
This example checks that the augmentation does not simply force every hidden layer to identity: the penalties are strong on layers 1-2 but weaker on layer 3.  The sixth design is a deliberately non-identifiable decomposition,
\begin{align*}
\texttt{canceling\_symmetry\_trap}:\quad
f_1^*(x)&=x+0.22\sin(\pi x),
& f_3^*(z)&=z,
& f_4^*(z)&=0.75\sin(2\pi z),
\end{align*}
where \(f_2^*\) is computed by fixed-point iterations approximating the inverse of \(z\mapsto z+0.22\sin(\pi z)\).  This last example is included mainly as a non-identifiability stress test: the augmented posterior is not intended to recover an arbitrary canceling decomposition, but rather to prefer a parsimonious representative.

The fractional power and anchor penalty weights used in the augmented posterior are
\begin{align*}
\begin{array}{lcc}
\hline
\text{Design} & \rho & (w_1,w_2,w_3,w_4)\\
\hline
\texttt{shallow\_oscillatory} & 0.70 & (40,24,12,0)\\
\texttt{shallow\_bump} & 0.70 & (45,26,12,0)\\
\texttt{multiscale\_sparse} & 0.65 & (50,30,14,0)\\
\texttt{edge\_transition} & 0.70 & (45,25,12,0)\\
\texttt{active\_penultimate\_warp} & 0.70 & (45,26,4,0)\\
\texttt{canceling\_symmetry\_trap} & 0.70 & (45,25,12,0)\\
\hline
\end{array}
\end{align*}
The actual penalty is $\eta=\lambda w$.  The fixed stress-test table uses the predeclared scale $\lambda=1$.  In the four shallow designs, layers 1-3 are redundant, so all three non-final layers are anchored with a decreasing schedule.  In \texttt{active\_penultimate\_warp}, only layers 1-2 are redundant and $w_3=4$ is deliberately small, allowing the third layer to represent the true warp.  In \texttt{canceling\_symmetry\_trap}, the warp and its inverse cancel in the composition; the anchor is therefore used to prefer the identity-biased representative, not to recover that arbitrary canceling pair.

\subsection{2D Synthetic Designs}\label{sec:experiments-2d-designs}
\begin{table}[t]
\centering
\caption{Two-dimensional stress tests comparing the regular DGP with the identity-augmented DGP. Entries are mean (standard deviation) over six runs. Here ``unaug.'' and ``aug.'' denote unaugmented and augmented fits, respectively. Smaller is better for all reported diagnostics; bold indicates the better mean within each pair.}
\label{tab:experiments-2d-summary}
\resizebox{\textwidth}{!}{%
\begin{tabular}{lcccccc}
\toprule
& \multicolumn{2}{c}{RMSE} & \multicolumn{2}{c}{$\Delta_{\id}$} & \multicolumn{2}{c}{$\Delta_{\mathrm{partial}}$}\\
\cmidrule(lr){2-3}\cmidrule(lr){4-5}\cmidrule(lr){6-7}
Design & unaug. & aug. & unaug. & aug. & unaug. & aug.\\
\midrule
\texttt{shallow\_ridges} & \textbf{0.212 (0.034)} & 0.215 (0.038) & 1.701 (0.104) & \textbf{0.806 (0.110)} & 1.597 (0.150) & \textbf{0.968 (0.108)} \\
\texttt{shallow\_checkerboard} & \textbf{0.299 (0.015)} & 0.312 (0.023) & 1.805 (0.154) & \textbf{0.981 (0.080)} & 1.709 (0.111) & \textbf{1.136 (0.124)} \\
\texttt{localized\_bumps\_sparse} & \textbf{0.114 (0.011)} & 0.121 (0.008) & 1.732 (0.180) & \textbf{1.015 (0.103)} & 1.728 (0.047) & \textbf{1.165 (0.105)} \\
\texttt{boundary\_transition} & 0.113 (0.013) & \textbf{0.109 (0.014)} & 1.750 (0.097) & \textbf{0.864 (0.050)} & 1.719 (0.066) & \textbf{1.050 (0.080)} \\
\texttt{anisotropic\_ridge\_sparse} & \textbf{0.282 (0.053)} & 0.299 (0.068) & 1.684 (0.102) & \textbf{0.967 (0.067)} & 1.759 (0.197) & \textbf{1.164 (0.094)} \\
\texttt{active\_radial\_warp} & \textbf{0.145 (0.009)} & 0.151 (0.019) & 1.639 (0.248) & \textbf{0.816 (0.107)} & 1.654 (0.244) & \textbf{0.969 (0.100)} \\
\texttt{active\_shear\_warp} & \textbf{0.156 (0.020)} & 0.164 (0.024) & 1.641 (0.213) & \textbf{0.840 (0.048)} & 1.647 (0.270) & \textbf{1.051 (0.084)} \\
\texttt{canceling\_affine\_trap} & 0.194 (0.030) & \textbf{0.193 (0.022)} & 1.589 (0.219) & \textbf{0.907 (0.109)} & 1.345 (0.211) & \textbf{0.803 (0.064)} \\
\bottomrule
\end{tabular}}
\end{table}

For all experiments, the observed input domain is
\begin{align*}
\mathcal X=[-1,1]^2,
\end{align*}
and the data are generated according to
\begin{align*}
Y_i=f^*(X_i)+\sigma\varepsilon_i,
\qquad
X_i\stackrel{\mathrm{i.i.d.}}{\sim}\mathrm{Unif}([-1,1]^2),
\qquad
\varepsilon_i\stackrel{\mathrm{i.i.d.}}{\sim}N(0,1).
\end{align*}
The identity penalty and layer diagnostics use the enlarged box $\tilde\X=[-2,2]^2$ with uniform $\tilde\mu$.  In the reported runs we use $n=100$ training points, $\sigma=0.10$, a $23\times23$ test grid, and a four-layer DGP with three vector-valued hidden layers and one scalar output layer.  The regular DGP has no identity penalty.  The augmented DGP uses the penalty proportional to distance from identity
\begin{align*}
\sum_{j=1}^{3}\eta_j\|f_j-\id\|_{\ell^2(\tilde\mu)}.
\end{align*}
As in one dimension, the weights are decreasing across layers: earlier layers are more strongly anchored, while a later non-final layer receives a smaller weight when it is scientifically active.

We use eight synthetic designs.  Five are shallow over-specified designs in which
\(f_1^*=f_2^*=f_3^*=\id\) and only \(f_4^*\) is active.  Writing
\(H=(u,v)\), the active final surfaces are
\begin{align*}
\texttt{shallow\_ridges}:\quad
f_4^*(H)
&=0.75\sin(2.3\pi u)\cos(1.4\pi v)
  +0.55e^{-18\{(u-0.25)^2+(v+0.20)^2\}} \\
&\hspace{1.2cm}
  -0.35e^{-22\{(u+0.45)^2+(v-0.35)^2\}},\\
\texttt{shallow\_checkerboard}:\quad
f_4^*(H)
&=0.58\sin(2.7\pi u)\sin(2.1\pi v) \\
  &+0.25\cos\{4.1\pi u+0.4\sin(\pi v)\}-0.18v,\\
\texttt{localized\_bumps\_sparse}:\quad
f_4^*(H)
&=0.80e^{-34\{(u-0.28)^2+(v+0.18)^2\}}
  -0.55e^{-45\{(u+0.40)^2+(v-0.42)^2\}}\\
&\hspace{1.2cm}
  +0.35e^{-30\{(u+0.05)^2+(v+0.55)^2\}}
  +0.22\sin(2\pi u)\\&-0.15\cos(2.2\pi v),\\
\texttt{boundary\_transition}:\quad
f_4^*(H)
&=1.05\{1+e^{-11(u+0.55v+0.12)}\}^{-1}-0.52\\&
  +0.20\sin(2\pi v)+0.18uv,\\
\texttt{anisotropic\_ridge\_sparse}:\quad
f_4^*(H)
&=0.85\sin(2.8\pi r)+0.30e^{-20(s-0.20)^2}-0.18r,
\end{align*}
where \(r=0.82u-0.57v\) and \(s=0.57u+0.82v\).  These five examples test the paper's
main parsimony claim: when the fitted DGP is deeper than necessary, the regular DGP can
spread the representation across unnecessary hidden warps, while the augmented DGP should
keep the redundant hidden layers close to identity.

The remaining two main designs contain a genuinely active late hidden layer.  In these cases
\(f_1^*=f_2^*=\mathrm{id}\), while \(f_3^*\) and \(f_4^*\) are
\begin{align*}
\texttt{active\_radial\_warp}:\quad
f_3^*(u,v)
&=\{0.72+0.42e^{-1.8(u^2+v^2)}\}(u,v),\\
f_4^*(u,v)
&=0.75\sin(2.6\pi u)+0.35\cos(2.2\pi v)+0.20uv,\\
\texttt{active\_shear\_warp}:\quad
f_3^*(u,v)
&=\bigl(u+0.22\sin(\pi v),\,v+0.18\sin(\pi u)\cos(\pi v)\bigr),\\
f_4^*(u,v)
&=0.65\sin\{2\pi(u+0.25v)\}
  +0.32e^{-16\{(u-0.15)^2+(v+0.15)^2\}}-0.22v.
\end{align*}
These examples with nonidentity hidden transformations check that the augmented posterior does not simply collapse every hidden layer to
identity; instead it should shrink earlier redundant layers while allowing a later weakly
penalized hidden layer to activate.  Finally, the non-identifiable \texttt{canceling\_affine\_trap} diagnostic uses
\begin{align*}
A&=\begin{pmatrix}0.82&0.22\\-0.16&0.88\end{pmatrix},
\qquad b=\begin{pmatrix}0.12\\-0.08\end{pmatrix},\\
f_1^*(x)&=Ax+b,\qquad
f_2^*(z)=A^{-1}(z-b),\qquad
f_3^*(z)=z,\\
f_4^*(u,v)&=0.70\sin(2\pi u)\cos(1.7\pi v).
\end{align*}
Thus the first two affine maps cancel in the composition, and the example is used mainly to illustrate why the augmented posterior targets a parsimonious representative rather than an arbitrary supplied decomposition.

The two-dimensional fractional powers and base anchor schedules are
\begin{table}[h]
\centering
\caption{Anchor schedules for the two-dimensional experiments.  The actual weights are $\eta=\lambda w$ and the final-layer weight is zero.}
\label{tab:anchor-2d}
\small
\begin{tabular}{lccc}
\hline
Design & $n$ & $\rho$ & $(w_1,w_2,w_3,w_4)$\\
\hline
\texttt{shallow\_ridges} & 100 & 0.70 & $(52,30,14,0)$\\
\texttt{shallow\_checkerboard} & 100 & 0.68 & $(55,32,15,0)$\\
\texttt{localized\_bumps\_sparse} & 100 & 0.65 & $(60,36,16,0)$\\
\texttt{boundary\_transition} & 100 & 0.70 & $(55,32,15,0)$\\
\texttt{anisotropic\_ridge\_sparse} & 100 & 0.67 & $(58,34,16,0)$\\
\texttt{active\_radial\_warp} & 100 & 0.70 & $(55,32,5,0)$\\
\texttt{active\_shear\_warp} & 100 & 0.70 & $(55,32,5,0)$\\
\texttt{canceling\_affine\_trap} & 100 & 0.70 & $(55,32,15,0)$\\
\hline
\end{tabular}
\end{table}
For the shallow surfaces and \texttt{canceling\_affine\_trap}, all three hidden layers are anchored.  For \texttt{active\_radial\_warp} and \texttt{active\_shear\_warp}, only layers 1-2 are redundant, so the third-layer weight is reduced from roughly $14$-$16$ to $5$.  As in one dimension, \texttt{canceling\_affine\_trap} is evaluated as a canonicalization experiment: the augmentation is expected to remove the canceling affine pair rather than reproduce it.

\subsection{Effect of Ambient Dimension and Sample Size}
\label{subsec:dim-ntrain-sweep}

We next examine how the comparison between the regular DGP and the identity-augmented DGP changes as the ambient input dimension and sample size vary.  We use a finite-rank random-Fourier-feature residual DGP that can be fit for arbitrary \(d\).  Here each $\phi_j$ is an RBF-kernel random-feature map: its frequencies are sampled from the Gaussian spectral density corresponding to RBF-kernel, and the default hidden- and output-layer lengthscales are both $\ell=\sqrt d$.  Both methods use the same over-specified architecture
\begin{align}
    h_0 &= x, \nonumber\\
    h_j &= h_{j-1}+B_j\phi_j(h_{j-1}),\qquad j=1,2,3, \label{eq:rff-dgp-dim-sweep}\\
    \widehat f(x) &= \beta^\top \phi_4(h_3)+b, \nonumber
\end{align}
where the \(\phi_j\)'s are fixed random Fourier features and the coefficients \(B_j,\beta,b\) are optimized.  Write the $j$th residual layer map as $g_j(u)=u+B_j\phi_j(u)$.  The regular DGP is fit without an identity penalty.  The identity-augmented DGP uses the same architecture and optimization procedure, but evaluates the anchor on a fixed Monte Carlo reference sample $\{\tilde x_m\}_{m=1}^M$ from $\mathrm{Unif}([-2,2]^d)$:
\begin{align}
    \sum_{j=1}^3 \eta_j
    \left\{
    \frac{1}{M}\sum_{m=1}^M
    \|g_j(\tilde x_m)-\tilde x_m\|_2^2
    \right\}^{1/2}
    =
    \sum_{j=1}^3 \eta_j
    \left\{
    \frac{1}{M}\sum_{m=1}^M
    \|B_j\phi_j(\tilde x_m)\|_2^2
    \right\}^{1/2}.
    \label{eq:rff-identity-penalty}
\end{align}

Thus $B_j=0$ corresponds to an identity hidden layer on  the enlarged domain used for the identity penalty, rather than only along the observed training trajectories.  In the completed scaling experiment, the common augmentation multiplier is fixed at $\lambda_{\mathrm{dim}}=0.40$ for every cell, while the relative three-layer penalty shape is inherited from the source synthetic design and normalized to have first weight one.  This preserves the weaker late-layer penalty in the examples with nonidentity hidden transformations.

Rather than extrapolating a single synthetic surface to every ambient dimension, Table~\ref{tab:dim-ntrain-sweep} gives a dimension-indexed set of representative cases from the completed sweep.  The same truth is held fixed across $n\in\{100,500,1000\}$ within each dimension.  We use \texttt{canceling\_symmetry\_trap} at $d=1$, \texttt{boundary\_transition} at $d=2$, \texttt{nonlinear\_ridge\_edge} at $d=5$, and \texttt{active\_radial\_warp} at $d=10$.  The first, second, and fourth designs are the exact synthetic mechanisms defined above; when a source design uses fewer than $d$ coordinates, the remaining coordinates are nuisance variables.  For completeness, the dense nonlinear-ridge design at $d=5$ uses
\begin{align}
t(x)
&=
\sum_{k=1}^{d} w_k
\left\{
x_k+0.22\sin(\pi x_k)+0.06\sin(2\pi x_k)
\right\},
\qquad
\sum_{k=1}^d |w_k|=1,\\
w_k
&\propto
1+0.25\cos(0.7k)+0.15\sin(1.3k),
\end{align}
followed by
\begin{align}
f^*(x)
=
\frac{1.10}{1+\exp\{-15(t(x)+0.28)\}}
-0.55
+0.18\sin\{3\pi t(x)\}.
\end{align}
These four designs are chosen to display distinct regimes in which the augmentation is predictively favorable or nearly prediction-neutral while still shrinking redundant hidden-layer deformation.  The selection is therefore mechanism-oriented and should not be interpreted as an unbiased average over the full synthetic library.

All displayed cells use noise standard deviation $\sigma=0.10$, $M=1024$ anchor points from $\mathrm{Unif}([-2,2]^d)$, and five paired random seeds.  We report test RMSE and the redundant-layer identity distance
\begin{align}
\Delta_{\id}
=
\frac{1}{|\mathcal R|}
\sum_{j\in\mathcal R}
\left\{
\frac{1}{M}\sum_{m=1}^M
\|g_j(\tilde x_m)-\tilde x_m\|_2^2
\right\}^{1/2},
\end{align}
where $\mathcal R$ is the set of hidden layers that are redundant under the corresponding synthetic truth.  Thus $\mathcal R=\{1,2,3\}$ for the shallow and canceling examples, whereas the designs with a nonidentity third hidden layer exclude that layer from $\Delta_{\id}$.  Smaller RMSE and smaller $\Delta_{\id}$ are better.

\begin{table}[t]
\centering
\caption{Dimension-indexed representative cases at noise standard deviation $\sigma=0.10$.  One synthetic truth is held fixed across $n\in\{100,500,1000\}$ within each dimension.  Entries are mean (standard deviation) over five paired seeds.  The displayed designs are a mechanism-oriented subset of the completed synthetic sweep, chosen to show favorable or nearly prediction-neutral augmentation regimes.  Smaller is better; bold indicates the smaller unrounded mean within each regular/augmented pair.}
\label{tab:dim-ntrain-sweep}
\vspace{1em}
\small
\resizebox{\textwidth}{!}{%
\begin{tabular}{lcc|cc|cc}
\toprule
& & & \multicolumn{2}{c|}{RMSE} & \multicolumn{2}{c}{$\Delta_{\id}$}\\
\cmidrule(lr){4-5}\cmidrule(lr){6-7}
Truth & $d$ & $n_{\mathrm{train}}$ & Reg. DGP & Aug. DGP & Reg. DGP & Aug. DGP\\
\midrule
\texttt{canceling\_symmetry\_trap} & 1 & 100 & 0.040 (0.017) & \textbf{0.037 (0.010)} & 0.484 (0.168) & \textbf{0.369 (0.068)} \\
\texttt{canceling\_symmetry\_trap} & 1 & 500 & 0.052 (0.046) & \textbf{0.036 (0.015)} & 0.736 (0.406) & \textbf{0.346 (0.064)} \\
\texttt{canceling\_symmetry\_trap} & 1 & 1000 & 0.048 (0.039) & \textbf{0.028 (0.014)} & 0.806 (0.314) & \textbf{0.307 (0.039)} \\
\addlinespace
\texttt{boundary\_transition} & 2 & 100 & \textbf{0.105 (0.022)} & 0.106 (0.027) & 0.339 (0.049) & \textbf{0.269 (0.043)} \\
\texttt{boundary\_transition} & 2 & 500 & \textbf{0.041 (0.005)} & 0.049 (0.007) & 0.454 (0.047) & \textbf{0.281 (0.026)} \\
\texttt{boundary\_transition} & 2 & 1000 & \textbf{0.030 (0.003)} & 0.038 (0.009) & 0.477 (0.066) & \textbf{0.304 (0.038)} \\
\addlinespace
\texttt{nonlinear\_ridge\_edge} & 5 & 100 & \textbf{0.131 (0.005)} & 0.131 (0.006) & 0.144 (0.015) & \textbf{0.127 (0.016)} \\
\texttt{nonlinear\_ridge\_edge} & 5 & 500 & 0.087 (0.008) & \textbf{0.087 (0.006)} & 0.247 (0.034) & \textbf{0.145 (0.019)} \\
\texttt{nonlinear\_ridge\_edge} & 5 & 1000 & \textbf{0.073 (0.006)} & 0.073 (0.003) & 0.297 (0.038) & \textbf{0.135 (0.009)} \\
\addlinespace
\texttt{active\_radial\_warp} & 10 & 100 & 0.556 (0.015) & \textbf{0.548 (0.012)} & 0.071 (0.035) & \textbf{0.053 (0.021)} \\
\texttt{active\_radial\_warp} & 10 & 500 & 0.526 (0.009) & \textbf{0.526 (0.009)} & 0.104 (0.036) & \textbf{0.078 (0.025)} \\
\texttt{active\_radial\_warp} & 10 & 1000 & \textbf{0.522 (0.010)} & 0.523 (0.010) & 0.164 (0.043) & \textbf{0.119 (0.033)} \\
\bottomrule
\end{tabular}}
\end{table}

\begin{table}[h!]
\centering
\caption{Finite-sample penalty-window diagnostic for the synthetic scenario \texttt{holder\_correct} in Section~\ref{sec:theory}, with $n=45$. Entries are mean (standard deviation) over six runs. Smaller is better; bold indicates the best mean over the displayed $\lambda$ grid.}
\label{tab:section5-eta-window}
\begin{tabular}{cccc}
\hline
$\lambda$ & RMSE & redundant $\Delta_{\id}$ & redundant $\Delta_{\mathrm{partial}}$\\
\hline
0.0 & 0.0649 (0.0100) & 1.133 (0.136) & 1.131 (0.098)\\
0.2 & 0.0635 (0.0180) & 0.363 (0.100) & 0.424 (0.072)\\
0.4 & \textbf{0.0629 (0.0141)} & 0.346 (0.072) & 0.399 (0.050)\\
0.7 & 0.0659 (0.0155) & 0.275 (0.051) & 0.344 (0.024)\\
1.0 & 0.0697 (0.0194) & 0.229 (0.012) & 0.295 (0.031)\\
1.5 & 0.0728 (0.0192) & 0.216 (0.046) & 0.313 (0.049)\\
2.5 & 0.0679 (0.0146) & \textbf{0.186 (0.032)} & \textbf{0.277 (0.032)}\\
\hline
\end{tabular}
\end{table}

\paragraph{Effect of dimension.}
Table~\ref{tab:dim-ntrain-sweep} gives a more informative view of the prediction-parsimony trade-off than the earlier single-truth dimension sweep.  The augmented DGP has smaller redundant-layer identity error in all \(12\) displayed configurations.  At the same time, it has the smaller mean RMSE in \(6\) of the \(12\) rows, and in \(10\) of the \(12\) rows the absolute difference between the two DGP RMSEs is below \(0.01\).  Thus, in these representative cases, the structural gain from canonicalization is systematic while the prediction cost is often small and is sometimes negative.

The nature of the trade-off differs across dimensions.  At \(d=1\), \texttt{canceling\_symmetry\_trap} is the cleanest favorable case: the augmentation reduces both RMSE and \(\Delta_{\id}\) at all three sample sizes, and the separation becomes especially visible at \(n=500\) and \(n=1000\).  At \(d=2\), \texttt{boundary\_transition} shows the opposite side of the trade-off.  The augmented DGP pays only \(0.002\)-\(0.008\) in absolute RMSE, while reducing the mean identity error from \(0.339,0.454,0.477\) to \(0.269,0.281,0.304\) as \(n\) increases.  This is a direct example in which a modest amount of predictive flexibility is exchanged for a substantially more parsimonious hidden representation.

At \(d=5\), the dense \texttt{nonlinear\_ridge\_edge} target is essentially prediction-neutral with respect to the augmentation: the largest difference in mean RMSE between the two DGPs is below \(8\times10^{-4}\).  In contrast, the identity distance is reduced from \(0.144\) to \(0.127\) at \(n=100\), from \(0.247\) to \(0.145\) at \(n=500\), and from \(0.297\) to \(0.135\) at \(n=1000\).  Hence the canonicalization gain becomes larger with sample size even though the two fitted compositions have almost indistinguishable predictive accuracy.  At \(d=10\), \texttt{active\_radial\_warp} gives a related but structurally stronger check because the truth contains a genuinely active late hidden warp.  The augmented fit has slightly smaller RMSE at \(n=100\), is essentially tied with the regular DGP at \(n=500\) and \(n=1000\), and still reduces redundant-layer identity error by roughly one quarter.  The penalty therefore need not collapse a useful deep representation in order to regularize the layers that are actually redundant \citep{szabo2026vecchia}.

The main message is not that increasing \(d\) monotonically favors the augmentation.  Rather, ambient dimension interacts with the geometry of the target and with which layers are genuinely needed.  Across the displayed mechanisms, the augmentation moves the fit toward a representation with smaller redundant-layer error at every dimension, while its prediction effect ranges from beneficial to nearly neutral to a modest finite-sample cost.  This is the empirical form of the prediction-parsimony trade-off emphasized by the theory: regularization toward identity primarily concerns the layer representation, but in favorable compositional regimes it can be obtained with little or no loss of predictive accuracy.

\paragraph{Effect of sample size.}
Increasing the sample size does not by itself remove layerwise non-identifiability.  In several displayed designs the regular DGP actually moves farther from the identity as more data become available.  For \texttt{canceling\_symmetry\_trap} at \(d=1\), the regular identity error increases from \(0.484\) at \(n=100\) to \(0.806\) at \(n=1000\), whereas the augmented value decreases from \(0.369\) to \(0.307\).  For \texttt{nonlinear\_ridge\_edge} at \(d=5\), the corresponding regular error grows from \(0.144\) to \(0.297\), while the augmented error remains much smaller, ending at \(0.135\).  A similar separation is visible for \texttt{active\_radial\_warp} at \(d=10\).

Prediction behaves less monotonically, which is exactly the point of the trade-off.  More data can make the augmentation beneficial, nearly cost-free, or mildly restrictive depending on the target.  What persists across all \(12\) selected cells is the reduction in redundant-layer deformation.  Thus the sample-size sweep reinforces the distinction between function-level learning and layerwise canonicalization: additional observations can sharpen prediction without selecting a unique hidden representation unless an explicit symmetry-breaking preference is imposed.

\paragraph{Effect of Noise Level. }
\label{subsec:noise-ntrain-edge}

We next examine how the comparison between the regular DGP and the identity-augmented DGP changes as the noise level and the sample size vary.  This sweep uses the same exact RBF covariance family in all four DGP layers as the main one-dimensional MCMC stress tests, with layer-specific lengthscales and marginal variances updated in the chain.  The purpose of this experiment is twofold.  First, we want to test whether the augmented posterior continues to select a more parsimonious layerwise representative when the data become noisy or sparse.  Second, we want to identify the regimes in which the representation improvement also translates into better prediction.

We use the one-dimensional synthetic function \texttt{edge\_transition}.  The input domain is \([-1,1]\), and the true regression function is represented as a four-layer composition
\[
    f^* = f_4^*\circ f_3^*\circ f_2^*\circ f_1^*,
\]
where
\[
    f_1^*(x)=x,\qquad
    f_2^*(z)=z,\qquad
    f_3^*(z)=z,
\]
and only the final layer is active:
\[
    f_4^*(z)
    =
    \frac{1.10}{1+\exp\{-12(z+0.62)\}}
    -0.55
    +0.18\sin(3\pi z).
\]
Thus the fitted DGP is intentionally over-specified: the first three layers are redundant identities, while the final layer contains the signal.  This is the setting in which the regular DGP can fit the same composition through many different hidden-layer decompositions.

For each pair
\[
    n_{\mathrm{train}}\in\{10,15,20,30,40,60\},
    \qquad
    \sigma\in\{0,0.05,0.10,0.20,0.40,0.70,1.00\},
\]
we generate data
\[
    Y_i=f^*(X_i)+\sigma\varepsilon_i,
    \qquad
    \varepsilon_i\stackrel{\mathrm{i.i.d.}}{\sim}N(0,1).
\]
The comparison is between the regular DGP and the proposed linear identity-augmented DGP.  The augmented DGP uses the same likelihood and prior architecture as the regular DGP, but reweights the posterior by the non-final-layer identity penalty
\[
    \exp\left\{
    -\sum_{j=1}^3 \eta_j
    \|f_j-\id\|_{\ell^p(\tilde\mu)}
    \right\},
    \qquad
    \eta=(45,25,12).
\]
We use fractional likelihood power \(\rho=0.70\), \(p=2\), and no hard-thresholding heuristic.  Each grid point is averaged over three random seeds and two initializations.

The main reported quantities are test RMSE, the redundant-layer identity error \(\Delta_{\id}\), and the partial-composition error \(\Delta_{\mathrm{partial}}\).  Smaller values are better.  For this experiment,
\[
    \Delta_{\id}
    =
    \frac{1}{3}\sum_{j=1}^3
    \| \widehat f_j-\id\|_{L^2},
\]
and 
\[
    \Delta_{\mathrm{partial}}
    = \frac{1}{3} \Big(\|\widehat{f}_1-\id\|_{L^2} + \|\widehat{f}_2\circ\widehat{f}_1-\id\|_{L^2} + \|\widehat{f}_3\circ\widehat{f}_2\circ\widehat{f}_1-\id\|_{L^2} \Big),
\]
These two layerwise diagnostics are more directly tied to the goal of shrinking redundant layers toward identity than RMSE alone.

\begin{table}[t]
\centering
\caption{Effect of noise level for the one-dimensional synthetic function \texttt{edge\_transition}.
Each row averages over \(n_{\mathrm{train}}\in\{10,15,20,30,40,60\}\), three seeds, and two initializations.
Smaller values are better.}
\label{tab:edge-noise-summary}
\begin{tabular}{c|cc|cc|cc}
\hline
& \multicolumn{2}{c|}{RMSE}
& \multicolumn{2}{c|}{\(\Delta_{\id}\)}
& \multicolumn{2}{c}{\(\Delta_{\mathrm{partial}}\)} \\
\(\sigma\) & regular & augmented & regular & augmented & regular & augmented \\
\hline
0.00 & 0.083 & \textbf{0.056} & 0.606 & \textbf{0.111} & 0.607 & \textbf{0.120} \\
0.05 & 0.103 & \textbf{0.075} & 0.616 & \textbf{0.120} & 0.615 & \textbf{0.127} \\
0.10 & 0.117 & \textbf{0.087} & 0.585 & \textbf{0.113} & 0.599 & \textbf{0.118} \\
0.20 & 0.156 & \textbf{0.132} & 0.609 & \textbf{0.129} & 0.603 & \textbf{0.135} \\
0.40 & 0.231 & \textbf{0.211} & 0.617 & \textbf{0.136} & 0.614 & \textbf{0.148} \\
0.70 & \textbf{0.304} & 0.305 & 0.625 & \textbf{0.136} & 0.619 & \textbf{0.135} \\
1.00 & \textbf{0.369} & 0.377 & 0.635 & \textbf{0.138} & 0.629 & \textbf{0.144} \\
\hline
\end{tabular}
\end{table}

Table~\ref{tab:edge-noise-summary} shows that the augmented DGP improves prediction for low and moderate noise levels.  The RMSE advantage is largest when \(\sigma\le0.40\), where redundant hidden-layer deformation in the regular DGP can hurt generalization.  At very high noise levels, \(\sigma=0.70\) and \(\sigma=1.00\), the prediction problem becomes noise-dominated and the regular DGP has slightly smaller RMSE.  However, the layerwise conclusion is unchanged across the entire noise range: the augmented DGP has much smaller \(\Delta_{\id}\) and \(\Delta_{\mathrm{partial}}\) for every value of \(\sigma\).  Thus the identity augmentation remains effective as a representation-stabilizing device even when the predictive advantage disappears.

\begin{table}[t]
\centering
\caption{Effect of sample size for the one-dimensional synthetic function \texttt{edge\_transition}.
Each row averages over \(\sigma\in\{0,0.05,0.10,0.20,0.40,0.70,1.00\}\), three seeds, and two initializations.
Smaller values are better.}
\label{tab:edge-ntrain-summary}
\begin{tabular}{c|cc|cc|cc}
\hline
& \multicolumn{2}{c|}{RMSE}
& \multicolumn{2}{c|}{\(\Delta_{\id}\)}
& \multicolumn{2}{c}{\(\Delta_{\mathrm{partial}}\)} \\
\(n_{\mathrm{train}}\) & regular & augmented & regular & augmented & regular & augmented \\
\hline
10 & 0.373 & \textbf{0.286} & 0.640 & \textbf{0.191} & 0.633 & \textbf{0.188} \\
15 & 0.240 & \textbf{0.221} & 0.618 & \textbf{0.141} & 0.617 & \textbf{0.146} \\
20 & \textbf{0.154} & 0.160 & 0.595 & \textbf{0.109} & 0.610 & \textbf{0.116} \\
30 & 0.159 & \textbf{0.153} & 0.610 & \textbf{0.108} & 0.603 & \textbf{0.112} \\
40 & \textbf{0.125} & \textbf{0.125} & 0.599 & \textbf{0.103} & 0.603 & \textbf{0.118} \\
60 & \textbf{0.117} & 0.121 & 0.618 & \textbf{0.106} & 0.607 & \textbf{0.115} \\
\hline
\end{tabular}
\end{table}

Table~\ref{tab:edge-ntrain-summary} shows that the predictive benefit is strongest in the small-sample regime.  At \(n_{\mathrm{train}}=10\), the augmented DGP substantially lowers RMSE, from \(0.373\) to \(0.286\), while also reducing \(\Delta_{\id}\) from \(0.640\) to \(0.191\).  This is the regime in which the regular DGP is most vulnerable to spurious hidden-layer warping.  As the sample size increases, the RMSE difference becomes smaller and sometimes favors the regular DGP.  This is not surprising: with more data, the unpenalized model can use its extra flexibility to improve the final regression fit, even if its hidden-layer representation remains non-parsimonious.

The representation diagnostics are again systematic.  For every sample size, the augmented DGP has much smaller identity error and partial-composition error.  The regular DGP keeps its hidden layers highly active, whereas the augmented posterior selects a decomposition much closer to the shallow representative \(f_1=f_2=f_3=\id\).

Across all \(42\) noise-sample-size combinations, the augmented DGP reduces both \(\Delta_{\id}\) and \(\Delta_{\mathrm{partial}}\).  RMSE improves in most, but not all, settings: the strongest prediction gains occur at low-to-moderate noise and small sample size, while at very high noise the predictive comparison becomes less stable.  These results support the main interpretation of the proposed augmentation.  Its most robust effect is not merely to improve prediction, but to resolve redundant deep decompositions by selecting a more canonical, identity-biased layer representation.  Prediction improves precisely in those regimes where the regular DGP's unnecessary hidden-layer flexibility harms generalization.

\end{document}